\documentclass[a4paper, 11pt,reqno]{amsart}

\usepackage[T1]{fontenc}      % Uses 'T1' font encoding.
\usepackage[foot]{amsaddr}    % Manages options for displayed author information.
\usepackage{comment}
\usepackage{siunitx}
\usepackage[british]{babel}   % Support for different languages.
\usepackage{csquotes}         % Adds quotation marks according to language.
\usepackage[babel]{microtype} % Makes things look nicer (spacing and so on).

\usepackage{mathtools}        % Extension of 'amsmath' which fixes some bugs.
\usepackage{amssymb}          % Extra symbols and fonts.
\usepackage{amsthm}           % Helps to define "theorem-like" environments.
\usepackage{amstext}          % Defines a '\text' macro for math environments.
\usepackage{mleftright}       % Fixes spacing issue.
\usepackage{gensymb}          % Allows to use degrees.
\usepackage[makeroom]{cancel} % Defines cancellation of terms in formulae.
\usepackage{mathdots}         % Adds some definitions of dots.
\usepackage{mathrsfs}         % Adds rsfs fonts.
\usepackage{dsfont}           % Allows some fancy numbers (\mathds).
\usepackage{bm}               % For boldface in maths.
\usepackage{cancel}
\usepackage{stickstootext}
\usepackage[stickstoo,varbb]{newtxmath} % These two lines are an alternative for stix2 which do essentially the same but fix several issues with some symbols...
\makeatletter
\DeclareFontFamily{U}{ntxmia}{\skewchar \font =127}
 \DeclareFontShape{U}{ntxmia}{m}{it}{
                        <-> \ntxmath@scaled ntxmia
                      }{}    
                      \DeclareFontShape{U}{ntxmia}{b}{it}{
                        <-> \ntxmath@scaled ntxbmia
                      }{}
\makeatother

\usepackage{graphicx}         % Provides commands to include graphics.
\usepackage{psfrag}           % Gives some extra useful things for graphics.
\usepackage{caption}          % Options for formatting float captions.
\usepackage{subcaption}
\usepackage{tikz}             % Language to create graphics programmatically.
\usetikzlibrary{backgrounds}

\usepackage{booktabs}         % Enhances the quality of tables.

\usepackage{enumitem}         % Allows customization for enumerations.

\usepackage[margin=1in]{geometry} % Provides commands to define the page layout.

\usepackage[textsize=footnotesize]{todonotes}
\usepackage[numbers,sort]{natbib}

\makeatletter
\def\NAT@spacechar{~}% NEW
\makeatother

\usetikzlibrary{calc}
\def\centerarc[#1](#2)(#3:#4:#5)% Syntax: [draw options] (center) (initial angle:final angle:radius)
    { \draw[#1] ($(#2)+({#5*cos(#3)},{#5*sin(#3)})$) arc (#3:#4:#5); }

\usepackage{hyperref}              % For hyperlinks within the document
\usepackage[capitalise]{cleveref}  % For improved referencing options; use [capitalise] if desired
\hypersetup{colorlinks=true,
   citecolor=blue,
   filecolor=blue,
   linkcolor=blue,
   urlcolor=blue
}
\usepackage{url}

\makeatletter
\if@cref@capitalise
\crefname{figure}{figure}{figures}
\crefname{claim}{Claim}{Claims}
\crefname{conjecture}{Conjecture}{Conjectures}
\else
\crefname{figure}{Figure}{Figures}
\crefname{claim}{claim}{claims}
\crefname{conjecture}{conjecture}{conjectures}
\fi
\makeatother
\Crefname{figure}{Figure}{Figures}
\Crefname{claim}{Claim}{Claims}
\crefname{conjecture}{Conjecture}{Conjectures}

\usepackage{algorithm}
\usepackage{algpseudocode}

\allowdisplaybreaks

\newtheorem{definition}{Definition}[section]
\newtheorem{claim}{Claim}

\newtheorem{proposition}[definition]{Proposition}
\newtheorem{theorem}[definition]{Theorem}
\newtheorem{corollary}[definition]{Corollary}
\newtheorem{lemma}[definition]{Lemma}
\newtheorem{fact}[definition]{Fact}
\newtheorem{conjecture}[definition]{Conjecture}

\newtheorem{problem}[definition]{Problem}

\theoremstyle{definition}
\newtheorem{remark}[definition]{Remark}
\newtheorem*{ConsW}{Construction~of~$\boldsymbol{W}$}

\newenvironment{claimproof}{%
\let\origqed=\qedsymbol%
\renewcommand{\qedsymbol}{$\blacktriangleleft$}%
\begin{proof}}{\end{proof}\let\qedsymbol=\origqed}

\numberwithin{equation}{section}

\renewcommand{\binom}[2]{\ensuremath{\mleft(\kern-.1em\genfrac{}{}{0pt}{}{#1}{#2}\kern-.1em\mright)}}    % This makes binomial numbers nicer with stix2 (in displayed equations). Remove if stix2 is not loaded.
\newcommand{\inbinom}[2]{\ensuremath{\bigl(\kern-.1em\genfrac{}{}{0pt}{}{#1}{#2}\kern-.1em\bigr)}} % This is better for inline equations, as it will keep sizes of parentheses consistent and not create extra vertical space.

\newcommand*\nume{\ensuremath{\mathrm{e}}}

\newcommand{\cA}{\mathcal{A}}
\newcommand{\cB}{\mathcal{B}}
\newcommand{\cC}{\mathcal{C}}

\newcommand{\cE}{\mathcal{E}}
\newcommand{\cF}{\mathcal{F}}
\newcommand{\cG}{\mathcal{G}}

\newcommand{\cI}{\mathcal{I}}
\newcommand{\cJ}{\mathcal{J}}

\newcommand{\cP}{\mathcal{P}}
\newcommand{\cQ}{\mathcal{Q}}

\newcommand{\cU}{\mathcal{U}}

\makeatletter
\def\moverlay{\mathpalette\mov@rlay}
\def\mov@rlay#1#2{\leavevmode\vtop{%
  \baselineskip\z@skip \lineskiplimit-\maxdimen
  \ialign{\hfil$\m@th#1##$\hfil\cr#2\crcr}}}
\newcommand{\charfusion}[3][\mathord]{
    #1{\ifx#1\mathop\vphantom{#2}\fi
        \mathpalette\mov@rlay{#2\cr#3}
      }
    \ifx#1\mathop\expandafter\displaylimits\fi}
\makeatother
\newcommand{\cupdot}{\charfusion[\mathbin]{\cup}{\cdot}}
\newcommand{\eps}{\epsilon}

\newcommand{\aedc}[1]{{\color{red}{ \bf [~Alberto:\ }\emph{#1}\textbf{~]}}}

\newcommand{\COMMENT}[1]{}
\newcommand{\COMNEW}[1]{}
\title[Local resilience of random geometric graphs]{On the local resilience of random geometric graphs\linebreak{} with respect to connectivity and long cycles}

\author[A.~Espuny D\'iaz]{Alberto Espuny D\'iaz}
\email{espuny-diaz@informatik.uni-heidelberg.de}
\address[Espuny D\'iaz]{Institut f\"ur Informatik, Universit\"at Heidelberg, 69120 Heidelberg, Germany.}
\author[L.~Lichev]{Lyuben Lichev}
\email{lyuben.lichev@ist.ac.at}
\address[Lichev]{Institute of Science and Technology Austria (ISTA), 3400 Klosterneuburg, Austria.}
\author[A.~Wesolek]{Alexandra Wesolek}
\email{alexandrawesolek@gmail.com}
\address[Wesolek]{Institut für Mathematik, Technische Universität Berlin, 10623 Berlin, Germany.}

\thanks{A.~Espuny Díaz was supported by the Deutsche Forschungsgemeinschaft (DFG, German Research Foundation) through projects no.\ 447645533 and 513704762. L.~Lichev was supported by the Austrian Science Fund (FWF) grant no.\ 10.55776/ESP624. A.~Wesolek was supported by the DFG under Germany’s Excellence Strategy – The Berlin Mathematics Research Center MATH+ (EXC-2046/1, project ID: 390685689).}

\date{\today}

\begin{document}
\begin{abstract}
Given an increasing graph property $\mathcal{P}$, a graph $G$ is $\alpha$-resilient with respect to $\mathcal{P}$ if, for every spanning subgraph $H\subseteq G$ where each vertex keeps more than a $(1-\alpha)$-proportion of its neighbours, $H$ has property $\mathcal{P}$.
We study the above notion of local resilience with $G$ being a random geometric graph $G_d(n,r)$ obtained by embedding $n$ vertices independently and uniformly at random in $[0,1]^d$, and connecting two vertices by an edge if the distance between them is at most $r$. 

First, we focus on connectivity. 
We show that, for every $\eps>0$, for $r$ a constant factor above the sharp threshold for connectivity $r_c$ of $G_d(n,r)$, the random geometric graph is $(1/2-\eps)$-resilient for the property of being $k$-connected, with $k$ of the same order as the expected degree. 
However, contrary to binomial random graphs, for sufficiently small $\eps>0$, connectivity is not born $(1/2-\eps)$-resilient in $2$-dimensional random geometric graphs.

Second, we study local resilience with respect to the property of containing long cycles. 
We show that, for $r$ a constant factor above $r_c$, $G_d(n,r)$ is $(1/2-\eps)$-resilient with respect to containing cycles of all lengths between constant and $2n/3$. 
Proving $(1/2-\eps)$-resilience for Hamiltonicity remains elusive with our techniques. 
Nevertheless, we show that $G_d(n,r)$ is $\alpha$-resilient with respect to Hamiltonicity for a fixed constant $\alpha = \alpha(d)<1/2$.
\end{abstract}

% \vspace{1em}
% Keywords: .

\maketitle

\section{Introduction}\label{sec:intro}

For a graph $G$ that satisfies a certain property $\cP$, how much do we have to modify $G$ so that it stops satisfying $\cP$?
This general question motivated the study of graph resilience with respect to different properties.
A special attention was paid to the following localised version of the above question.
Given a graph $G$ and some $\alpha\in[0,1]$, we say that a subgraph $H\subseteq G$ is an \emph{\(\alpha\)-subgraph} of~$G$ if for every $v\in V(G)$ we have $d_H(v)\geq\alpha d_G(v)$.
If $G$ satisfies an increasing property $\mathcal{P}$, we define the \emph{local resilience} of~$G$ with respect to~$\mathcal{P}$ as the supremum of the values $\alpha\in[0,1]$ such that \emph{every} $(1-\alpha)$-subgraph of~$G$ satisfies~$\mathcal{P}$.
Alternatively, one may think that an adversary is allowed to remove edges from $G$ subject to the condition that the proportion of edges incident to any given vertex cannot be decreased by more than a factor of $\alpha$.
Then, the local resilience of $G$ with respect to $\mathcal{P}$ is the minimum $\alpha$ such that the adversary can produce a $(1-\alpha)$-subgraph of $G$ without property~$\mathcal{P}$.

Many natural statements can be expressed in the language of local resilience.
For example, recall the classical theorem of Dirac~\cite{Dirac52} which asserts that every graph $G$ on $n\geq 3$ vertices with minimum degree $\delta(G)\geq n/2$ contains a Hamilton cycle, and that this is the best possible minimum degree condition.
This statement can be rephrased in the language of local resilience of the complete graph. 
Indeed, fix $\alpha^*\coloneqq(\lceil n/2\rceil-1)/(n-1)$.
Then, Dirac's theorem states that, for every $\alpha>\alpha^*$, every $\alpha$-subgraph of $K_n$ contains a Hamilton cycle and that this is best possible, implying that the local resilience of $K_n$ with respect to Hamiltonicity is $1-\alpha^*\sim1/2$.
Likewise, all classical extremal results with minimum degree conditions guaranteeing some property $\mathcal{P}$ can be seen as statements about the local resilience of $K_n$ with respect to $\mathcal{P}$.

A lot of research has been done on the local resilience of random graphs with respect to different properties.
The binomial random graph $G(n,p)$, sampled by including each of the possible $\inbinom{n}{2}$ edges on vertex set $[n]\coloneqq\{1,\ldots,n\}$ independently with probability~$p$, received particular attention.
(We discuss results in this context in the coming sections.)
In this paper, we focus on the local resilience of \emph{random geometric graphs}, which were first introduced by \citet{Gil61}.
Given positive integers~$d$ and~$n$ and a real number $r > 0$, the $d$-dimensional random geometric graph $G_d(n,r)$ is a graph on vertex set $[n]$ whose edges are generated as follows.
Consider $n$ mutually independent uniform random variables $X_1,\ldots,X_n$ on $[0,1]^d$.
For each pair of distinct elements $i,j\in [n]$, the edge $ij$ is added to the graph if and only if $\lVert X_i-X_j\rVert\leq r$, where $\lVert\cdot\rVert$ denotes the Euclidean distance.

\subsection{Local resilience and Hamiltonicity}\label{sec:rel work}

The systematic study of local resilience was introduced by \citet{SV08}.\footnote{We remark that our definition of local resilience is slightly different from that in~\cite{SV08} and closer to more recent versions which have appeared in the literature~\cite{Mont19,NST19}. Nevertheless, all the results we mention in our paper hold with our definition as well.}
In their pioneering paper, among other results, they considered the local resilience of random and pseudorandom graphs with respect to Hamiltonicity.
In particular, they showed that, for $p\geq (\log n)^4/n$ and any fixed constant $\eps>0$, asymptotically almost surely (that is, with probability tending to $1$ as $n$ grows to infinity, which we abbreviate as a.a.s.) the local resilience of~$G(n,p)$ with respect to Hamiltonicity is at least $1/2-\eps$.
Since then, the property of containing a Hamilton cycle has become a benchmark in the study of local resilience for random graphs.

It is worth noting that the constant $1/2$ in the result of \citet{SV08} is best possible in great generality: this is a consequence of a result of \citet{Sti96}, who showed that the vertices of every graph can be partitioned into two sets in such a way that every vertex $v$ has at least $d(v)/2-1$ neighbours in the part to which it belongs.
In particular, for any $n$-vertex graph $G$ whose minimum degree tends to infinity with $n$, we may consider such a partition and delete all edges between the two parts to obtain a $(1/2-o(1))$-subgraph which is not connected.
This implies that the local resilience of~$G$ with respect to any property which necessitates connectivity is at most $1/2+o(1)$.

On the other hand, the result of \citet{SV08} is not optimal with respect to the lower bound on~$p$.
The problem of determining the local resilience of $G(n,p)$ with respect to Hamiltonicity for the full range of $p$ was studied further~\cite{BKS11b,BKS11a,FK08} until it was finally resolved by \citet{LS12}, who proved the following theorem.

\begin{theorem}[{\cite[Theorem~1.1]{LS12}}]\label{thm:LS12}
    For every\/ $\eps\in (0,1/2]$, there exists a constant\/ $C=C(\eps)>0$ such that, for\/ $p\geq C\log n/n$, a.a.s.\ every\/ $(1/2+\eps)$-subgraph of\/ $G(n,p)$ is Hamiltonian.
\end{theorem}

This result can be regarded as a version of Dirac's theorem for binomial random graphs.
Recently, \citet{Mont19} and \citet{NST19} independently sharpened this result to a ``hitting time'' result. 
Moreover, \cref{thm:LS12} was recently generalised by \citet{CEKKO19}, who considered a version of the problem where the adversary is allowed to delete more edges incident to some, but not all, vertices of the graph, giving an analogue of a classical result of \citet{Posa62} for random graphs.

The local resilience with respect to Hamiltonicity has also been studied in other models of random graphs.
One example is the model of random regular graphs.
Given integers $1\leq d<n$, a random $d$-regular graph $G_{n,d}$ is obtained by sampling an element from the set of all $n$-vertex $d$-regular graphs uniformly at random.
It is well known that the random $d$-regular graph is a.a.s.\ Hamiltonian for all $d\geq3$~\cite{RW92,RW94,CFR02,KSVW01}.
Improving on earlier work of \citet{BKS11b}, \citet{CEGKO21} proved a result analogous to \cref{thm:LS12} for random regular graphs.

\begin{theorem}[{\cite[Theorem~1.2]{CEGKO21}}]\label{thm:CEGKO21}
    For every\/ $\eps\in (0,1/2]$, there exists a constant\/ $D=D(\eps)>0$ such that, for all\/ $d\in [D,n-1]$, a.a.s.\ every\/ $(1/2+\eps)$-subgraph of\/ $G_{n,d}$ is Hamiltonian.
\end{theorem}

Just like for binomial random graphs, the constant $1/2$ is best possible.
Moreover, \citet{CEGKO21} also showed that a polynomial dependency between~$D$ and~$\eps$ in \cref{thm:CEGKO21} cannot be avoided.

In this paper, we are concerned with studying the local resilience of random geometric graphs with respect to Hamiltonicity.
This question has been brought up by \citet[Problem~43]{FriezeBiblio}.
By analogy with \cref{thm:LS12,thm:CEGKO21}, the following conjecture seems appealing.

\begin{conjecture}\label{mainconjecture}
    For every\/ $\eps\in (0,1/2]$ and integer\/ $d\geq1$, there exists a constant\/ $C=C(d,\eps)>0$ such that, for\/ $r\geq C(\log n/n)^{1/d}$, a.a.s.\ every\/ $(1/2+\eps)$-subgraph of\/ $G_d(n,r)$ is Hamiltonian.
\end{conjecture}

It is worth noting that, if true, this conjecture is best possible up to the value of $C$.
Indeed, the results about the threshold for Hamiltonicity \cite{Petit01,DMP07,BBKMW11,MPW11} in random geometric graphs show that the statement fails for $r=o((\log n/n)^{1/d})$, and the constant $1/2$ cannot be improved.
We defer some further discussion about this to \cref{sect:introconnectivity}.

While the local resilience of random geometric graphs has not been studied before, it is worth noting that \citet{Petit01} studied Hamiltonicity of subgraphs of a random geometric graph under random deletion of its edges.
However, this work does not provide any meaningful insights towards \cref{mainconjecture}.

We remark here two main challenges that must be addressed when considering \cref{mainconjecture}.
First, for any increasing property $\mathcal{P}$, when aiming to prove that a.a.s.\ $G_d(n,r)\in\mathcal{P}$ for all $r\geq r^*$, it suffices to assume that $r=r^*$. 
This is due to the monotonicity of the problem.
One difficulty in the study of local resilience is the absence of monotonicity: indeed, graphs with larger vertex degrees may still be less resilient with respect to some increasing property as the adversary has the right to delete more edges from them.
Therefore, when considering this problem, we must really consider all possible values of $r$.
Second, all known results about Hamiltonicity in $G_d(n,r)$ exploit the fact that these graphs are locally dense, meaning that vertices which are (geometrically) close form very large cliques~\cite{BBP-GP17,Es23,EH23,FP-G20,Petit01,DMP07,BBKMW11,Man23, MPW11,FMMS21}.
This is not the case in the local resilience setting, where the adversary may delete all cliques.
To be precise, for any fixed $\eps<1/6$ and $r$ in the range considered in \cref{mainconjecture}, a.a.s.\ there are $(1/2+\eps)$-subgraphs of $G_d(n,r)$ which are $K_4$-free.
Indeed, for sufficiently large values of $r$, such subgraphs can be constructed by splitting the vertices randomly into three sets and removing all edges contained in each of the sets. 
Therefore, in order to solve \cref{mainconjecture}, new techniques seem necessary.

We are currently unable to resolve \cref{mainconjecture}, and even the analogous conjecture for perfect matchings remains interesting and widely open.
However, in this paper, we make progress towards \cref{mainconjecture} and several related problems.
We present our main results in the following sections.

\subsection{Connectivity}\label{sect:introconnectivity}

We begin with an analysis of the local resilience of random geometric graphs with respect to connectivity, which is a necessary condition for Hamiltonicity.
We show that this necessary condition holds in our context by providing the following analogue of \cref{mainconjecture} for connectivity.

\begin{theorem}\label{thm:connectivityintro}
For every\/ $\eps\in (0,1/2]$ and integer\/ $d\geq1$, there exists a constant\/ $C_1=C_1(d,\eps)>0$ such that, for every\/ $r\geq C_1(\log n/n)^{1/d}$, a.a.s.\ every\/ $(1/2+\eps)$-subgraph of\/ $G_d(n,r)$ is connected. 
\end{theorem}

\Cref{thm:connectivityintro} provides a local resilience analogue of classical results of \citet{GJ96} and \citet{Penrose97,Penrose99} on the threshold for connectivity of random geometric graphs.\COMMENT{\citet{GJ96} obtained the threshold for $d=1$; later, \citet{Penrose97} obtained the threshold for $d=2$, and subsequently extended it (with a hitting time result) to all other dimensions~\cite{Penrose99}.}
Moreover, we can extend this result to $k$-connectivity for a large range of values of $k$. 
We defer the more technical statement and the proof to \cref{sect:conn} (see \cref{thm:connectivity}).
Together with the aforementioned result of \citet{Sti96}, \cref{thm:connectivityintro} implies that a.a.s.\ the local resilience of random geometric graphs $G_d(n,r)$ with respect to connectivity is $1/2\pm o(1)$ whenever $r=\omega((\log n/n)^{1/d})$.
If we restrict our attention to $(1/2-\eps)$-subgraphs of $G_d(n,r)$, we can say more about the size of a largest component.

\begin{proposition}\label{prop:conn_optimal}
For every\/ $\eps\in (0,1/2]$, there exists a constant\/ $C_2 = C_2(\eps) > 0$ such that, for every integer\/ $d\ge 1$ and\/ $r\ge C_2 (\log n/n)^{1/d}$, a.a.s.\ there exists a\/ $(1/2-\eps)$-subgraph of\/ $G_d(n,r)$ with components of order at most\/ $5rn$.
\end{proposition}

We believe that Proposition~\ref{prop:conn_optimal} holds for all smaller radial values as well.
However, our simple proof requires very good concentration on the number of vertices of $G_d(n,r)$ in certain regions of area $\Theta(r^d)$.
This only holds when $r\ge C(\log n/n)^{1/d}$ for a suitable constant $C > 0$.
A question related to \cref{prop:conn_optimal} is whether a.a.s.\ there exists a $(1/2-\eps)$-subgraph of $G_d(n,r)$ all whose components have order $o(rn)$.
This is clearly not possible when $d=1$, in which case \cref{prop:conn_optimal} is best possible.
Understanding this problem in higher dimensions seems appealing.

Going back to \cref{thm:connectivityintro}, we remark that the expression $C_1 (\log n/n)^{1/d}$ must be larger than the (sharp) connectivity threshold of $G_d(n,r)$.
This implies that, for fixed $\eps$, the constant $C_1(d,\eps)$ needs to grow at speed $\Omega(\sqrt{d})$ as a function of $d$ (see, e.g.,~\cite[Theorem~13.2]{Pen03}).
In fact, our proof shows that it suffices to have a constant $C_1$ which grows with speed $\Theta(\sqrt{d})$ (see Theorem~\ref{thm:connectivity} for more details).
The behaviour of $C_1$ with respect to $\eps$ is more mysterious. 
While we are uncertain if $C_1$ can be chosen universally for all $\eps \in (0, 1/2]$ as a function of $d$ when $d\ge 3$, we can prove that this is not the case when $d\in \{1, 2\}$.

\begin{theorem}\label{thm:C}
The following statements hold.
\begin{enumerate}[label=$(\mathrm{\roman*})$]
    \item\label{thm:Citem1} For every\/ $\eps\in(0,1/2]$, if\/ $r \leq \log n/4\eps n$, then a.a.s.\ there exists a disconnected\/ $(1/2+\eps)$-subgraph of\/ $G_1(n,r)$.
    \item\label{thm:Citem2} For every\/ $C > 0$, there exists\/ $\eps = \eps(C)\in (0,1/2]$ such that, for every\/ $r \le C (\log n/n)^{1/2}$, a.a.s.\ there exists a disconnected\/ $(1/2+\eps)$-subgraph of\/ $G_2(n,r)$.
\end{enumerate}
\end{theorem}

Let us discuss an important consequence of this result.
Rather than an instance of the random graph, let us consider a process which generates random graphs by adding edges one by one.
The corresponding process for binomial random graphs is the \emph{Erd\H{o}s-R\'enyi random graph process} $(G_i)_{i=0}^{n(n-1)/2}$, where $G_0$ is the graph with vertex set $[n]$ and no edges and, for each $i\in [\inbinom{n}{2}]$, the graph $G_i$ is obtained from $G_{i-1}$ by adding an edge chosen uniformly at random among those not contained in $G_{i-1}$. 
For a monotone increasing graph property $\mathcal{P}$, as the number of edges increases throughout this process, there must be a least index $i$ such that $G_i\in\mathcal{P}$.
This index, which is a random variable, is called the \emph{hitting time} for $\mathcal{P}$ in the random graph process.
A classical result of \citet{BT85} asserts that a.a.s.\ the hitting time for connectivity in this random graph process coincides with the hitting time for having minimum degree at least $1$.
\citet{HT19} showed that a local resilience version of this result holds: a.a.s.\ every $G_i$ which has minimum degree at least $1$ in the random graph process satisfies that every $(1/2+\eps)$-subgraph of $G_i$ is connected.

A similar process to the random graph process can be defined for random geometric graphs.
Recall that here one considers $n$ independent uniform random variables on $[0,1]^d$.
Let these variables be $X_1,\ldots,X_n$.
For each $e=ij\in\inbinom{[n]}{2}$, consider the random variable $\lVert e\rVert\coloneqq\lVert X_i-X_j\rVert$.
These define a total order on $\inbinom{[n]}{2}$ (note that all distances are distinct with probability $1$).
We label all pairs as $e_1,\ldots,e_{\binom{n}{2}}$ so that $\lVert e_i\rVert<\lVert e_j\rVert$ whenever $1\leq i<j\leq\inbinom{n}{2}$.
Then, for each $i\in\{0\}\cup[\inbinom{n}{2}]$, we let $G_i\coloneqq([n],\{e_j:j\in[i]\})$.
This again defines a random sequence of nested graphs, to which we refer here as the \emph{$d$-dimensional random geometric graph process}.
A classical result of \citet{Penrose97,Penrose99}\COMMENT{The first paper proves the result only for $d=2$. The second paper proves the hitting time for all dimensions, and also for all values of $k$ in $k$-connectivity.} shows that a phenomenon similar to that of the Erd\H{o}s-R\'enyi random graph process occurs when $d\geq2$: in the random geometric graph process, a.a.s.\ every $G_i$ which has minimum degree at least~$1$ is connected (see also \cite[Theorem~13.17]{Pen03}).
The same does not hold when $d=1$: in this case, there is a positive probability that, as edges are added one by one throughout the process, all isolated vertices disappear while leaving a disconnected graph.
In fact, the sharp threshold for connectivity and the sharp threshold for having minimum degree at least~$1$ in $G_1(n,r)$ differ by a factor of~$2$~\cite{GJ96}.

\cref{thm:C}~\ref{thm:Citem1} provides a range of $r$ above the sharp threshold for connectivity where $G_1(n,r)$ is not $(1/2-\eps)$-resilient with respect to connectivity.
This shows a contrasting behaviour with respect to binomial random graphs, and more in line with the result for random regular graphs, proving that a polynomial dependency between $C$ and $\eps$ cannot be avoided.
More interesting is the consequence of \cref{thm:C}~\ref{thm:Citem2}: it implies that, in contrast with the case of the Erd\H{o}s-R\'enyi random graph process, a local resilience version of the hitting time result of \citet{Penrose97,Penrose99} also fails when $d=2$.
It would be interesting to determine whether this is also the case for larger $d$ or, on the contrary, a hitting time version holds for sufficiently large $d$.

\begin{problem}\label{prob:hitting}
    Determine whether there exists a positive integer $d_0$ such that the following holds:
    for every $d\geq d_0$ and every $\eps\in (0,1/2]$, a.a.s.\ every graph in the $d$-dimensional random geometric graph process with minimum degree at least one satisfies that each of its $(1/2+\eps)$-subgraphs is connected.
\end{problem}

It is worth bringing up the fact that \citet{Mont19} and \citet{NST19} obtained a local resilience version of the hitting time result for Hamiltonicity (\citet{NST19} also proved an analogous result for perfect matchings).
\cref{thm:C} shows that such a result is impossible for random geometric graphs when $d\in\{1,2\}$.
If the answer to \cref{prob:hitting} is affirmative, it would be interesting to determine whether a local resilience hitting time result for Hamiltonicity is possible in the random geometric graph process for sufficiently large $d$.

\subsection{Long cycles}

Next, we turn our attention to the existence of long cycles in $G_d(n,r)$. 
In view of \cref{prop:conn_optimal}, it would be meaningful to show that every $(1/2+\eps)$-subgraph of $G_d(n,r)$ contains cycles of linear length.
It turns out that not only linearly long cycles exist, but for any fixed vertex $v$ in $G_d(n,r)$, at least one such cycle goes through $v$.

\begin{theorem}\label{thm:long cycles intro}
For every $\eps\in (0,1/2]$ and integer $d\geq 1$, there exist positive constants $c_3 = c_3(d, \eps)$ and $C_3 = C_3(d, \eps)$ such that, if $r\in [C_3 (\log n/n)^{1/d}, c_3]$, then a.a.s.\ $G_d(n,r)$ satisfies the following property:
for every integer $L\in [c_3^{-1}, 2n/3]$, every
$(1/2+\eps)$-subgraph $H \subseteq G_d(n,r)$ and every vertex $v\in V(G_d(n,r))$, there is a cycle of length $L$ in $H$ containing $v$.
\end{theorem}

We note that, under the assumptions of Theorem~\ref{thm:long cycles intro}, a.a.s.\ all vertices are contained in cycles of any constant even length $L\ge 4$; for more details, see Remark~\ref{rem:long cycles}.

If Conjecture~\ref{mainconjecture} holds, in the light of Theorem~\ref{thm:long cycles intro}, it is tempting to believe that well-connected random geometric graphs are $(1/2\pm o(1))$-resilient with respect to the property of containing a cycle of any sufficiently large length.

\begin{conjecture}\label{conj:pancyclicity}
For a suitable choice of $c_3$ and $C_3$, Theorem~\ref{thm:long cycles intro} holds for all $L\in [c_3^{-1}, n]$.
\end{conjecture}

In fact, it is also conceivable that a version of \cref{mainconjecture} for pancyclicity (that is, the property of containing cycles of all lengths between 3 and $n$) holds as well.

\subsection{Lower bounds on the local resilience with respect to Hamiltonicity}

As a way of making progress towards \cref{mainconjecture}, we also prove lower bounds on the local resilience of random geometric graphs with respect to Hamiltonicity. 
For each $d\ge 1$, let $\theta_d$ denote the volume of a $d$-dimensional ball with radius 1.

\begin{theorem}\label{thm:lowerboundgeneral}
For every constant\/ $\eps\in (0,1/(2d^{d/2}\theta_d)]$ and integer\/ $d\geq1$, there exists a constant\/ $C_4=C_4(d,\eps)>0$ such that, for\/ $C_4(\log n/n)^{1/d}\leq r=o(1)$, a.a.s.\ every\/ $(1-1/(2d^{d/2}\theta_d)+\eps)$-subgraph of\/ $G_d(n,r)$ is Hamiltonian.
\end{theorem}

This shows that a.a.s.\ every $(3/4+\eps)$-subgraph of $G_1(n,r)$ is Hamiltonian, while for $G_2(n,r)$ we can only prove the conclusion for $(c+\eps)$-subgraphs, where $c\approx 0.92$. 
The bound on the local resilience that we obtain becomes smaller as $d$ grows.
This approach can be improved by a more careful (albeit technical) analysis.
We provide a stronger result for the case $d=1$, where the analysis is simple and already showcases all the ideas.

\begin{theorem}\label{thm:lowerboundd=1}
For every constant\/ $\eps\in (0,1/3]$, there exists a constant\/ $C_5=C_5(\eps)>0$ such that, for\/ $C_5\log n/n\leq r=o(1)$, a.a.s.\ every\/ $(2/3+\eps)$-subgraph of\/ $G_1(n,r)$ is Hamiltonian.
\end{theorem}

We believe the constant $2/3$ cannot be improved with our current approach.\COMMENT{The constants that we get without completely optimising the approach are the following. $d=2$: 0.89; $d=3$: 0.968155; $d=4$: 0.991178. So the improvement becomes essentially negligible very fast.}
The analysis here can yield improvements for all $d\geq1$ in \cref{thm:lowerboundgeneral}, but the resulting constant will still tend to~$0$ as~$d$ grows (see \cref{remark:lowerboundPosa}).
Of course, we do not believe these results to be tight and would be interested in seeing improvements upon them.

\subsection{The one-dimensional torus and powers of cycles}

A well-studied variant of the random geometric graph is the \emph{toroidal random geometric graph}. 
For positive integers $n$ and $d$ and a positive real number $r$, the toroidal random geometric graph $T_d(n,r)$ is defined analogously to $G_d(n,r)$, the only difference being that the $n$ random points are sampled from $\mathbb{R}^d/\mathbb{Z}^d$.
This model removes the effect of the boundary and, therefore, it sometimes results in a more straightforward analysis.
We believe that \cref{mainconjecture} should also hold if we replace $G_d(n,r)$ by $T_d(n,r)$.
Note, moreover, that \cref{thm:connectivityintro,thm:C,thm:long cycles intro,thm:lowerboundgeneral,thm:lowerboundd=1} as well as \cref{prop:conn_optimal} also hold in the toroidal setting.

Even when $d=1$, \cref{mainconjecture} and its analogue for the torus remain open.
A possibility to tackle this problem is to reduce it to a deterministic, structural problem using the following observation.
Given a graph~$G$, its $k$-th power, denoted $G^k$, is obtained by adding an edge between any two vertices of $G$ whose (graph) distance in $G$ is at most $k$.
We denote a cycle on $n$ vertices by $\cC_n$.

\begin{lemma}\label{lemma:sandwich}
    For every constant\/ $\eps\in (0,1]$, there exists a constant\/ $C_6=C_6(\eps)>0$ such that, for\/ $r\in [C_6\log n/n,1/2]$, a.a.s.\ $\cC_n^{(1-\eps)rn}\subseteq T_1(n,r)\subseteq \cC_n^{(1+\eps)rn}$.
\end{lemma}

In view of this, in order to prove the case $d=1$ of the toroidal version of \cref{mainconjecture}, it would suffice to study the local resilience of (high) powers of cycles with respect to Hamiltonicity.
We propose the following stronger conjecture, which may be of independent interest.

\begin{conjecture}\label{powerconjecture}
    For all integers\/ $n\geq3$ and\/ $k\in [1,n/2]$, every graph\/ $H\subseteq \cC_n^k$ with\/ $\delta(H)\geq k+1$ is Hamiltonian.
\end{conjecture}

Note that the case $k=1$ holds trivially and the cases $k\geq n/2-1$ follow as a consequence of Dirac's theorem.
Thus, if true, \cref{powerconjecture} would extend Dirac's theorem to powers of cycles.
Here we prove the first non-trivial case of the conjecture.

\begin{proposition}\label{prop:square}
    For every integer\/ $n\geq4$, every graph\/ $H\subseteq \cC_n^2$ with\/ $\delta(H)\geq3$ is Hamiltonian.
\end{proposition}

\begin{remark}
    The approach to local resilience of random geometric graphs via local resilience of powers of graphs is not exclusive to the torus.
    Indeed, one may simply observe that the conclusion of \cref{lemma:sandwich} can be replaced by $P_n^{(1-\eps)rn}\subseteq G_1(n,r)\subseteq P_n^{(1+\eps)rn}$ and, therefore, the case $d=1$ of \cref{mainconjecture} could also be derived from a structural statement.
    Providing natural analogues of Lemma~\ref{lemma:sandwich} and Conjecture~\ref{powerconjecture} for higher dimensions would definitely be of interest.
\end{remark}

\subsection{More related work}

In addition to the binomial random graph and the random regular graph, the local resilience of random graphs with respect to Hamiltonicity has also been considered in other settings.
For example, in the setting of binomial random directed graphs, \citet{Mont20} gave a best possible hitting time result, improving on earlier work~\cite{HSS16,FNNPS17}.
Hamiltonicity has also been considered in binomial random hypergraphs; see the works of \citet{CEP20} for Hamilton Berge cycles, of \citet{APP21} for tight Hamilton cycles, and of \citet{PT22} for loose Hamilton cycles in $3$-uniform hypergraphs.

In addition to Hamiltonicity, other graph properties have been studied in the context of local resilience.
This is the case, for instance, of pancyclicity~\cite{KLS10}, (almost) spanning trees~\cite{BCS11}, different (almost) $F$-factors~\cite{BLS12,CGSS14,NS20}, squares of (almost) Hamilton cycles~\cite{NS17,FSST22}, or general bounded-degree graphs~\cite{HLS12,ABET20,ABEST22}.
We believe studying these questions in random geometric graphs would be very interesting.

The study of local resilience of graph properties fits more generally within the framework of the study of robustness.
For more references on this topic, we refer the reader to the survey of \citet{Suda17}.

\subsection{Paper organisation}

In \cref{sec:notation}, we introduce the necessary notation for the paper, as well as an auxiliary probabilistic result.
We devote \cref{sect:conn} to our results on connectivity, and \cref{sec:cycles} to the results about long cycles.
In \cref{section:bounds}, we prove our lower bounds on the local resilience of random geometric graphs with respect to Hamiltonicity.
Lastly, in the short \cref{sec:torus}, we discuss the $1$-dimensional case on the torus.

%%%%%%%%%%%%%%%%%%%%%%%%%%%%%%%%%%%%%%
%%%%%%%%%%%%%%%%%%%%%%%%%%%%%%%%%%%%%%

\section{Notation and preliminary results}\label{sec:notation}

For any $n\in\mathbb{Z}$, we denote $[n]\coloneqq\{i\in\mathbb{N}:1\leq i\leq n\}$. 
We denote by $\log$ the logarithm with base $\nume$.
Given any real numbers $a,b,c\in\mathbb{R}$, we write $a=b\pm c$ to mean that $a\in[b-c,b+c]$.

We sometimes introduce parameters (say, $a$, $b$ and $c$) with the notation $0<a\lll b,c\leq 1$; parameters presented in this way are said to appear in a \emph{hierarchy}.
Parameters in a hierarchy are chosen from right to left.
Indeed, when a statement holds for $0<a\lll b,c\leq 1$, we mean that there exists some unspecified increasing function $f\colon (0,\infty)\times (0, \infty)\to (0, \infty)$ such that the statement holds for all $0<b,c\leq 1$ and for all $0<a\leq f(b,c)$.
This extends naturally to longer hierarchies or to hierarchies where some parameter depends on more than two other parameters.
Intuitively, when we use a hierarchy as above, the reader should simply understand that, for all $b,c\in(0,1]$, we choose~$a$ sufficiently small (in terms of $b$ and $c$) for the subsequent claims to hold.
When writing asymptotic statements, we will often implicitly assume that $n$ is sufficiently large for the claims to hold (and, in particular, much larger than all parameters defined in a hierarchy, which are always treated as constants).
For asymptotic comparisons, we use the standard `big-O' notation, where usually all asymptotic statements will be with respect to a parameter $n$ corresponding to the number of vertices of a graph. 
When using asymptotic statements with respect to other variables, the variables in question will appear as a subindex, such as $O_x$ and $o_t$.
For the sake of readability, we ignore rounding issues whenever they do not affect the validity of the claims.

The notation $|\cdot|$ is used to denote either the \emph{size} of a finite set or the \emph{$d$-dimensional volume} of a region of $\mathbb{R}^d$.
In particular, given a finite set of points $V\subseteq\mathbb{R}^d$ and a region $R\subseteq\mathbb{R}^d$, we write $|V\cap R|$ for the number of points of $V$ contained in $R$.
Given any point $p\in\mathbb{R}^d$ and any $0<r_1<r_2$, we will write $B(p,r_1)$ to denote the (closed) ball of radius~$r_1$ centred at $p$, and $B(p,r_1,r_2) \coloneqq B(p,r_2)\setminus B(p,r_1)$ denotes the annulus with radii $r_1$ and $r_2$.
Moreover, for a set $\Lambda\subseteq \mathbb R^d$, we denote by $\partial \Lambda$ the topological boundary of the set $\Lambda$.
We denote by $\theta_d$ the volume of a $d$-dimensional ball with radius $1$.
We will use the fact that, for all $d\geq1$,
\begin{equation}\label{equa:ballbound}
    \theta_d=|B(0,1)|=\frac{\uppi^{d/2}}{\mathrm{Gamma}(d/2+1)} \geq \max\left(\frac{\uppi^{d/2}}{2d^{d/2+1}}, \frac{1}{d^{d/2}}\right),
\end{equation}
where $\mathrm{Gamma}(\cdot)$ stands for the Gamma function.
Throughout the paper, vertices of $G_d(n,r)$ are often identified with their geometric positions in $[0,1]^d$ without explicit mention.

Given a graph $G$, we denote by $V(G)$ its vertex set and by $E(G)$ its edge set.
As is standard, if $\{u,v\}\in E(G)$, we abbreviate the notation to $uv=\{u,v\}$.
For any disjoint sets $U,W\subseteq V(G)$, we denote by $G[U]$ the subgraph of $G$ \emph{induced} by~$U$, and by $G[U,W]$ the bipartite subgraph of $G$ \emph{induced} by~$U$ and~$W$.
Given a set $U\subseteq V(G)$, we denote by $G-U\coloneqq G[V(G)\setminus U]$.
Given a vertex $v\in V(G)$, we define its \emph{neighbourhood} $N_G(v)\coloneqq\{u\in V(G):uv\in E(G)\}$ and its degree $d_G(v)\coloneqq|N_G(v)|$.
When talking about the neighbourhood of a vertex $v$ and considering two or more graphs, if we wish to emphasise that the graph in which a vertex $u$ is a neighbour of $v$ is $G$, we will say that $u$ is a \emph{$G$-neighbour} of $v$.
For any integer $k\geq1$, we say that $G$ is \emph{$k$-connected} if $|V(G)|>k$ and, for every $U\subseteq V(G)$ of size $|U|<k$, the graph $G-U$ is connected.

A \emph{path} is a graph $P$ whose vertex set admits a labelling $V(P)=\{u_1,\ldots,u_{\ell+1}\}$ such that $E(P)=\{u_iu_{i+1}:i\in[\ell]\}$.
(In particular, we allow for a path to consist of a single vertex.)
A \emph{cycle} is a graph~$C$ which can be obtained from a path $P$ with $\ell\geq 2$ by setting $V(C)\coloneqq V(P)$ and $E(C)\coloneqq E(P)\cup\{u_1u_{\ell+1}\}$.
The \emph{length} of a path or a cycle is its number of edges.
For a graph $G$ and two vertices $u,v\in V(G)$, we say that $G$ contains a \emph{$(u,v)$-path} if it contains a path starting at $u$ and ending at $v$.
Moreover, for two sets $U,V\subseteq V(G)$, we say that $G$ contains a \emph{$(U,V)$-path} if there exist $u\in U$ and $v\in V$ such that $G$ contains a $(u,v)$-path.
Extending the abbreviated notation for edges, a list of vertices should be understood as a path; for example, for vertices $u,v,x,y$, we write $P=uvxy$ to denote the path with vertex set $V(P)=\{u,v,x,y\}$ and edge set $E(P)=\{uv,vx,xy\}$.
This abbreviation allows to concatenate paths to construct longer paths (or cycles).
For the sake of clarity when using concatenations, if $P$ is a $(u,v)$-path, we write $uPv$ to emphasise the endpoints of~$P$.
As an example, if~$P$ is a $(u,v)$-path and $Q$ is a $(v,w)$-path with $V(P)\cap V(Q)=\{v\}$, we may write $uPvQw$ to denote the path obtained by concatenating $P$ and $Q$.
Similarly, if $Q$ is an $(x,y)$-path with $V(P)\cap V(Q)=\varnothing$ and~$vx$ is an edge, we write $uPvxQz$ to mean the concatenation of $P$ and $Q$ through the edge $vx$.
We denote iterated concatenations with the symbol $\bigtimes$.

Sometimes, rather than considering a random geometric graph, it is useful to consider a fixed set of points $V$ on $[0,1]^d$ and construct a geometric graph on this vertex set.
Given $r>0$, we denote by $G(V,r)$ the graph with vertex set $V$ where two vertices $x,y\in V$ are joined by an edge whenever $\lVert x-y\rVert\leq r$ (we omit the dimension from the notation $G(V,r)$ since it is implied by the set $V$).

Some of our proofs share certain ingredients, and so it is useful to provide some global definitions here.
In particular, it is often the case that we ``split'' the hypercube $[0,1]^d$ into smaller regions and use them to define large structures in $G_d(n,r)$.
To be more precise, we often tessellate the hypercube $[0,1]^d$ with smaller axis-parallel hypercubes of the same side length (which will be the inverse of an integer).
We refer to these smaller hypercubes as \emph{cells}, and denote the set of all cells by~$\cQ$.
Formally, we consider each cell as a closed hypercube, so different cells may intersect at their boundaries.
In practice, these intersections are negligible: when considering the vertices of $G_d(n,r)$, the probability that any of them lie on the boundary of one of the cells is $0$.
Given a tessellation of $[0,1]^d$ with cells, we will often use an auxiliary graph $\Gamma$ with vertex set $\cQ$ where two cells are joined by an edge whenever they share a common $(d-1)$-dimensional face.
(Note that $\Gamma$ is a $d$-dimensional grid.)

We will make use of the following well-known Chernoff bound (see, e.g., the book of \citet[Corollary 2.3]{JLR}).

\begin{lemma}[Chernoff's bound]\label{lem:chernoff}
Let $X$ be the sum of $n$ independent Bernoulli random variables and let $\mu \coloneqq \mathbb{E}[X]$.
Then, for all $\delta \in [0,1]$, we have that $\mathbb{P}[|X-\mu|\geq\delta\mu]\leq 2\nume^{-\delta^2\mu/3}$.
\end{lemma}

%%%%%%%%%%%%%%%%%%%%%%%%%%%%%%%%%%%%%%
%%%%%%%%%%%%%%%%%%%%%%%%%%%%%%%%%%%%%%

\section{Connectivity}\label{sect:conn}

\subsection{The local resilience of random geometric graphs with respect to connectivity}\label{section:conditions}

Our aim in this section is to prove \cref{thm:connectivityintro} and \cref{prop:conn_optimal}.
In fact, we prove the following stronger result, which clearly implies \cref{thm:connectivityintro}.

\begin{theorem}\label{thm:connectivity}
For every constant\/ $\epsilon\in (0, 1/2]$ and integer\/ $d\geq1$, there exist constants\/ $c_1 = c_1(d, \eps) > 0$ and\/ $C_1=C_1(d, \eps)>0$ such that, for every\/ $r\in [C_1(\log n/n)^{1/d}, \sqrt{d}]$, a.a.s.\ every\/ $(1/2+\epsilon)$-subgraph of\/ $G_d(n,r)$ is\/ $(c_1r^dn)$-connected. 
Moreover, one may take $C_1=K\eps^{-2/d}\sqrt{d}$, where $K$ is an absolute~constant.
\end{theorem}

We point out that \cref{thm:connectivity} is optimal simultaneously with respect to the constant $1/2$, the connectivity of the subgraphs of $G_d(n,r)$ (up to the value of $c_1$) and the growth of $C_1$ with respect to~$d$.
For the second point, note that the connectivity of a graph is always at most its minimum degree and, in the given range of~$r$, a.a.s.\ the minimum degree of $G_d(n,r)$ is proportional to $r^dn$.
Moreover, concerning the last point, for any fixed $\eps$, the expression $C_1(\log n/n)^{1/d}$ has to be at least as large as the sharp threshold for connectivity of the random geometric graph $G_d(n,r)$.
This shows that $C_1=\Omega_d(\sqrt{d})$ is necessary in \cref{thm:connectivity} (see~\cite[Theorem~13.2]{Pen03} for a concrete expression), which matches our bound of $C_1=O_d(\sqrt{d})$ (for fixed $\eps$).

The proof of Theorem~\ref{thm:connectivity} relies on the following ideas. 
First, divide the hypercube $[0,1]^d$ into cells of side length $s = \Omega(r)$ (and, thus, volume $\Omega(\log n/n)$).
We say that a cell is good if, roughly speaking, it contains at least $s^d n/2$ vertices of $G_d(n,r)$ and every two vertices sufficiently close to the cell share almost the same neighbourhood in $G_d(n,r)$; we say that the cell is bad otherwise. 
We show that a.a.s.\ there are only a few bad cells and these are grouped into clusters of bounded size, which are far from each other. 
Then, fixing a $(1/2+\eps)$-subgraph $H$ of $G_d(n,r)$, we use this observation to deduce that there is a large cluster of good cells containing a set of points that spans a $(c_1r^dn)$-connected graph $H'\subseteq H$.
This relies on the fact that all pairs of vertices in neighbouring good cells have many common neighbours in any $(1/2+\eps)$-subgraph~$H$ of~$G_d(n,r)$.
Finally, we show that the points outside $H'$ could be added consecutively to $H'$ so that every vertex has at least $c_1 r^d n$ neighbours among the vertices already in the graph.

Several technical details of this proof can be simplified significantly if one does not try to optimise the dependency of $C_1$ on $d$.
For the sake of the reader who wants to understand the basic ideas, we include this simpler proof in \cref{appen:connectivity}.

We now proceed to prove \cref{thm:connectivity}.
We begin with an auxiliary combinatorial lemma.

\begin{lemma}\label{lem:k-conn}
Fix an integer $k\ge 1$ and a graph $G$ with vertex set $V$.
Suppose that there is a set $V'\subseteq V$ such that the following two properties hold.
\begin{enumerate}[label=$(\mathrm{A}\arabic*)$]
    \item\label{lem:k-connitem1} For every set $U\subseteq V$ of size $k-1$, any pair of vertices in $V'\setminus U$ is connected by a path in $G-U$.
    \item\label{lem:k-connitem2} Set $t \coloneqq |V|-|V'|$. 
    Then, there is an ordering $v_1, \ldots, v_t$ of the vertices in $V\setminus V'$ such that, for every $i\in [t]$, $v_i$ has at least $k$ neighbours in $V'\cup \{v_1, \ldots, v_{i-1}\}$.
\end{enumerate}
Then, the graph $G$ is $k$-connected.
\end{lemma}

\begin{proof}
Fix an arbitrary set $U\subseteq V$ of size $k-1$.
On the one hand, the first assumption implies that all vertices in $V'\setminus U$ belong to the same connected component of $G-U$.
On the other hand, for every $i\in [t]$, if $v_i\notin U$, then $v_i$ is adjacent to a vertex among $(V'\cup \{v_1, \ldots, v_{i-1}\})\setminus U$.
Thus, an immediate induction shows that each of the vertices in $\{v_1, \ldots, v_t\}\setminus U$ connects by a path in $G-U$ to $V'\setminus U$, which concludes the proof.
\end{proof}

Given $d\geq1$, $\eps\in(0,1/2]$ and $r\le \sqrt{d}$, let $\cQ$ be a tessellation of $[0,1]^d$ with cells of side length $s = s(d,\eps,r) = 1/\lceil 4\sqrt{d}/\delta r\rceil$, where $\delta = \delta(d, \eps) > 0$ is the unique positive solution of the equation
\begin{equation}\label{eq:delta}
((1+\delta)^d-(1-\delta)^d) 4^{d+1}d^{d/2}=\eps.
\end{equation}
We say that two cells in $\cQ$ are \emph{neighbours} (or \emph{neighbouring}) if they are adjacent in the auxiliary graph~$\Gamma$, that is, if their intersection is a $(d-1)$-dimensional hypercube.
Note that, by the definition of~$s$, the union of any pair of neighbouring cells has diameter at most $\delta r$.\COMMENT{The square of the diameter is 
\[(d-1)s^2+4s^2=(d+3)s^2\leq(d+3)\frac{\delta^2r^2}{16d}\leq\frac14\delta^2r^2.\]
The bound on the diameter follows.}
Given a set of $n$ points $V\subseteq[0,1]^d$ and the geometric graph $G=G(V,r)$, we call a cell $q\in \cQ$ \emph{bad} (for $G$) if at least one of the following conditions holds:
\begin{enumerate}[label=$(\mathrm{\alph*})$]
    \item\label{bad_condition_1} $|V\cap q|\le s^d n/2$;
    \item\label{bad_condition_2} there exist a cell $q'\in N_{\Gamma}(q)\cup\{q\}$ and distinct vertices $v\in V\cap q$ and $v'\in V\cap q'$ such that
    \[|N_G(v)\cap N_G(v')|\le (1-\eps)\max(|N_G(v)|, |N_G(v')|),\]
    that is, the common neighbours of $v$ and $v'$ in $G$ are at most a $(1-\eps)$-proportion of all neighbours of $v$ or of all neighbours of $v'$ in $G$.
\end{enumerate}
Otherwise, we say that the cell is \emph{good}.
Finally, we construct an auxiliary graph $\Gamma_{\mathrm{bad}}=\Gamma_{\mathrm{bad}}(G)$ whose vertex set is the set of bad cells in $\cQ$ and where a pair of bad cells form an edge if they are at (Euclidean) distance at most $4r$ from each other.

\begin{lemma}\label{lem:close points}
The following statements hold.
\begin{enumerate}[label=$(\mathrm{B}\arabic*)$]
    \item\label{lem:close pointsitem1} For every fixed integer\/ $d\geq1$ and\/ $\eps\in (0, 1/2]$, there exists a constant\/ $M = M(d,\eps)>0$ such that the following holds.
    If\/ $(\log n/n)^{1/d}\le r = o(1)$ and\/ $G\sim G_d(n,r)$, then a.a.s.\ every connected component of\/ $\Gamma_{\mathrm{bad}}$ has fewer than\/ $M$ vertices.
    \item\label{lem:close pointsitem2} Let $d\geq1$ and\/ $\eps\in (0, 1/2]$ be fixed, and let $r=\omega((\log n/n)^{1/d})$ with $r\le \sqrt{d}$.
    Let $G\sim G_d(n,r)$.
    Then, a.a.s.\ there are no bad cells.
\end{enumerate}
\end{lemma}

\begin{proof}
Let $d$ and $\eps$ be given, and define $\delta=\delta(d,\eps)$ as in \eqref{eq:delta}. For each point $x\in [0,1]^d$, define the regions
\[
S_1(x) \coloneqq [0,1]^d\cap B(x,(1+\delta)r),\quad S_2(x) \coloneqq [0,1]^d\cap B(x,(1-\delta)r)\quad \text{and}\quad S_3(x) \coloneqq S_1(x)\setminus S_2(x).
\]
Note that, if $u,u'$ are two vertices as in \ref{bad_condition_2}, then, as the diameter of the union of two neighbouring cells is at most $\delta r$, we have that
\begin{equation}\label{equa:bad10}
    S_2(x)\subseteq [0,1]^d\cap B(u,r)\cap B(u',r)\quad\text{ and }\quad [0,1]^d\cap (B(u,r)\cup B(u',r))\subseteq S_1(x).
\end{equation}

By \eqref{eq:delta}, our choice of $\delta$ yields
\begin{align}
|S_3(x)|&\le ((1+\delta)^d-(1-\delta)^d)|B(x,r)| \nonumber\\
&= ((1+\delta)^d-(1-\delta)^d) (4d)^{d/2} \left|B\left(x, \frac{r}{2\sqrt{d}}\right)\right|\nonumber\\
&\le ((1+\delta)^d-(1-\delta)^d) (4d)^{d/2} 2^d |S_1(x)|=\frac{\eps |S_1(x)|}{4},\label{eq:S-s}
\end{align}
where the second inequality follows from the fact that, for any ball $B$ with radius $r/2\sqrt{d}\le 1/2$ centred at a point in $[0,1]^d$, $|B|\le 2^d |B\cap [0,1]^d|$.
Together with \eqref{eq:S-s}, the following claim ensures that, if a cell is bad, then there is a nearby region with atypically many or atypically few vertices.

\begin{claim}\label{claim:reduced}
    If a cell $q\in\cQ$ with centre $x$ does not satisfy \ref{bad_condition_1}, the presence of a cell $q'\in N_\Gamma(q)\cup\{q\}$ and vertices $u\in V(G)\cap q$ and $u'\in V(G)\cap q'$ witnessing \ref{bad_condition_2} implies that
    \[|V(G)\cap S_3(x)|\geq\frac{\eps}{2}|V(G)\cap S_1(x)|.\]
\end{claim}

\begin{claimproof}
    By \eqref{equa:bad10}, we conclude that
    \begin{align*}
        |V(G)\cap S_2(x)|&\leq|N_{G}(u)\cap N_{G}(u')|+2\\
        &\leq(1-\eps)(\max(|V(G)\cap B(u,r)|,|V(G)\cap B(u',r)|)-1)+2\\
        &\leq(1-\eps)|V(G)\cap S_1(x)|+2\\
        &\leq(1-\eps/2)|V(G)\cap S_1(x)|,
    \end{align*}
    where the second inequality uses \ref{bad_condition_2} and the last inequality relies on $|V(G)\cap S_1(x)|$ being sufficiently large, which is guaranteed if \ref{bad_condition_1} does not hold for $q$.
    The desired inequality now follows by substituting this into the equality $|V(G)\cap S_1(x)|=|V(G)\cap S_2(x)|+|V(G)\cap S_3(x)|$.
\end{claimproof} 

We focus on the proof of \ref{lem:close pointsitem1}.
First, since the diameter of a single cell is bounded from above by~$\delta r/2$,\COMMENT{In fact, the diameter is at most $\delta r/4$.} given a cell $q\in \cQ$ with centre~$x$, the ball $B(x,(16+\delta)r)$ contains all cells at distance at most~$16r$ from~$q$.
In particular, the number of these cells is at most
\[\frac{|B(0,(16+\delta)r)|}{s^d}\leq|B(0,16+\delta)|\bigg(\frac{8\sqrt{d}}{\delta}\bigg)^d\eqqcolon N.\]
Let $\ell\coloneqq \lceil \max\{15\cdot4^dd^{d/2}/\delta^{d}, \ 50\cdot2^d/\theta_d, \ 25\cdot8^{d}d^{d/2}/\eps\theta_d\}\rceil$.
We are going to show that $M \coloneqq 4N\ell$ satisfies the condition of the lemma.

Suppose that a connected component $\Gamma_c$ of $\Gamma_{\mathrm{bad}}$ with $|V(\Gamma_c)|\geq M$ exists.
Given such a component, we algorithmically construct a set $W\subseteq V(\Gamma_c)$ of bad cells at (Euclidean) distance at least $(12-\delta)r$ from each other as follows. 
Set $S \coloneqq V(\Gamma_c)$. 
At the first step, select an arbitrary cell $q\in S$, add it to~$W$ and delete from~$S$ all cells at distance at most $(12-\delta)r$ from $q$. 
Then, as long as $S$ is not empty, select a cell $q\in S$ at smallest distance to $V(\Gamma_c)\setminus S$, add $q$ to $W$ and delete from $S$ all cells at distance at most $(12-\delta)r$ from~$q$.

Observe that the above construction ensures that $|W|\ge 4\ell$ since at most $N$ cells are deleted from~$S$ at every step.
Moreover, the graph $\Gamma_c'$ with vertex set $W$ obtained by adding an edge between every two cells in $W$ at distance at most $16r$ from each other is connected: indeed, by the connectivity of~$\Gamma_c$, every cell added to $W$ is at distance at most $4r$ from an already deleted cell, and cells have diameter at most $\delta r/2$, so the conclusion follows by the triangle inequality.
Then, by pruning a spanning tree of~$\Gamma_c'$, we find a set $W'\subseteq W$ such that $|W'| = 4\ell$ and $\Gamma_c'[W']$ is a connected graph.

The conclusion at this point is that, in order for a connected component $\Gamma_c$ of $\Gamma_{\mathrm{bad}}$ with $|V(\Gamma_c)|\geq M$ to exist, one must be able to find a set of cells $W'\subseteq V(\Gamma_c)$ of size $4\ell$ such that
\begin{enumerate}
    \item every pair of cells in $W'$ are at distance at least $(12-\delta)r$ from each other,
    \item\label{property:blah} there is a labelling $q_1,\ldots,q_{4\ell}$ of the cells in $W'$ such that, for each $i\in[4\ell]\setminus\{1\}$, $q_i$ is at distance at most $16r$ from some cell $q_j$ with $j\in[i-1]$, and 
    \item all cells in $W'$ are bad.
\end{enumerate}
If this occurs, then all cells in a subset $U_1\subseteq W'$ of size $|U_1|=2\ell$ are bad because of condition \ref{bad_condition_1}, or all cells in a subset $U_2\subseteq W'$ of size $|U_2'|=2\ell$ are bad because of condition \ref{bad_condition_2} but not because of \ref{bad_condition_1}.
Our goal now is to show that the probability that such sets $W'$, $U_1$ and $U_2$ exist tends to $0$.
Let $\cU$ denote the family of all possible sets of $2\ell$ cells which are contained in one of the sets $W'$ described above.
For each $U\in\cU$, let $\cB_{\mathrm{a}}(U)$ be the event that all cells in $U$ are bad because of condition \ref{bad_condition_1}, and let $\cB_{\mathrm{b}}(U)$ be the event that all cells in $U$ are bad because of condition \ref{bad_condition_2} but not because of condition \ref{bad_condition_1}.
With this notation, from the discussion above, the probability that $\Gamma_{\mathrm{bad}}$ contains a connected component with at least $M$ cells is bounded from above by 
\begin{equation}\label{eq:BABUbound}
    \sum_{U\in\cU}\mathbb{P}[\cB_{\mathrm{a}}(U)]+\mathbb{P}[\cB_{\mathrm{b}}(U)].
\end{equation}

The number of choices for sets $U\in\cU$ that we need to consider can be bounded as follows.
First, one must construct a set $W'$ of $4\ell$ cells satisfying property~\eqref{property:blah}.
Such a set can be constructed sequentially in such a way that each added cell (except the first one) is at distance at most $16r$ from some previously considered cell.
The number of choices for the first cell is at most $|\cQ|$, and for each subsequent cell we have at most $4\ell N$ choices.\COMMENT{For the second cell we have at most $N$ choices, for the third cell at most $2N$, and so on. Thus, $4\ell N$ is a common upper bound for all of them. This uses the fact that, since the diameter of a cell is at most $\delta r/2$, any cell at distance at most $16r$ from another is contained in the ball of radius $(16+\delta)r$ centred at the centre of this other cell.}
Once $W'$ is defined, one can choose a subset of size $2\ell$ from all the chosen cells.
Thus, using the lower bound on $r$ given in the statement, we conclude that
\begin{equation}\label{eq:union_bound}
    |\cU|\leq|\cQ|(4\ell N)^{4\ell-1}\binom{4\ell}{2\ell}=O(|\cQ|)=o(n).
\end{equation}

Now, fix $U\in\cU$.
Consider the probability that all cells in $U$ are bad because of condition \ref{bad_condition_1}.
For this to happen, the union of the cells must contain at most $\ell s^dn$ vertices.
By Chernoff's inequality (\cref{lem:chernoff}), we conclude that
\begin{equation}\label{eq:empty}
    \mathbb{P}[\cB_{\mathrm{a}}(U)]\leq\mathbb{P}[\mathrm{Bin}(n, 2\ell s^d)\leq \ell s^dn] \leq 2\exp(-\ell s^d n/6)=o(1/n),
\end{equation}
where the last inequality holds by the choice of $\ell\geq15\cdot4^dd^{d/2}\delta^{-d}$.\COMMENT{We want to verify that $\ell s^d n/6\geq c\log n$ for some $c>1$. Since $r=o(1)$, it follows that $s=(1-o(1))\delta r/4\sqrt{d}$, and by the lower bound on $r$ we conclude that $s^d\geq\delta^d\log n/(2\cdot4^dd^{d/2}n)$.
Substituting this, it suffices to verify that $\ell>12\cdot4^dd^{d/2}/\delta^d$, which is satisfied by our choice of $\ell$. (Note that we could improve the $15$ to other smaller constant ($>6$), but this is not important).}

Next, we turn our attention to the possibility that all cells in $U$ are bad because of \ref{bad_condition_2} but not because of \ref{bad_condition_1}.
Label the cells in $U$ as $q_1,\ldots,q_{2\ell}$.
For each $i\in[2\ell]$, let $x_i$ denote the centre of $q_i$. 
Note that, by assumption on the set $U$, any pair of cells in $U$ are at distance at least $(12-\delta)r$ from each other.
Thus, since $\delta\leq 1$, the regions $\{S_1(x_i)\}_{i\in[2\ell]}$ are pairwise disjoint.
For each $i\in [2\ell]$, define the events
\begin{align*}
\cA_i &\coloneq \left\{|V(G)\cap S_1(x_i)|<|S_1(x_i)|n/2\right\},\\
\cB_i &\coloneq \{|V(G)\cap S_3(x_i)|\geq\eps|V(G)\cap S_1(x_i)|/2\}.
\end{align*}
In view of \cref{claim:reduced}, in order to bound $\mathbb{P}[\cB_{\mathrm{b}}(U)]$ from above, it suffices to bound the probability that all of the events $\cB_i$ occur:
\begin{equation}\label{eq:BUbound1}
    \mathbb{P}[\cB_{\mathrm{b}}(U)]\leq\mathbb P\left[\bigcap_{i\in[2\ell]}\cB_i\right].
\end{equation}

Let us first consider the events $\cA_i$.
Let $\cJ$ be the (random) set of indices $i\in [2\ell]$ for which $\cA_i$ does not hold. 
Note that, if $|\cJ|<\ell$, then there are $\ell$ indices $i\in[2\ell]$ for which $\cA_i$ holds.
Thus, using a union bound, we conclude that 
\begin{equation}\label{eq:B-s}
\mathbb{P}[|\cJ|<\ell]\leq\binom{2\ell}{\ell} \max_{I\subseteq [2\ell],\, |I|=\ell} \mathbb P\left[\bigcap_{i\in I}\cA_i\right]\leq4^{\ell} \max_{I\subseteq [2\ell],\, |I|=\ell}\Bigg(\prod_{i\in I} \mathbb P\Bigg[\cA_i\,\,\Bigg|\,\bigcap_{j\in I,\, j < i}\cA_j\Bigg]\Bigg).
\end{equation}

Fix a set $I\subseteq [2\ell]$ with $|I|=\ell$ and an index $i\in I$.
Upon conditioning on the event $\bigcap_{j\in I, j < i}\cA_j$, and since the regions $\{S_1(x_j)\}_{j\in[2\ell]}$ are pairwise disjoint and $r=o(1)$, the random variable $|V(G)\cap S_1(x_i)|$ stochastically dominates a binomial random variable with parameters $3n/4$ and $|S_1(x_i)|$.\COMMENT{All vertices not contained in the $S_1(x_j)$ that we condition on are still ``freely'' available to fall into $S_1(x_i)$. The bound on the number of such ``free'' vertices is very crude. Note also that the probability of success for subsequent events increases, but this is fine since we only need to dominate this variable.}
Therefore, using Chernoff's inequality (\cref{lem:chernoff}), we conclude that\COMMENT{From Chernoff,
\[\mathbb P\left[\mathrm{Bin}\left(\frac34n, |S_1(x_i)|\right)\leq|S_1(x_i)|\frac{n}{2}\right]\le 2\exp\left(-\frac{|S_1(x_i)|n}{36}\right)\le \exp\left(-\frac{|S_1(x_i)|n}{37}\right),\]
with the last inequality holding for sufficiently large $n$.}
\[
\mathbb P\Bigg[\cA_i\,\,\Bigg|\,\bigcap_{j\in I,\, j < i}\cA_j\Bigg]\le \mathbb P\left[\mathrm{Bin}\left(\frac34n, |S_1(x_i)|\right)\leq|S_1(x_i)|\frac{n}{2}\right]\le \exp\left(-\frac{|S_1(x_i)|n}{37}\right).
\]
Recall that for every point $x\in [0,1]^d$ we have that $|S_1(x)|\geq|B(0,r)|/2^d\geq\theta_dr^d/2^d$.
Thus, by our choice of $\ell>37\cdot 2^d/\theta_d$, substituting the above bounds into \eqref{eq:B-s} and using the lower bound on~$r$ implies that
\begin{equation}\label{eq:B-s2}
    \mathbb{P}[|\cJ|<\ell]=o(1/n).
\end{equation}

Finally, we focus on the events $\cB_i$.
Fix a set $J\subseteq[2\ell]$ of size $|J|\ge \ell$ and condition on the event \mbox{$\{\cJ = J\}$} and on the random variables $|V(G)\cap S_1(x_i)|$ for all $i\in J\subseteq [2\ell]$.
Since the regions $\{S_1(x_i)\}_{i\in[2\ell]}$ are pairwise disjoint, in our conditional space, the events $(\cB_i)_{i\in J}$ are mutually independent. 
Combining this with Chernoff's inequality for the binomial random variables $(|V(G)\cap S_3(x_i)|)_{i\in J}$ with (conditional) means $(\mu_i)_{i\in J}$, we obtain that
\begin{align}\label{eq:C-s2}
\mathbb P\left[\bigcap_{i\in J}\cB_i\,\,\middle|\, \cJ=J,(|V(G)\cap S_1(x_i)|)_{i\in J}\right] 
&= \prod_{i\in J} \mathbb P\left[\cB_i\mid \cJ=J,|V(G)\cap S_1(x_i)|\right]\nonumber\\
&\leq \prod_{i\in J} \mathbb P\left[\mathrm{Bin}\left(|V(G)\cap S_1(x_i)|, \frac{|S_3(x_i)|}{|S_1(x_i)|}\right)\ge 2\mu_i\right]\nonumber\\
&\leq 2^{2\ell} \exp\left(-\sum_{i\in J} \frac{\mu_i}{3}\right),
\end{align}
where the first inequality uses the fact that, by~\eqref{eq:S-s},
\[\mu_i = |V(G)\cap S_1(x_i)|\frac{|S_3(x_i)|}{|S_1(x_i)|}\le |V(G)\cap S_1(x_i)|\frac{\eps}{4}.\]
Moreover, for every $i\in J$, using that the event $\cA_i$ does not hold, the bounds on $r$\COMMENT{We use the lower bound clearly at the end, but also the upper bound to say that we lose at most a $2^d$ fraction of the volume.} and the definition of~$S_3(x_i)$, we have that
\begin{align*}
\mu_i = |V(G)\cap S_1(x_i)| \frac{|S_3(x_i)|}{|S_1(x_i)|}\ge \frac{|S_1(x_i)|n}{2}\frac{|S_3(x_i)|}{|S_1(x_i)|}
&\ge \frac{((1+\delta)^d - (1-\delta)^d)\theta_d r^d n}{2^{d+1}}\\*
&\ge \frac{((1+\delta)^d - (1-\delta)^d)\theta_d}{2^{d+1}} \log n.
\end{align*}
Substituting this into \eqref{eq:C-s2}, together with our choice of $\ell\geq25\cdot8^{d}d^{d/2}/\eps\theta_d$, we have that\COMMENT{Using the fact that $|J|\geq\ell$ and that the lower bound on $\mu_i$ is independent of $i$, we have that
\[\exp\left(-\sum_{i\in J} \frac{\mu_i}{3}\right)\leq\exp\left(-\ell\frac{((1+\delta)^d - (1-\delta)^d)\theta_d}{3\cdot2^{d+1}} \log n\right)\]
(note that, for the asymptotic statement, we do not care about the $4^\ell$ term, which is only a constant).
Thus, it suffices to have that $\ell\frac{((1+\delta)^d - (1-\delta)^d)\theta_d}{3\cdot2^{d+1}}>1$.
Using \eqref{eq:delta}, we can rewrite this as $\ell\eps\theta_d>3\cdot8^{d+1}d^{d/2}$.
It thus suffices to have $\ell>24\cdot8^{d}d^{d/2}/\eps\theta_d$.}
\[\mathbb P\left[\bigcap_{i\in J}\cB_i\,\,\middle|\, \cJ=J,(|V(G)\cap S_1(x_i)|)_{i\in J}\right]\leq4^{\ell} \exp\left(-\sum_{i\in J} \frac{\mu_i}{3}\right) = o(1/n).\]
Combining this with \eqref{eq:B-s2} shows that\COMMENT{Here are the details.
We have that
\[\mathbb{P}\left[\bigcap_{i\in[2\ell]}\cB_i\right]=\mathbb{P}\left[\bigcap_{i\in[2\ell]}\cB_i\,\,\middle|\,\,|\cJ|\geq\ell\right]\mathbb{P}[|\cJ|\geq\ell]+\mathbb{P}\left[\bigcap_{i\in[2\ell]}\cB_i\,\,\middle|\,\,|\cJ|<\ell\right]\mathbb{P}[|\cJ|<\ell].\]
By \eqref{eq:B-s2}, the second term is $o(1/n)$, so let us focus on the first term.
We may rewrite it as
\[\mathbb{P}\left[\bigcap_{i\in[2\ell]}\cB_i\,\,\middle|\,\,|\cJ|\geq\ell\right]\mathbb{P}[|\cJ|\geq\ell]=\sum_{\substack{J\subseteq[2\ell]\\|J|\geq\ell}}\mathbb{P}\left[\bigcap_{i\in[2\ell]}\cB_i\,\,\middle|\,\,\cJ=J\right]\mathbb{P}[\cJ=J]\leq\sum_{\substack{J\subseteq[2\ell]\\|J|\geq\ell}}\mathbb{P}\left[\bigcap_{i\in J}\cB_i\,\,\middle|\,\,\cJ=J\right]\mathbb{P}[\cJ=J].\]
The number of summands in this expression is constant, and we showed above that each summand is $o(1/n)$.
Thus, we conclude that the whole sum is $o(1/n)$.}
\begin{equation}\label{eq:C-s}
\mathbb{P}\left[\bigcap_{i\in[2\ell]}\cB_i\right]=o(1/n).
\end{equation}
As a result, \ref{lem:close pointsitem1} follows by combining \eqref{eq:BABUbound}, \eqref{eq:union_bound}, \eqref{eq:empty}, \eqref{eq:BUbound1} and \eqref{eq:C-s}.

The proof of \ref{lem:close pointsitem2} is similar, but much more direct.
A combination of Chernoff's inequality and a union bound shows that a.a.s.\ for every cell $q\in\cQ$ we have that $|V(G)\cap q|=(1\pm o(1))s^dn$.\COMMENT{For a fixed cell $q$, $|V(G)\cap q|$ follows a binomial random variable with expectation $s^dn=\omega(\log n)$, so by Chernoff
\[\mathbb{P}[||V(G)\cap q|-s^dn|\geq\alpha s^dn]\leq2\nume^{\alpha^2s^dn/3}.\]
The claim follows by choosing, say, $\alpha=\alpha(n)=o(1)$ such that $\alpha^2s^dn=\omega(\log n)$ and a union bound over all cells.}
In particular, a.a.s.\ no cell is bad because of condition~\ref{bad_condition_1}.
As for condition \ref{bad_condition_2}, for each $q\in\cQ$, let~$\cB(q)$ denote the event that $q$ is bad because of \ref{bad_condition_2} but not because of \ref{bad_condition_1}.
For a fixed cell~$q\in\cQ$ with centre~$x$, we combine \eqref{eq:S-s}, \cref{claim:reduced}, the fact that $4/3<2\cdot 3/4$ and Chernoff's inequality to bound
\begin{align*}
    \mathbb{P}[\cB(q)]
    &\leq\mathbb{P}\left[|V(G)\cap S_3(x)|\geq\frac{\eps}{2}|V(G)\cap S_1(x)|\right]\\
    &\leq\mathbb{P}\left[|V(G)\cap S_3(x)|\geq \frac43\mathbb E[|V(G)\cap S_3(x)|]\right] + \mathbb{P}\left[|V(G)\cap S_1(x)|\leq \frac34\mathbb E[|V(G)\cap S_1(x)|]\right]\\
    &\leq4\exp(-|S_3(x)|n/48)=o(1/n),
\end{align*}
where the second inequality follows since the event $\{|V(G)\cap S_3(x)|\geq\eps|V(G)\cap S_1(x)|/2\}$ is contained in the event\COMMENT{This is easy to see by considering the complement. We want to show that if $\{|V(G)\cap S_3(x)|\leq 4\mathbb E[|V(G)\cap S_3(x)|]/3\}\cap \{|V(G)\cap S_1(x)|\geq 3\mathbb E[|V(G)\cap S_1(x)|]/4\}$ holds, then $\{|V(G)\cap S_3(x)|<\eps|V(G)\cap S_1(x)|/2\}$. Indeed, using \eqref{eq:S-s} we have that 
\begin{align*}
    |V(G)\cap S_3(x)|\leq \frac43\mathbb E[|V(G)\cap S_3(x)|]=\frac43|S_3(x)|n&\leq\frac43\frac{\eps|S_1(x)|n}{4}=\frac43\frac{\eps\mathbb{E}[|V(G)\cap S_1(x)|]}{4}\\
    &\leq\left(\frac43\right)^2\frac{\eps|V(G)\cap S_1(x)|}{4}<\frac{\eps|V(G)\cap S_1(x)|}{2}.
\end{align*}}
\[\left\{|V(G)\cap S_3(x)|\geq \frac43\mathbb E[|V(G)\cap S_3(x)|]\right\}\cup \left\{|V(G)\cap S_1(x)|\leq\frac34\mathbb E[|V(G)\cap S_1(x)|]\right\}.\]
The claim follows by a union bound over all $q\in\cQ$.
\end{proof}

In order to prove \cref{thm:connectivity} as an application of \cref{lem:k-conn}, we need one last auxiliary result.

\begin{lemma}\label{lem:down}
For every integer\/ $d\ge 1$ and $\eps\in (0, 1/2]$, fix $C_1' = C_1'(d, \eps) \coloneqq 6400\eps^{-2/d} \sqrt{d}$.
Let $C_1' (\log n/n)^{1/d}\le r = o(1)$ and $G\sim G_d(n,r)$.
Then, a.a.s.\ the following holds: every vertex $v\in V(G)$ whose first coordinate $e_1 = e_1(v)$ is at most $2/3$ (resp.\ at least $1/3$) has at least $(1/2-\eps/2)$-proportion of its neighbours in the half-space $h_v^+ \coloneqq \{(x_1,\ldots,x_d)\in\mathbb{R}^d: x_1 > e_1\}$ (resp.\ in the half-space $h_v^- \coloneqq \{(x_1,\ldots,x_d)\in\mathbb{R}^d: x_1 < e_1\}$).
\end{lemma}

\begin{proof}
We show the statement for $h_v \coloneqq h_v^+$ when $e_1\le 2/3$; the proof of the statement for $h_v^-$ when $e_1\ge 1/3$ is analogous.
Fix a vertex $v\in V(G)$ and define
\[p = p(v) \coloneqq |[0,1]^d\cap B(v,r)|\quad \text{ and }\quad p^+ = p^+(v) \coloneqq |[0,1]^d\cap B(v,r)\cap h_v|.\]

Condition on the event that $e_1=e_1(v)\leq2/3$.
Then, by our choice of $v$ and $h_v$ and the fact that $r \le 1/3$, we have that
\begin{equation}\label{eq:pqbound}
    p^+\geq p/2\geq\theta_dr^d/2^{d+1}.
\end{equation}
Note, moreover, that, conditionally on the position of $v$, the random variables $|V(G)\cap B(v,r)|$ and $|V(G)\cap B(v,r)\cap h_v|$ have distribution $\mathrm{Bin}(n-1,p)$ and $\mathrm{Bin}(n-1,p^+)$, respectively.
Thus, by Chernoff's inequality,
\begin{equation}\label{eq:q}
\mathbb P\left[|V(G)\cap B(v,r)|\ge \left(1+\frac{\eps}{10}\right)p(n-1)\right]\leq 2\exp\left(-\frac{\eps^2 p (n-1)}{300}\right)\le \exp\left(-\frac{\eps^2 p n}{400}\right)
\end{equation}
and, similarly,
\begin{equation}\label{eq:p}
\mathbb P\left[|V(G)\cap B(v,r)\cap h_v|\le \left(1-\frac{\eps}{10}\right)p^+(n-1)\right]\le \exp\left(-\frac{\eps^2 p^+ n}{400}\right).
\end{equation}
Note that, by \eqref{eq:pqbound} and \eqref{equa:ballbound}, we have that\COMMENT{Using \eqref{equa:ballbound} we have that \[\eps^2 \frac{\theta_d}{2^{d+1}} (C_1')^d\geq\frac{\eps^2\uppi^{d/2}}{2^{d+2}d^{d/2+1}}(C_1')^d.\] In order for this to be at least $800\log n$, isolating $C_1'$, we need that \[C_1'\geq(4\cdot800d)^{1/d}2\sqrt{d}/\eps^{2/d}\uppi^{1/2},\]
so it suffices to take a $C_1'$ which is an upper bound for this last expression. The term $d^{1/d}/\uppi^{1/2}$ is bounded from above by $1$, and the other term in $(4\cdot 800)^{1/d}$ is maximised when $d=1$. So our choice of $C_1'$ satisfies the inequality.}
\[\eps^2 p^+ n\geq \eps^2 \frac{\theta_dr^d}{2^{d+1}} n \geq \eps^2 \frac{\theta_d}{2^{d+1}} (C_1')^d \log n\geq800\log n.\]
Thus, using that $p\geq p^+$ and that $(1-\eps/10)p^+\ge (1+\eps/10)(1/2-\eps/2)p$, and substituting the above into \eqref{eq:q} and \eqref{eq:p}, we conclude that
\begin{align*}
&\mathbb P\left[|V(G)\cap B(v,r)\cap h_v|\leq\left(\frac12-\frac{\eps}{2}\right)|V(G)\cap B(v,r)|\right]\\
\le\; 
&\mathbb P\left[|V(G)\cap B(v,r)|\ge \left(1+\frac{\eps}{10}\right)p(n-1)\right]+\mathbb P\left[|V(G)\cap B(v,r)\cap h_v|\le \left(1-\frac{\eps}{10}\right)p^+(n-1)\right]\\
\le\; 
&2\exp\left(-\frac{\eps^2 p^+ n}{400}\right) = o(1/n),
\end{align*}
and the lemma follows from a union bound over all $n$ vertices of $G$.
\end{proof}

We now proceed to the main proof in this section.

\begin{proof}[Proof of Theorem~\ref{thm:connectivity}]
Set $C_1 = C_1'$ (which is defined in \cref{lem:down}), let $G\sim G_d(n,r)$, and condition on the event that $G$ satisfies the properties described in \cref{lem:close points} and, if $r=o(1)$, \cref{lem:down} (which hold a.a.s.).
Let $H\subseteq G$ be a $(1/2+\eps)$-subgraph of $G$.

The theorem follows by an application of \cref{lem:k-conn} for $H$ and $k \coloneqq \lfloor \eps s^d n/4\rfloor=\Omega(r^dn)$, using a suitably chosen set $V'$.
Fix a set $U\subseteq V(G)$ of size $k-1$, and suppose that a good cell $q\in\cQ$, a cell $q'\in\cQ$ neighbouring~$q$, and two vertices $v\in V(G)\cap q\setminus U$ and $v'\in V(G)\cap(q\cup q')\setminus U$ are given.
By \ref{bad_condition_1} and \ref{bad_condition_2}, it follows that~$v$ and~$v'$ have at least $\eps |N_G(v)\cap N_G(v')|\geq\eps(|V(G)\cap q|-2)\geq\eps(s^d n/2-2)\geq k$ common $H$-neighbours,\COMMENT{By \ref{bad_condition_2}, the number of common $G$-neighbours of $v$ and $v'$ is at least $(1-\eps)\max\{|N_G(v)|,|N_G(v')|\}\geq(1-\eps)(|N_G(v)|+|N_G(v')|)/2$. By considering $H$, we may delete at most $(1/2-\eps)|N_G(v)|$ common neighbours by deleting neighbours of $v$ and at most $(1/2-\eps)|N_G(v')|$ common neighbours by deleting neighbours of $v'$. Thus, the number of common $H$-neighbours of $v$ and $v'$ is at least $(1-\eps)(|N_G(v)|+|N_G(v')|)/2-(1/2-\eps)|N_G(v)|-(1/2-\eps)|N_G(v')|=\eps(|N_G(v)|+|N_G(v')|)/2$.
Lastly, of course, $|N_G(v)\cap N_G(v')|$ is bounded from above by the size of each of the sets, so the claimed bound follows: $|N_H(v)\cap N_H(v')|\geq\eps|N_G(v)\cap N_G(v')|$.\\
For the first inequality in the text, we use the fact that all vertices contained in $q$ are contained in the common neighbourhood of $v$ and $v'$ except $v$ and $v'$ themselves. The second inequality then follows by using \ref{bad_condition_1}. The last inequality holds since $k=\omega(1)$.} so they are in the same connected component of $H-U$.
Using that every good cell contains at least $s^dn/2\ge 2k$ vertices of $G$, we conclude that, for any set $W\subseteq\cQ$ of good cells such that $\Gamma[W]$ is connected, if we let $V'$ be the set of all vertices which lie in some cell of $W$, then $V'$ satisfies \cref{lem:k-conn}~\ref{lem:k-connitem1}.
\COMMENT{This is simply using, without mention, the fact that one can iteratively connect any pairs of vertices by jumping from neighbouring cell to neighbouring cell until reaching the cell containing the desired vertex.}
In particular, if $r = \omega((\log n/n)^{1/d})$, then we may simply set $V' = V(G)$ and conclude the proof since, by \cref{lem:close points}~\ref{lem:close pointsitem2}, there are no bad cells.

Now, suppose that $r = o(1)$ and consider the $d$-dimensional rectangle
$\Pi \coloneqq [1/3,2/3]\times [0,1]^{d-1}$.
By Lemma~\ref{lem:close points}~\ref{lem:close pointsitem1}, no bad cell intersecting $[1/3, 1/3+2r]\times [0,1]^{d-1}$ is connected in $\Gamma_{\mathrm{bad}}$ to any bad cell intersecting $[2/3-2r,2/3]\times [0,1]^{d-1}$. %Hence,
Using this property, we will show that there is a set~$W$ of good cells contained in~$\Pi$ such that $[0,1]^d\setminus (\bigcup_{q\in W} q)$ consists of two (topologically) connected regions at distance more than~$r$ from each other, and $\{1/3\}\times [0,1]^{d-1}$ and $\{2/3\}\times [0,1]^{d-1}$ are contained in different regions (which allows us to refer to them as the left and right parts, respectively).
In particular, $\Gamma[W]$ is a connected graph and, as discussed above, assumption~\ref{lem:k-connitem1} is satisfied for the set~$V'$ of all vertices in~$W$.
The set $W$ can be constructed as follows (see Figure~\ref{fig:setW} for a schematic representation of part of the construction).

\begin{ConsW}
Consider the union $U$ of the region $[2/3-2r,1]\times [0,1]^{d-1}$ and the balls with radius $4r$ centred at the centres of the bad cells connected to $[2/3-2r,1]\times [0,1]^{d-1}$ in $\Gamma_{\mathrm{bad}}$.
The region $[0,1]^d\setminus U$ may have several simply-connected components, but, since no component of~$\Gamma_{\mathrm{bad}}$ has more than $M=\Theta(1)$ vertices (by \cref{lem:close points}~\ref{lem:close pointsitem1}) and $r=o(1)$, one of these components must contain the $d$-dimensional rectangle $[0,1/3]\times[0,1]^{d-1}$.
Let the region containing this rectangle be denoted by $U'$. 
Then, let~$W$ be the set consisting of all cells contained in $[0,1]^d\setminus U'$ and at distance at most $2r-\sqrt{d}s$ from~$U'$.
Note that all such cells must be good by the definition of $\Gamma_{\mathrm{bad}}$.
Then, the union of the good cells in $W$ disconnects $[0,1/3]\times [0,1]^{d-1}$ and $[2/3,1]\times [0,1]^{d-1}$, and the two regions in $[0,1]^d\setminus\bigcup_{q\in W}q$ containing these $d$-dimensional rectangles remain at distance at least $2r - \sqrt{d}s - \sqrt{d}s > r$ from each other.
\end{ConsW}

\begin{figure}
    \centering
    \begin{tikzpicture}[scale=0.17,x=1cm,y=1cm]
    \clip (-10.5,-4.7) rectangle (45.5,45.5);

\draw [black,line width=1.5pt] (43/2+22,-1) -- (43/2+22,46);
\node[below] at (43/2+22,-1) {$\tfrac23$};
\draw [black,line width=1.5pt] (43/2+14,-1) -- (43/2+14,46);
\node[below] at (43/2+14,-1) {$\tfrac23-2r$};

        \foreach \i in {-10,...,45}{
\draw [line width=0.1pt] (\i,-0.5)-- (\i,45.5);
}
      \foreach \i in {0,...,45}{
\draw [line width=0.1pt] (-10.5,\i)-- (45.5,\i);
}

\draw [fill=black] (11,45) rectangle (17,32);
\draw [fill=black] (12,35) rectangle (18,31);
\draw [fill=black] (10,42) rectangle (11,35);
\draw [fill=black] (12,31) rectangle (22,30);
\draw [fill=black] (13,30) rectangle (24,29);
\draw [fill=black] (14,29) rectangle (25,28);
\draw [fill=black] (15,28) rectangle (26,27);
\draw [fill=black] (16,27) rectangle (26,26);
\draw [fill=black] (17,26) rectangle (27,25);
\draw [fill=black] (19,25) rectangle (27,24);
%\draw [purple, thick] (20,23.5) circle (8);
\draw [fill=black] (20,24) rectangle (27,23);
\draw [fill=black] (17,23) rectangle (27,22);
\draw [fill=black] (14,22) rectangle (27,21);
\draw [fill=black] (12,21) rectangle (27,20);
\draw [fill=black] (11,20) rectangle (26,19);
\draw [fill=black] (10,19) rectangle (25,18);
\draw [fill=black] (9,18) rectangle (24,17);
\draw [fill=black] (8,17) rectangle (22,16);
\draw [fill=black] (7,16) rectangle (17,15);
\draw [fill=black] (7,15) rectangle (15,14);
\draw [fill=black] (6,14) rectangle (14,13);
\draw [fill=black] (6,13) rectangle (13,12);
\draw [fill=black] (6,12) rectangle (13,11);
\draw [fill=black] (5,11) rectangle (12,10);
\draw [fill=black] (5,10) rectangle (12,9);
\draw [fill=black] (5,9) rectangle (12,8);
\draw [fill=black] (5,8) rectangle (12,7);
\draw [fill=black] (5,7) rectangle (12,6);
\draw [fill=black] (5,6) rectangle (12,5);
\draw [fill=black] (5,5) rectangle (12,4);
\draw [fill=black] (6,4) rectangle (13,3);
\draw [fill=black] (6,3) rectangle (13,2);
\draw [fill=black] (6,2) rectangle (14,1);
\draw [fill=black] (7,1) rectangle (15,0);
\foreach \i/\j in {17/44,18/44,17/43,17/42,18/32,18/31,19/31}{
\draw [fill=black] (\i,\j) rectangle (\i+1,\j+1);
}
\draw[blue,thick] (40.5,40.5) circle (8);
%\draw[blue,thick] (35.5,4.5) circle (8);
\centerarc[blue,thick](35.5,4.5)(-40:220:8)

\draw[red,thick] (40.5,40.5) circle (16);
%\draw[red,thick] (35.5,4.5) circle (16);
\centerarc[red,thick](35.5,4.5)(0:200:16)

\draw[blue,thick] (25.5,38.5) circle (8);
\draw[blue,thick] (20.5,7.5) circle (8);
\draw[red,thick] (25.5,38.5) circle (16);

%\draw[red,thick] (20.5,7.5) circle (16);
\centerarc[red,thick](20.5,7.5)(-30:210:16)

%\draw [purple, thick] (20,23.5) circle (8);
\centerarc[magenta,thick](20,23.5)(-90:73:7.77)
\centerarc[magenta,thick](20.5,7.5)(93:260:8.23)
\centerarc[magenta,thick](25.5,38.5)(110:250:8.23)
\foreach \i/\j in {35/3,36/4,36/3,20/7}{
\draw [fill=orange] (\i,\j) rectangle (\i+1,\j+1);
}

\foreach \i/\j in {40/40,41/40,42/40,25/38}{
\draw [fill=orange] (\i,\j) rectangle (\i+1,\j+1);
}
    \end{tikzpicture}
    \caption{The construction of $W$ in the case $d=2$. A part of the grid depicting bad cells (in orange) in two connected components of $\Gamma_{\mathrm{bad}}$ of size $4$ which intersect $[2/3-2r,1]\times[0,1]$, and the set $W$ (in black). The red circles have radius $4r$ and the blue ones have radius $2r$. The set $U$ thus consists of the union of the rectangle $[2/3-2r,1]\times[0,1]$ and the red circles. (This union may have holes, as in this example). The pink contour leaves the area at distance at most $2r-\sqrt{d}s$ from $[0,1]^2\setminus U'$ on its left.}
    \label{fig:setW}
\end{figure}
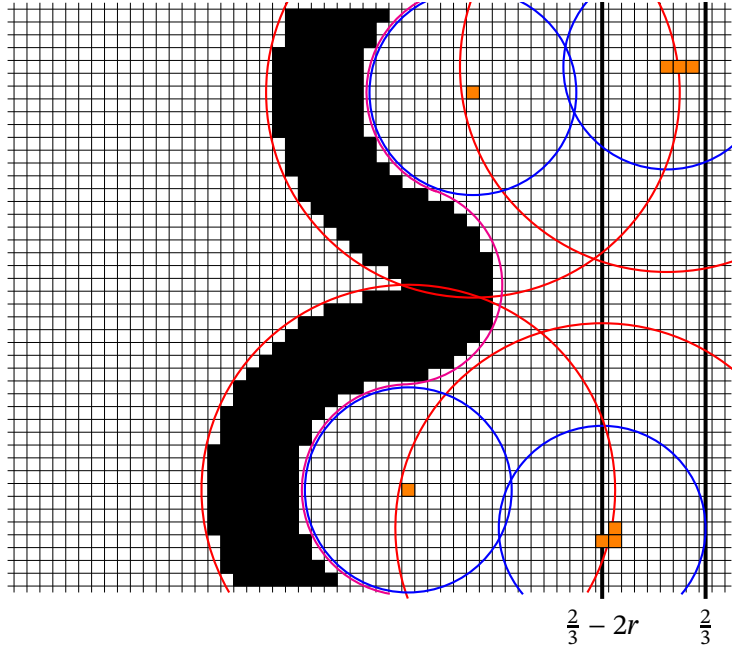

We will next prove that the vertices in $V(G)\setminus V'$ satisfy \cref{lem:k-conn}~\ref{lem:k-connitem2}.
To this end, let $v_1,\ldots,v_m$ (resp.\ $u_1,\ldots,u_\ell$) be an ordering of the vertices in the left part (resp.\ right part) by decreasing (resp.\ increasing) order of their first coordinate.
By Lemma~\ref{lem:down}, and since the left and right parts are at distance more than $r$ from each other, for every $i\in [m]$ we know that 
\[|N_{H}(v_i)\cap (V'\cup \{v_1, \ldots, v_{i-1}\})|\geq \eps |N_G(v_i)|/2,\]
and an analogous statement holds for $u_1,\ldots,u_\ell$.\COMMENT{Note that, for each $i\in[m]$, the set $V'\cup \{v_1, \ldots, v_{i-1}\}$ contains all $G$-neighbours of $v_i$ to its right. Therefore, by \cref{lem:down}, it contains at least a $(1/2-\eps/2)$-proportion of $N_G(v_i)$. By passing to $H$, at least an $\eps/2$-proportion survives.}
It remains to show that for every $v\in V(G)\setminus V'$ we have $|N_G(v)|\ge 2k/\eps$: indeed, the ordering $v_1,\ldots,v_m,u_1,\ldots,u_\ell$ would then satisfy assumption~\ref{lem:k-connitem2} and the proof would be completed. 

We show that every vertex has at least $2k/\eps$ neighbours in $G$.
Indeed, for a fixed vertex $v\in V(G)$, its expected degree satisfies
\[\mathbb{E}[d_G(v)]\geq\frac{\theta_dr^dn}{2^d}\geq\frac{2^dd^{d/2}\theta_ds^dn}{\delta^d}\geq\frac{2^{d+2}d^{d/2}\theta_dk}{\eps\delta^d}\geq \frac{4k}{\eps},\]
where the second inequality uses the definition of $s$, the third one uses the definition of $k$, and the fourth one uses \eqref{equa:ballbound} and the fact that $\delta\leq1$.
Note also that
\[\mathbb{E}[d_G(v)]\geq\frac{\theta_dr^dn}{2^d}\geq\frac{1}{2^dd^{d/2}}C_1^d \log n\ge 25\log n,\]
where the second inequality uses \eqref{equa:ballbound} and the lower bound on $r$, and the third uses the definition of~$C_1$.
Now, an application of Chernoff's inequality and a union bound conclude the proof. \COMMENT{For the fixed vertex $v$ we have by Chernoff that
\[\mathbb{P}[d_G(v)\leq k]\leq\mathbb{P}[d_G(v)\leq\mathbb{E}[d_G(v)]/2]\leq2\nume^{-\mathbb{E}[d_G(v)]/12}\leq2\nume^{-25\log n/12}=o(1/n),\] 
so a union bound over all vertices works.}
\end{proof}

We close this section by showing that the constant $1/2$ in \cref{thm:connectivity} is best possible in a strong sense, by proving \cref{prop:conn_optimal}.
The proof is quite simple and shows that, for a suitably large~$C_2$, one may divide the hypercube $[0,1]^d$ into strips of volume between $2r$ and $4r$ via hyperplanes orthogonal to the $x$-axis and delete the edges crossing these hyperplanes to obtain the required $(1/2-\eps)$-subgraph of $G_d(n,r)$.

\begin{proof}[Proof of Proposition~\ref{prop:conn_optimal}]
The statement holds trivially if $r\ge 1/5$, so we may assume that \mbox{$r < 1/5$}.
We divide $[0,1]^d$ into strips $S_1, \ldots, S_k$ which are orthogonal to the $x$-axis and have width \mbox{$1/k\in [2r, 4r]$} (which is possible since there is at least one integer in the interval $[1/4r, 1/2r]$). 
Set \mbox{$S_0,S_{k+1}\coloneqq\varnothing$}.
Then, for every $i\in[k]$ and every $y\in S_i$, we have that 
\begin{itemize}
    \item $B(y,r)\cap[0,1]^d\subseteq S_{i-1}\cup S_i$ or \mbox{$B(y,r)\cap[0,1]^d\subseteq S_i\cup S_{i+1}$}, and
    \item $|S_{i-1}\cap B(y,r)|,|S_{i+1}\cap B(y,r)|\leq|S_i\cap B(y,r)|$.
\end{itemize}

Let $G \sim G_d(n,r)$ and $V\coloneqq V(G)$. 
An application of Chernoff's bound (\cref{lem:chernoff}) for the binomial variables $|V\cap S_i|$, together with a union bound, shows that a.a.s.\ for every $i\in [k]$, we have\COMMENT{The expected value is $n/k\leq4rn$, which is much larger than $\log n$; in particular, Chernoff's inequality gives very high concentration.} 
\begin{equation}\label{equa:optcon1}
    |V\cap S_i|\leq5rn.
\end{equation}
Moreover, for sufficiently large $C_2$, an application of Chernoff's bound shows that, for every $i\in[k]$ and every vertex $v\in V\cap S_i$,\COMMENT{For the second inequality, the expectation is $|S_i\cap B(v,r)|n\geq|B(v,r)\cap[0,1]^d|n/2$, so by Chernoff's inequality
\[\mathbb P[|V\cap S_i\cap B(v,r)|\le (1-\eps)|B(v,r)\cap[0,1]^d|n/2]\leq2\exp(-\eps^2|B(v,r)\cap[0,1]^d|n/6)\leq2\exp(-\eps^22^{-d}r^d\theta_dn/6) = o(n^{-2}),\]
where the last inequality holds for $C_2$ sufficiently large.
The previous inequality follows similarly.}
\[\mathbb P[|V\cap B(v,r)|\ge (1+\eps)|B(v,r)\cap[0,1]^d|n] = o(1/n)\]
and
\[\mathbb P[|V\cap S_i\cap B(v,r)|\le (1-\eps)|B(v,r)\cap[0,1]^d|n/2] = o(1/n).\]
Hence, by a union bound over all $v\in V$, for every $i\in [k]$ and every vertex $v\in V\cap S_i$, a.a.s.
\begin{equation}\label{equa:optcon2}
    |V\cap S_i\cap B(v,r)|\ge (1-\eps)\frac{|B(v,r)\cap[0,1]^d|n}{2}\geq\frac{1-\eps}{2+2\eps} |V\cap B(v,r)|.
\end{equation}

Condition on the event that $G$ satisfies the properties above.
Now, consider a subgraph $H$ of $G$ obtained by deleting all edges of $G$ which are not contained in one of the strips $S_i$.
Clearly, each component of $H$ is contained in a strip, and so has order at most $5rn$ by \eqref{equa:optcon1}.
Moreover, by \eqref{equa:optcon2}, $H$ is a $(1/2-\eps)$-subgraph of $G$.\COMMENT{Note that $\frac{1-\eps}{2+2\eps}>\frac12-\eps$, so the fact that we are ignoring the central vertex (so there should be $-1$ everywhere when considering degrees) does not affect us.}
\end{proof}

\subsection{Connectivity in random geometric graphs is not born resilient}

We devote this section to proving Theorem~\ref{thm:C}.
First, we focus on the case $d=1$.
The proof relies upon the fact that a.a.s.\ there is an interval $I\subseteq[0,1]$ of length $\log n - o(\log n)$ containing no vertex of $G_1(n,r)$. 
It can then be shown that deleting all edges crossing $I$ a.a.s.\ leaves a disconnected $(1/2+\eps)$-subgraph of $G_1(n,r)$.

\begin{proof}[Proof of Theorem~\ref{thm:C}~\ref{thm:Citem1}]
If $r\leq (\log n-(\log n)^{1/3})/n$, then a.a.s.\ $G\sim G_1(n,r)$ is not connected~\cite[Theorem~12]{GJ96}, so we may assume that $\eps\in(0,1/4]$ and $(\log n-(\log n)^{1/3})/n\leq r\leq \log n/4\eps n$.

Fix $\ell = \ell(n) \coloneqq (\log n-(\log n)^{1/2})/n < r$ and consider a family $\cI$ of $\lfloor n/(3\log n+6rn)\rfloor$ pairwise disjoint intervals of length $\ell+2r$ in the segment $[1/3, 2/3]$.
Note that, if $[a-r,b+r] \in \cI$, then $a-r<a-r+\ell<a<b<b+r-\ell<b+r$.
For every interval $I = [a-r,b+r] \in \cI$, define the event~$\cE_I$ that $V(G)\cap [a,b] = \varnothing$, and the event $\cF_I$ that the intervals $[a-r+\ell,a]$ and $[b,b+r-\ell]$ contain $(r-\ell)n\pm (rn)^{2/3}$ vertices and the intervals $[a-r,a]$ and $[b,b+r]$ contain $rn\pm (rn)^{2/3}$ vertices.
By Chernoff's inequality (for the event $\cF_I$ conditionally on~$\cE_I$), one can verify that, for every~$I\in \mathcal I$,\COMMENT{We can bound $\mathbb P[\cF_I\mid \cE_I]$ from above by the sum of the probabilities that $[a-r+\ell,a]$ and $[b,b+r-\ell]$ contain $(r-\ell)n \pm (rn)^{2/3}$ points and the intervals $[a-r,a]$ and $[b,b+r]$ contain $rn\pm (rn)^{2/3}$ points.
The number of points in the intervals $[a-r,a]$ and $[b,b+r]$ follows a binomial distribution $X\sim\mathrm{Bin}(n,r/(1-\ell))$ and the number of points in the intervals $[a-r+\ell,a]$ and $[b,b+r-\ell]$ follows a binomial distribution $Y\sim\mathrm{Bin}(n,(r-\ell)/(1-\ell))$ (the $1/(1-\ell)$ correction is due to the conditioning on $\cE_I$).
By Chernoff,
\[\mathbb{P}[X\neq rn\pm (rn)^{2/3}]\leq\mathbb{P}[X\neq(1\pm(rn)^{-1/3}/2)\mathbb{E}[X]]\leq2\exp(-(rn)^{-2/3}rn/12(1-p))\leq\exp(-(rn)^{1/3}/15)=o(1),\]
and, similarly,
\[\mathbb{P}[Y\neq(r-\ell)\pm(rn)^{2/3}]\leq\mathbb{P}[Y\neq(r-\ell)\pm((r-\ell)n)^{2/3}]\leq\exp(-((r-\ell)n)^{1/3}/15)=o(1)\]
by the choice of $\ell$ and the lower bound on $r$.}
\[\mathbb P[\cE_I\cap \cF_I] = \mathbb P[\cE_I]\mathbb P[\cF_I\mid \cE_I] = (1\pm o(1))(1-\ell)^{n} = (1\pm o(1))\nume^{-\ell n}.\]
Similarly, for every pair of distinct intervals $I,J\in \mathcal I$,
\begin{align*}
\mathbb P[\cE_I\cap \cF_I\cap \cE_J\cap \cF_J] = \mathbb P[\cE_I\cap \cE_J]\mathbb P[\cF_I\cap \cF_J\mid \cE_I\cap \cE_J] = (1\pm o(1))(1-2\ell)^{n} = (1\pm o(1))\nume^{-2\ell n}. 
\end{align*}
Thus, a second moment argument shows that a.a.s.\ the event $\cE_I\cap \cF_I$ holds for at least one interval $I\in \cI$.\COMMENT{For each $I\in\cI$, let $X_I$ be the indicator random variable that $\cE_I\cap \cF_I$ holds, and let $X\coloneqq\sum_{I\in\cI}X_I$.
We have shown that $\mathbb{E}[X_I]=(1\pm o(1))\nume^{-\ell n}$ and $\mathbb{E}[X_IX_J]=(1\pm o(1))\nume^{-2\ell n}$.
Together with the fact that $\operatorname{Var}[X_I]=\mathbb{E}[X_I](1-\mathbb{E}[X_I])=(1\pm o(1))\nume^{-\ell n}$ (since $\nume^{-\ell n}=o(1)$), we conclude that $\mathbb{E}[X]=\lfloor n/(3\log n+6rn)\rfloor(1\pm o(1))\nume^{-\ell n}$ and 
\[\operatorname{Var}[X]\leq\lfloor n/(3\log n+6rn)\rfloor(1\pm o(1))\nume^{-\ell n}+\lfloor n/(3\log n+6rn)\rfloor^2(1\pm o(1))\nume^{-2\ell n}\leq n\nume^{-\ell n}.\]
Therefore, by Chebyshev's inequality, 
\[\mathbb{P}[X=0]\leq\frac{\operatorname{Var}[X]}{\mathbb{E}[X]^2}\leq\frac{n\nume^{-\ell n}}{(1\pm o(1))\lfloor n/(3\log n+6rn)\rfloor^2\nume^{-2\ell n}}=\Theta(\log^2n\nume^{-\sqrt{\log n}})=o(1).\]}
We condition on this event, and let $I = [a-r,b+r]\in \cI$ be one interval for which $\cE_I\cap \cF_I$ holds.

Let $H\subseteq G$ be obtained by deleting all edges $uv$ with $u\in [0,a)$ and $v\in (b,1]$.
The probability that $V(G)\cap [0,1/3) = \varnothing$ or $V(G)\cap (2/3,1] = \varnothing$ is exponentially small (both in the original probability space and in our conditional setting), so a.a.s.\ $H$ is a disconnected graph.
Moreover, $H$ is a $(1/2+\eps)$-subgraph of $G$ since, for every vertex $v\in V(G)\cap ([0,a)\cup (b,1])$, we have that\COMMENT{The first inequality is simply from the conditioning above. Now say that $r=(1\pm o(1))x\log n/n$ for some $x\in[1,1/4\eps]$. Then, substituting the values of $r$ and $\ell$, we have that 
\[\frac{rn-(rn)^{2/3}-1}{(2r-\ell)n+2(rn)^{2/3}}=(1\pm o(1))\frac{x}{2x-1}.\]
This function is decreasing as $x$ increases, so the minimum is achieved when $x=1/4\eps$.
Lastly, note that \[\frac{1/4\eps}{2/4\eps-1}=\frac{1}{2-4\eps}>\frac12+\eps.\]}
\[\frac{d_H(v)}{d_G(v)}\geq \frac{rn-(rn)^{2/3}-1}{(2r-\ell)n+2(rn)^{2/3}}\geq (1-o(1))\frac{1}{2-4\eps}>\frac12+\eps,\]
which finishes the proof. 
\end{proof}

Now, we fix $d=2$ and direct our attention to the proof of Theorem~\ref{thm:C}~\ref{thm:Citem2}. 
The main idea of the proof is similar to the $1$-dimensional case and goes roughly as follows.
We look for an annulus in~$[0,1]^2$ with radii of order $\log n/n$ which contains no vertices of $G_2(n,r)$.
Then, one can delete all edges of $G_2(n,r)$ which cross this annulus to obtain a disconnected graph.
This idea unfortunately fails for two dimensions since the width of the empty annulus is not sufficient to compensate the fact that, due to the curvature, some vertices in the ball inside the annulus are still expected to have more neighbours beyond the empty annulus than inside the ball, so the resulting graph after those edge deletions will not be a $(1/2+\eps)$-subgraph of $G_2(n,r)$.
To correct this, we look for two concentric annuli in $[0,1]^2$ with radii of order $\log n/n$ such that (i) the external boundary of the internal annulus coincides with the internal boundary of the external annulus; (ii) the external annulus contains no vertices of $G_2(n,r)$, and (iii) the density of vertices in the internal annulus is atypically high.
While this last assumption requires to reduce the width of the external annulus, careful comparison of large deviations shows that, by choosing the parameters of the problem in a suitable way, the density boost in the internal annulus manages to compensate the effect of the curvature.

We turn to the technical details. First, we prepare the terrain with a couple of auxiliary lemmas.
Set $\zeta=\zeta(n)\coloneqq(\log n/n)^{1/2}$.
For any given constant $A>0$, we define $t_A$ as the unique positive solution of the equation
\[\uppi ((1+t)A)^2 - \uppi A^2 = \uppi A^2(t^2+2t) = 1.\]
We define a function $c\colon[0,+\infty)^4\to\mathbb{R}$ by setting
\[c = c(A,C,t,\nu) \coloneqq 1 - \uppi A^2 (t^2+2t) - 4\uppi C (A-C) \left(1 - (1+\nu)\log\left(\frac{\nume}{1+\nu}\right)\right).\]

\begin{fact}\label{fact:AC}
For all fixed $C > 0$, $A>C$ and $t\in [0, t_A]$, there is a unique non-negative solution $\nu_t$ to the equation 
\begin{equation}\label{eq:relate t nu}
c(A,C,t,\nu)=0.
\end{equation}
Moreover, if $t\in[0,t_A)$ and $\nu\in[0,\nu_t)$, then $c(A,C,t,\nu)>0$.
\end{fact}

\begin{proof}
The function $f\colon[0, \infty)\to\mathbb{R}$ defined by $f(\nu)=(1+\nu)\log(\nume/(1+\nu))$ is decreasing (so $c$ is also decreasing as a function of $\nu$), $f(0)=1$ and $f(\nu)$ tends to $-\infty$ when $\nu\to \infty$.
Since $c(A,C,t,0)\ge 0$ by the definition of $t_A$, the claim follows.
\end{proof}

Let $C>0$ and $A\geq2C$ be fixed constants.
For every $t\ge 0$, define $Z_t \coloneqq ((1+t)A+2C)\zeta$.
For any point $x\in[0,1]^2$ and any $t,\nu\geq0$, we define the events
\begin{align*}
\cE_{x,t} = \cE_{x,t}(A) &\coloneqq \{V(G_2(n,r))\cap B(x, A\zeta, (1+t)A\zeta) = \varnothing\},\\
\cF_{x,\nu} = \cF_{x,\nu}(A,C) &\coloneqq \{|V(G_2(n,r))\cap B(x,(A-2C)\zeta,A\zeta)|\in [(1+\nu)4\uppi C(A-C)\log n,(\log n)^2]\},\\
\cG_{x,t} = \cG_{x,t}(A,C) 
&\coloneqq \{|V(G_2(n,r))\cap B(x,(1+t)A\zeta,Z_t)|=4\uppi(C+(1+t)A)C\log n\pm(\log n)^{2/3}\}.
\end{align*}

\begin{lemma}\label{lem:AC}
Fix $C > 0$ and $A\ge 2C$.
For every $t\in [0, t_A)$ and every $\nu\in (0, \nu_t)$,
a.a.s.\ $\cE_{x,t}\cap \cF_{x,\nu}\cap \cG_{x,t}$ holds for some point $x\in [Z_{t_A},1-Z_{t_A}]^2$.
\end{lemma}

\begin{proof}
Fix $t\in [0, t_A)$, $\nu \in (0, \nu_t)$ and a point $x\in [Z_{t_A},1-Z_{t_A}]^2$. 
Then,
\begin{equation}\label{eq:P(E)}
\mathbb P[\cE_{x,t}] = (1 - \uppi A^2(t^2+2t)\zeta^2)^n = \exp(-(1+o(1)) \uppi A^2 (t^2+2t) \log n) = n^{-\uppi A^2 (t^2+2t)\pm o(1)}.    
\end{equation}
Note that, by our assumption that $t\in [0,t_A)$, we have that $\uppi A^2 (t^2+2t)\in [0,1)$.
At the same time, setting $i_0 \coloneqq \lceil (1+\nu)4\uppi C(A-C)\log n\rceil$ and $p \coloneqq 4\uppi C(A-C)\zeta^2 /(1 - \uppi A^2 (t^2+2t) \zeta^2)$, we have that
\begin{equation}\label{eq:FcondE}
\mathbb P[\cF_{x,\nu}\mid \cE_{x,t}] = \sum_{i = i_0}^{(\log n)^2} \binom{n}{i} p^i (1-p)^{n-i}.
\end{equation}
Moreover, as for every $i\ge i_0$ we have that
\COMMENT{For the inequality, we use that 
\begin{align*}
    \frac{(n-i)p}{(i+1)(1-p)}&\leq\frac{(n-i_0)p}{(i_0+1)(1-p)}\\
    &\leq\frac{(n-(1+\nu)4\uppi C(A-C)\log n)4\uppi C(A-C)\log n}{(n-\uppi A^2 (t^2+2t)\log n)(1+(1+\nu)4\uppi C(A-C)\log n)\left(1-\frac{4\uppi C(A-C)\log n}{n - \uppi A^2 (t^2+2t) \log n}\right)}\\
    &=\frac{(n-(1+\nu)4\uppi C(A-C)\log n)4\uppi C(A-C)\log n}{(1+(1+\nu)4\uppi C(A-C)\log n)\left(n - \uppi A^2 (t^2+2t) \log n-4\uppi C(A-C)\log n\right)}.
\end{align*}
Now one observes immediately that the numerator is $(1-o(1))4\uppi C(A-C)n\log n$, while the denominator is $(1+o(1))(1+\nu)4\uppi C(A-C)n\log n$. Given the signs of the smaller order terms, the desired bound follows.}
\[\binom{n}{i+1}p^{i+1}(1-p)^{n-i-1}\bigg{/} \binom{n}{i}p^{i}(1-p)^{n-i} = \frac{(n-i)p}{(i+1)(1-p)}\le \frac{1}{1+\nu},\]
the sum in~\eqref{eq:FcondE} is of the order of its first term, that is,\COMMENT{The first bound follows since we have a geometric sum. In fact, we have that 
\[\mathbb P[\cF_{x,\nu}\mid \cE_{x,t}]=\sum_{i = i_0}^{n^{0.1}} \binom{n}{i} p^i (1-p)^{n-i}\leq\binom{n}{i_0} p^{i_0} (1-p)^{n-i_0} \sum_{j=0}^{\infty} \frac{1}{(1+\nu)^j}=\frac{1+\nu}{\nu}\binom{n}{i_0} p^{i_0} (1-p)^{n-i_0}.\]
For the second bound, using Stirling's approximation, we have that
\begin{align*}
    \binom{n}{i_0}\sim\frac{\sqrt{2\uppi n}}{\sqrt{2\uppi i_0}\sqrt{2\uppi(n-i_0)}}\frac{\left(\frac{n}{\nume}\right)^n}{\left(\frac{n-i_0}{\nume}\right)^{n-i_0}\left(\frac{i_0}{\nume}\right)^{i_0}}=\Theta\left(i_0^{-1/2}\frac{n^n}{(n-i_0)^{n-i_0}i_0^{i_0}}\right).
\end{align*}
We also use the usual bounds $1-p=\nume^{-(1+o(1))p}$.
Moreover, since $i_0p=o(1)$, we have that $\nume^{i_0p}\leq2$.}
\[\mathbb P[\cF_{x,\nu}\mid \cE_{x,t}]=\Theta\left(\binom{n}{i_0} p^{i_0} (1-p)^{n-i_0}\right)=\Theta\left(i_0^{-1/2}\frac{n^n}{(n-i_0)^{n-i_0}}\left(\frac{p}{i_0}\right)^{i_0}\nume^{-np}\right),\]
where the second equality follows by Stirling's approximation.
In particular, since we have that $np/i_0=(1+o(1))/(1+\nu)$ by definition, then\COMMENT{We have that
\begin{align*}
\log\mathbb P[\cF_{x,\nu}\mid \cE_{x,t}]&=n\log n-(n-i_0)\log(n-i_0)+i_0 \log p-i_0\log i_0-np-\log i_0/2\\
&=i_0+i_0\log n+i_0 \log p-i_0\log i_0-np+o(\log n)\\
&=i_0\log\frac{\nume np}{i_0}-np+o(\log n)=i_0\log\frac{\nume}{1+\nu}-np+o(\log n),
\end{align*}
where the first equality uses the fact that $\log(n-i_0)=\log n-i_0/n-O(i_0^2/n^2)$ (this is used to obtain the first $i_0$ as well as to hide several terms as smaller order terms), and in the fourth we are using the fact that (as follows similarly as in a previous footnote) $np/i_0=(1+o(1))/(1+\nu)$.}
\begin{align}\label{eq:nu}
\log\mathbb P[\cF_{x,\nu}\mid \cE_{x,t}]&=i_0 \log\left(\frac{\nume}{1+\nu}\right) - np + o(\log n)\nonumber\\*
&= \left((1+\nu)4\uppi C(A-C)\log\left(\frac{\nume}{1+\nu}\right) - 4\uppi C(A-C) \pm o(1)\right) \log n.
\end{align}
Finally, an immediate application of Chernoff's bound (\cref{lem:chernoff}) shows that $\mathbb P[\cG_{x,t}\mid \cE_{x,t}\cap \cF_{x,\nu}] = 1-o(1)$.\COMMENT{The annulus under question has area $4\uppi C(C+(1+t)A)\log n/n)$. Upon conditioning on each possible outcome of the events $\cE_{x,t}$ and $\cF_{x,\nu}$, the number of points that are randomly assigned a position is at most $n$.
The probability that each point falls in the desired surface must be corrected by a factor of $1/(1-\uppi((1+t)A)^2\log n/n)$ by conditioning that it does not fall in a place that affects the events we conditioned on, but this correction is essentially negligible.
So the expected number of vertices in this annulus is at most $(1+O(\log n/n))4\uppi C(C+(1+t)A)\log n$ (and at least $(1-O(n^{-0.9}))4\uppi C(C+(1+t)A)\log n$, so it is $\Theta(\log n)$).
The probability that there is a deviation of order $\log^{2/3}n$ can thus be bounded by $\exp(-\Theta(\log^{1/3} n))=o(1)$.}
Combining this with~\eqref{eq:P(E)} and \eqref{eq:nu}, we have that 
\begin{equation}\label{equa:EFGprob1}
    \mathbb{P}[\cE_{x,t}\cap \cF_{x,\nu}\cap \cG_{x,t}] = \mathbb{P}[\cE_{x,t}] \mathbb{P}[\cF_{x,\nu}\mid \cE_{x,t}] P[\cG_{x,t}\mid \cE_{x,t}\cap \cF_{x,\nu}] = n^{c-1\pm o(1)}
\end{equation}
with $c = c(A,C,t,\nu) > 0$ (see \cref{fact:AC}).

Now, fix $t\in [0,t_A)$ and $\nu\in (0,\nu_{t})$ and consider a set $S$ of $k \coloneq \lceil n/(\log n)^3\rceil$ points $\{x_1,\ldots,x_k\}$ in~$[0,1]^2$ at distance at least $2Z_{t_A}$ from the boundary and from each other. 
We will show that $\cE_{x,t}\cap \cF_{x,\nu}\cap \cG_{x,t}$ holds a.a.s.\ for at least one of the points $x\in S$.
Let $X$ be the sum of the indicator random variables of the events $\cE_{x_i,t}\cap \cF_{x_i,t}\cap \cG_{x_i,t}$. 
From \eqref{equa:EFGprob1}, we have that $\mathbb E[X] = k n^{c-1\pm o(1)}=\omega(1)$. 
On the other hand, since the set $\{B(x_i,Z_t)\}_{i\in[k]}$ consists of pairwise disjoint balls, computations similar to the ones above show that, for every pair of distinct integers $i,j\in [k]$,\COMMENT{The first two are exactly the same (for the second one, one must slightly change the definition of $p$, but the correction does not change the asymptotics). The last two are also very similar, just direct applications of Chernoff's inequality. The important one is the third one. Here, one can use stochastic domination to write that 
\[\sum_{i = i_0}^{n^{0.1}} \binom{n-n^{0.1}}{i} p^i (1-p)^{n-n^{0.1}-i}\leq\mathbb P[\cF_{x_j,\nu}\mid \cE_{x_i,t}\cap \cE_{x_j,t}\cap \cF_{x_i,\nu}]\leq\sum_{i = i_0}^{n^{0.1}} \binom{n}{i} p^i (1-p)^{n-i}.\]
The log of the upper bound gives the same as before. 
For the lower bound, however, we get something worse: using the Stirling approximation, we have that
\[P[\cF_{x_j,\nu}\mid \cE_{x_i,t}\cap \cE_{x_j,t}\cap \cF_{x_i,\nu}]=\Omega\left(i_0^{-1/2}\frac{(n-n^{0.1})^{n-n^{0.1}}}{(n-n^{0.1}-i_0)^{n-n^{0.1}-i_0}}\left(\frac{p}{i_0}\right)^{i_0}\nume^{-(n-n^{0.1})p}\right).\]
Taking logarithms and arguing like before, the right side behaves in the same way (all the new terms also become smaller order terms).}
\begin{align*}
\mathbb P[\cE_{x_i,t}\cap \cE_{x_j,t}] 
&= (1\pm o(1)) \mathbb P[\cE_{x,t}]^2,\\
\mathbb P[\cF_{x_i,\nu}\mid \cE_{x_i,t}\cap \cE_{x_j,t}] 
&= (1\pm o(1))\mathbb P[\cF_{x,\nu}],\\
\mathbb P[\cF_{x_j,\nu}\mid \cE_{x_i,t}\cap \cE_{x_j,t}\cap \cF_{x_i,\nu}] 
&= (1\pm o(1))\mathbb P[\cF_{x,\nu}],\\
\mathbb P[\cG_{x_i,t}\mid \cE_{x_i,t}\cap \cE_{x_j,t}\cap \cF_{x_i,\nu}\cap \cF_{x_j,\nu}] 
&= 1-o(1),\\
\mathbb P[\cG_{x_j,t}\mid \cE_{x_i,t}\cap \cE_{x_j,t}\cap \cF_{x_i,\nu}\cap \cF_{x_j,\nu}\cap \cG_{x_i,t}] 
&= 1-o(1).
\end{align*}
In particular, $\operatorname{Var}[X]$ is bounded from above by
\[\mathbb E[X]+\sum_{\substack{(i,j)\in[k]^2\\i\neq j}} \left(\mathbb P[\cE_{x_i,t}\cap \cE_{x_j,t}\cap \cF_{x_i,\nu}\cap \cF_{x_j,\nu}\cap \cG_{x_i,t}\cap \cG_{x_i,t}] - \mathbb P[\cE_{x,t}\cap \cF_{x,\nu}\cap \cG_{x,t}]^2\right) = o(\mathbb E[X]^2),\]
so by Chebyshev's inequality we conclude that a.a.s.\ $\cE_{x,t}\cap \cF_{x,\nu}\cap \cG_{x,t}$ holds for at least one of the points $x\in S$.
\end{proof}

When $t=0$,~\eqref{eq:relate t nu} implies that
\[1 - (1+\nu_0)\log\left(\frac{\nume}{1+\nu_0}\right) = \frac{1}{4\uppi C(A-C)}.\]
Note that, for any fixed $C > 0$, $\nu_0$ decreases to $0$ as $A$ grows to infinity.
Thus, by using Taylor expansion in the regime when $\nu_0\to 0$, we have that 
\[1 - (1+\nu_0)\log\left(\frac{\nume}{1+\nu_0}\right) = (1+\nu_0)\log(1+\nu_0) - \nu_0 = \frac{\nu_0^2}{2} - O_{\nu_0}(\nu_0^3).\]
In particular, for all $C>0$ and sufficiently large $A>0$, one may observe that\COMMENT{Note that here is the first key moment in the construction: the dependence of $\nu_0$ (and, in fact, of any $\nu_t$ for $t < t_A$) on $C,A$ is roughly $1/\sqrt{CA}$ for $d=2$, and $1/\sqrt{CA^{d-1}}$ in general.}
\begin{equation}\label{eq:value_nu}
\nu_0 = \frac{1+o_A(1)}{\sqrt{2\uppi C (A-C)}}.
\end{equation}
Fix any $x\in\mathbb{R}^2$.
For every point $y\in\mathbb{R}^2$, every $t\geq0$ and $\rho = \rho(n)\ge 0$, define the regions
\begin{align*}
R_1(y, \rho) &\coloneqq B(y,\rho)\cap B(x,A\zeta),\\
R_2(y, t, \rho) &\coloneqq B(y,\rho)\setminus B(x,(1+t)A\zeta).
\end{align*}
We continue with a geometric lemma.

\begin{lemma}\label{lem:geom}
Let $C > 0$ and $x\in\mathbb{R}^2$ be fixed.
For all sufficiently large $A$, there exist $t\in (0, t_A)$, $\nu\in (0, \nu_t)$ and $\eps = \eps(C, A, t, \nu) > 0$ such that the following hold.
\begin{enumerate}[label=$(\mathrm{\roman*})$]
    \item For all $y\in \partial B(x,A\zeta)$ we have that $(1+\nu)|R_1(y, C\zeta)|\geq(1+10\eps)|R_2(y,t,C\zeta)|$.
    \item For all $z\in \partial B(x, (1+t)A\zeta)$ we have that $|R_2(z,t,C\zeta)|\geq(1+10\eps)(1+\nu)|R_1(z,C\zeta)|$.
\end{enumerate}
\end{lemma}

We remark that this lemma is ``scale-free'', in the sense that $\zeta$ could be replaced (in the statement as well as the definitions of $R_1$ and $R_2$) by any positive real number without changing the conclusion.
Our choice of phrasing is adapted to our later application.

\begin{proof}[Proof of \cref{lem:geom}]
\begin{figure}
\centering
\begin{tikzpicture}[scale=0.7,line cap=round,line join=round,x=1cm,y=1cm]
\clip(-14,-0.3) rectangle (6,6.61);
\draw [line width=0.5pt] (-4,0) circle (5cm);
\draw [line width=0.5pt] (-4,0) circle (5.3cm);
\draw [line width=0.5pt] (-7,4) circle (1.4142135623730951cm);
\draw [line width=0.5pt] (-7.18,4.24) circle (1.4142135623730951cm);
\draw [line width=0.5pt] (-8,3)-- (-4,0);
\draw [line width=0.5pt] (-7,4)-- (-4,0);
\draw [line width=0.5pt] (-7.18,4.24)-- (-7,4);

\draw [shift={(-4,0)},line width=1.8pt,color=black, fill=black, opacity=0.3] (126.86989764584402:2) arc (126.86989764584402:143.13010235415598:2);

\begin{scriptsize}
\draw [fill=black] (-4,0) circle (2pt);
\draw [fill=black] (-7,4) circle (2pt);
\draw [fill=black] (-7.18,4.24) circle (2pt);
\draw [fill=black] (-5.76,4.68) circle (2pt);
\draw [fill=black] (-8,3) circle (2pt);

\draw[color=black] (-3.7,0) node {\large{$x$}};
\draw[color=black] (-7.1,3.5) node {\large{$y$}};
\draw[color=black] (-7.9,2.5) node {\large{$y_1$}};
\draw[color=black] (-5.3,4.5) node {\large{$y_2$}};
\draw[color=black] (-7.55+0.3,5-0.4) node {\large{$z$}};

\draw[color=black] (-5.6,1.6) node {\large{$\theta$}};

\draw[color=black] (-0.65,1.6) node {\large{$\partial B(x, A\zeta)$}};
\draw[color=black] (1.7,4.2) node {\large{$\partial B(x, (1+t)A\zeta)$}};
\end{scriptsize}
\end{tikzpicture}
\caption{The configuration from the proof of Lemma~\ref{lem:geom}.}
\label{fig:1}
\end{figure}
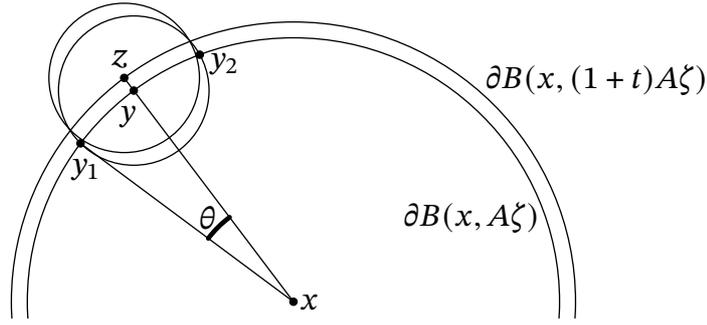

For $y\in\mathbb{R}^2$, define for brevity $R_1(y) \coloneqq R_1(y, C\zeta)$ and $R_2(y,t) \coloneqq R_2(y,t,C\zeta)$.
Fix a point $y\in \partial B(x, A\zeta)$ and let $z$ be the intersection point of the ray $xy$ and $\partial B(x,(1+t)A\zeta)$.
Note that $z$ is a function of $t$, but we omit the dependence from the notation for simplicity.
We will show that, when $t=0$, we have $(1+\nu_0)|R_1(y)| > |R_2(y,0)|$ (and hence $(1+\nu_0)|R_1(z)| > |R_2(z,0)|$ \COMMENT{Simply because $z=y$.}), but when $t = t_A$, we have $|R_2(z,t_A)| > (1+\nu_{t_A})|R_1(z)|$.\COMMENT{Note that $\nu_{t_A}=0$.}
Then, by continuity of
\[(t,\nu)\mapsto (|R_2(y,t)|-(1+\nu)|R_1(y)|, |R_2(z,t)|-(1+\nu)|R_1(z)|)\]
and the fact that the second component is strictly larger than the first one for any choice of parameters $t\in (0,t_A)$ and $\nu\in (0,\nu_t)$, we can find $t$ and $\nu$ for which
\[(1+\nu)|R_1(y)| > |R_2(y,t)|\qquad\text{ and }\qquad |R_2(z,t)| > (1+\nu)|R_1(z)|,\]
which implies the lemma.\COMMENT{The idea is that the first inequality we will prove (which yields the second inequality between parenthesis) ensures that the image of this function contains a point with both coordinates being negative. We need to prove that there exists a point in the image whose first coordinate is negative and whose second coordinate is positive. By continuity, and since the image is above the $y=x$ line, for this it suffices to show that there is some point in the image whose second coordinate is positive. And this is ensured by the last inequality we claim.}

Denote by $y_1, y_2$ the intersection points of the circle $\partial B(y,C\zeta)$ with $\partial B(x,A\zeta)$, and let $\theta$ be the measure of the angle $yxy_1$, see Figure~\ref{fig:1}.
Then, $R_1(y)$ is obtained from the sector between~$yy_1$ and~$yy_2$ in $B(y,C\zeta)$ by adding the caps that $yy_1$ and $yy_2$ separate from $B(x,A\zeta)$.
Since the angle $y_1yy_2$ has measure $\uppi-\theta$,
\[|R_1(y)|=\frac{\uppi-\theta}{2} C^2 \zeta^2 + 2\left(\frac{\theta}{2} - \frac{\sin\theta}{2}\right) A^2 \zeta^2 = \frac{\uppi-\theta}{2} C^2 \zeta^2 + (\theta - \sin\theta) A^2 \zeta^2.\]
Using that $x - \sin x = x^3/6 + O_x(x^5)$ when $x\to 0$, the fact that $\sin(\theta/2) = C/2A$ (which means that $\theta = C/A + O_A(C^3/A^3)$) and the relation $A = \omega_A(C)$, we have that the above expression rewrites as
\begin{equation}\label{eq:interior}
|R_1(y)|=\frac{\uppi}{2} C^2 \zeta^2 - \left(\frac{\theta}{2} C^2 - \frac{\theta^3}{6} A^2 + O_A(\theta^5 A^2)\right)\zeta^2 = \frac{\uppi}{2} C^2 \zeta^2 - \left(\frac{C^3}{3A} + o_A\left(\frac{C^3}{A}\right)\right)\zeta^2.
\end{equation}
In particular, for all sufficiently large $A$ and using \eqref{eq:value_nu}, we have that
\[|R_2(y,0)|-|R_1(y)| = 2\left(\frac{C^3}{3A} + o_A\left(\frac{C^3}{A}\right)\right)\zeta^2 < \frac{\uppi C^2 \zeta^2/3}{\sqrt{2\uppi C(A-C)}} < \nu_0 |R_1(y)|,\]
which implies that $(1+\nu_0)|R_1(y)| > |R_2(y,0)|$.\COMMENT{Here is the other crucial moment: the general version of this inequality is $C^{d+1}/A = O(C^d/\sqrt{C A^{d-1}})$, which is not true for large $A$ if $d\ge 3$.}

It remains to show that $|R_2(z,t_A)| > (1+\nu_{t_A})|R_1(z)| = |R_1(z)|$.
To this end, note that the angles $xy_1y$ and $yy_2x$ are acute and, thus, the part of the annulus between the rays $xy_2$ and $xy_1$ contains $B(x,A\zeta,(1+t_A)A\zeta)\cap B(y, C\zeta)$. 
Thus, by \eqref{eq:interior} and using the definition of $t_A$,
\begin{align}
|R_2(y,t_A)| \ge \uppi C^2\zeta^2 - |R_1(y)| - \frac{\theta}{\uppi} \zeta^2 
&= \frac{\uppi}{2} C^2\zeta^2 + \bigg(\frac{C^3}{3A}+o_A\bigg(\frac{C^3}{A}\bigg)-\frac{C}{\uppi A}\bigg)\zeta^2\nonumber\\*
&= |R_1(y)| + \bigg(\frac{2C^3}{3A}+o_A\bigg(\frac{C^3}{A}\bigg)-\frac{C}{\uppi A}\bigg)\zeta^2.\label{eq:R1y}
\end{align}
At the same time, using that $C$ is a fixed constant, $A$ is large and $t_A = 1/2\uppi A^2 + o_A(1/A^2)$, up to lower order terms, the region $(R_1(y)\setminus R_1(z))\cup (R_2(z,t_A)\setminus R_2(y,t_A))$ consists of the symmetric difference of $B(y,C\zeta)$ and $B(z,C\zeta)$. 
Setting $w$ to be an intersection point of the last two circles, $\varphi$ to be the measure of the angle $wyz$ and $\ell \coloneqq |yz| = At_A \zeta = (1+o_A(1))\zeta/2\uppi A$ (see Figure~\ref{fig 2}), a simple inclusion-exclusion principle and the fact that $\arccos(x) = \uppi/2 - x - o_x(x)$ as $x\to 0$, we obtain that
\begin{align}
|B(y,C\zeta)\triangle B(z,C\zeta)|&=2 \bigg(\dfrac{2\uppi - 2\varphi}{2\uppi}\uppi (C\zeta)^2 - \dfrac{2\varphi}{2\uppi} \uppi (C\zeta)^2 + 2 \, \dfrac{\ell\cdot C\zeta \sin\varphi}{2}\bigg)\nonumber\\ 
&=\left(2\uppi - 4\arccos\left(\dfrac{\ell}{2C\zeta}\right)\right) (C\zeta)^2 + 2\ell\cdot C\zeta \sqrt{1 - \dfrac{\ell^2}{4(C\zeta)^2}}\nonumber\\
&=(1+o_{A}(1))\frac{4\ell}{2C\zeta} (C\zeta)^2 = (1+o_A(1)) \frac{C\zeta^2}{\uppi A}.\label{eq:compute}
\end{align}

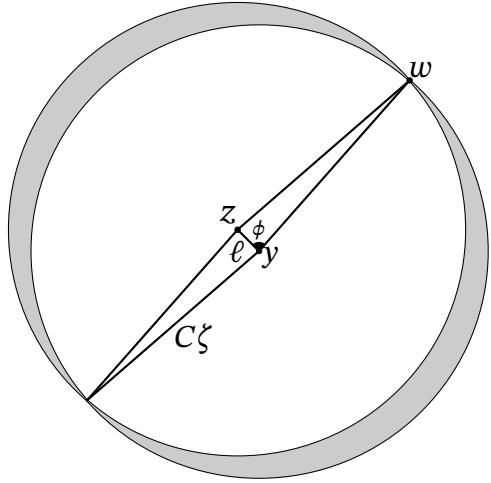
\begin{figure}
\centering
\begin{tikzpicture}[line cap=round,line join=round,x=1cm,y=1cm,scale=0.4,rotate=-45]
\clip(-0.5,-0.5) circle[radius=8.5];

\draw [shift={(0,-0.5)},line width=0.8pt,color=black,fill=black,fill opacity=0.10000000149011612] (0,0) -- (93.81407483429037:0.2703580656096896) arc (93.81407483429037:180:0.2703580656096896) -- cycle;

\draw [line width=0.8pt] (-1,-0.5) circle (7.516648189186455cm);
\draw [line width=0.8pt] (0,-0.5) circle (7.516648189186455cm);

\fill[gray!40,even odd rule] (0,-0.5) circle (7.516648189186455cm) (-1,-0.5) circle (7.516648189186455cm);

\draw [line width=0.8pt] (-1,-0.5)-- (0,-0.5);
\draw [line width=0.8pt] (-1,-0.5)-- (-0.5,7);
\draw [line width=0.8pt] (-0.5,7)-- (0,-0.5);
\draw [line width=0.8pt] (-1,-0.5)-- (-0.5,-8);
\draw [line width=0.8pt] (0,-0.5)-- (-0.5,-8);
\begin{scriptsize}
\draw [fill=black] (-1,-0.5) circle (2.5pt);
\draw[color=black] (-1.5623813401027786,-0.35512488171559686) node {\large{$z$}};
\draw [fill=black] (-0.5,7) circle (2.5pt);
\draw[color=black] (-0.5,7.53) node {\large{$w$}};
\draw [fill=black] (0,-0.5) circle (2.5pt);
\draw[color=black] (0.438268345408925,-0.3821606882765659) node {\large{$y$}};
\draw[color=black] (-0.5,-1) node {\large{$\ell$}};
\draw[color=black] (0.54087447323826635,-4.0590303805683545) node {\large{$C\zeta$}};
\draw[color=black] (-0.5,0) node {$\varphi$};
\end{scriptsize}
\end{tikzpicture}  
\caption{An illustration for the computation in~\eqref{eq:compute}.}
\label{fig 2}
\end{figure}
\noindent
Combining~\eqref{eq:R1y} and~\eqref{eq:compute}, we have that
\begin{align*}
|R_2(z,t_A)| - |R_1(z)| 
&= (|R_2(z,t_A)| -|R_2(y,t_A)|) + (|R_2(y,t_A)| - |R_1(y)|) + (|R_1(y)| - |R_1(z)|)\\
&=|R_2(y,t_A)| - |R_1(y)| + (1\pm o_A(1))|B(y,C\zeta)\triangle B(z,C\zeta)|\\
&= \bigg(\frac{2C^3}{3A} + o_A\bigg(\frac{C^3}{A}\bigg)\bigg)\zeta^2 > 0,
\end{align*}
which finishes the proof.
\end{proof}

We are now ready to prove Theorem~\ref{thm:C}.

\begin{proof}[Proof of \cref{thm:C}~\ref{thm:Citem2}]
First of all, note that, for every $C < 1/\uppi^{1/2}$ and $r\le C \zeta$, a.a.s.\ the random graph $G_2(n,r)$ is disconnected and the statement trivially holds (see~\cite[Theorem~13.2]{Pen03}).
Now, fix $C\geq1/\uppi^{1/2}$ and $r\in [\zeta/2\uppi^{1/2},C\zeta]$, let $A, t, \nu$ be as in Lemma~\ref{lem:geom}, let $\eps \coloneqq \min\{\eps(C,A,t,\nu), 1/20\}$, and fix any $x\in[Z_{t_A},1-Z_{t_A}]^2$. 
Let $G\sim G_2(n,r)$ and $V \coloneqq V(G)$.
Let $\cA_{x}$ be the event that, for every $v\in V\cap B(x,A\zeta)$, we have that
\begin{equation}\label{equa:eventA}
    |(V\setminus \{v\})\cap R_1(v,r)|\ge (1+5\eps)|V\cap R_2(v,t,r)|,
\end{equation}
and let $\cB_{x}$ be the event that, for every $v\in V\setminus B(x,(A+t)\zeta)$, we have that
\begin{equation}\label{equa:eventB}
    |(V\setminus \{v\})\cap R_2(v,t,r)|\ge (1+5\eps)|V\cap R_1(v,r)|.
\end{equation}
We are going to show that
\begin{equation}\label{eq:claim}
    \text{$\cA_{x}\cap \cB_{x}$ holds a.a.s.\ conditionally on $\cE_{x,t}\cap \cF_{x,\nu}\cap \cG_{x,t}$.}
\end{equation}

Consider the probability space conditioned on the event $\cE_{x,t}\cap \cF_{x,\nu}\cap \cG_{x,t}$.
We write $\mathbb{P}'$ and~$\mathbb{E}'$ to refer to probabilities and expectations in this conditional space, respectively.
We begin verifying \eqref{equa:eventA} for vertices $v\in V\cap B(x,A\zeta)$.
For vertices at distance more than $C\zeta$ from $\partial B(x,A\zeta)$, \eqref{equa:eventA} holds trivially.
Moreover, for every vertex $v\in V\cap B(x, A\zeta)$ at distance at most $C\zeta$ from $\partial B(x,A\zeta)$, Lemma~\ref{lem:geom} implies that $(1+\nu)|R_1(v,r)|\ge (1+10\eps)|R_2(v,t,r)|$.
Furthermore, by Chernoff's bound,
\[\mathbb P'[|(V\setminus \{v\})\cap R_1(v,r)|\le (1-\eps)\mathbb E'[|(V\setminus \{v\})\cap R_1(v,r)|]] = n^{-\Omega(1)},\]
and
\[\mathbb P'[|V\cap R_2(v,t,r)|\ge (1+\eps)\mathbb E'[|V\cap R_2(v,t,r)|]] = n^{-\Omega(1)}.\]
As a result, conditionally on $\cE_{x,t}\cap \cF_{x,\nu}\cap \cG_{x,t}$, with probability $1 - n^{-\Omega(1)}$, we have that
\begin{align*}
|(V\setminus \{v\})\cap R_1(v,r)|
&\ge (1-\eps)(1+\nu)|R_1(v,r)|(n-1)\ge (1-\eps)(1+9\eps)|R_2(v,t,r)|n\\*
&\ge (1\pm o(1)) (1-\eps)(1+9\eps)\frac{|V\cap R_2(v,t,r)|}{1+\eps}\ge (1+5\eps)|V\cap R_2(v,t,r)|,  
\end{align*}
where the third inequality uses that $\mathbb E[|V\cap R_2(v,t,r)|] = (1\pm o(1)) |R_2(v,t,r)|n$ thanks to $\cG_{x,t}$.

Next, we consider \eqref{equa:eventB}.
In the same way as above, if $u\in V\setminus B(x,((1+t)A+C)\zeta)$, then \eqref{equa:eventB} holds trivially.
For every vertex $u\in V\cap B(x,(1+t)A\zeta,((1+t)A+C)\zeta)$, Lemma~\ref{lem:geom} implies that $|R_2(u,t,r)|\ge (1+10\eps)(1+\nu)|R_1(u,r)|$.
Moreover, by Chernoff's bound,
\[\mathbb P'[|V\cap R_1(u,r)|\ge (1+\eps)\mathbb E'[|V\cap R_1(u,r)|]] = n^{-\Omega(1)},\]
and
\[\mathbb P'[|(V\setminus \{u\})\cap R_2(u,t,r)|\le (1-\eps)\mathbb E'[|(V\setminus \{u\})\cap R_2(u,t,r)|]] = n^{-\Omega(1)}.\]
Again, conditionally on $\cE_{x,t}\cap \cF_{x,\nu}\cap \cG_{x,t}$, with probability $1 - n^{-\Omega(1)}$, we have that
\begin{align*}
|(V\setminus \{v\})\cap R_2(u,t,r)|
&\ge(1\pm o(1))(1-\eps)|R_2(u,t,r)|(n-1)\\
&\ge (1\pm o(1))(1-\eps)(1+9\eps)(1+\nu)|R_1(u,r)|n\\
&\ge (1\pm o(1))(1-\eps)(1+9\eps)\frac{|V\cap R_1(u,r)|}{1+\eps}\ge (1+5\eps)|V\cap R_1(u,r)|.    
\end{align*}

Then, using a union bound over the vertices in $V\cap B(x,(A-C)\zeta,A\zeta)$ for the event $\overline{\cA}_{x}$ (which are at most $(\log n)^2$ by $\cF_{x,\nu}$) and over the vertices in $V\cap B(x,(1+t)A\zeta, ((1+t)A+C)\zeta)$ for the event~$\overline{\cB}_{x}$ (which are at most $(\log n)^2$ by $\cG_{x,t}$), we obtain that
\begin{align*}
\mathbb P[\cA_{x}\cap \cB_{x}\mid \cE_{x,t}\cap \cF_{x,\nu}\cap \cG_{x,t}] 
&\ge 1 - \mathbb P[\overline{\cA}_{x}\mid \cE_{x,t}\cap \cF_{x,\nu}\cap \cG_{x,t}] - \mathbb P[\overline{\cB}_{x}\mid \cE_{x,t}\cap \cF_{x,\nu}\cap \cG_{x,t}]\\*
&\ge 1 - (\log n)^2 n^{-\Omega(1)} = 1-o(1),
\end{align*}
so \eqref{eq:claim} holds.

It remains to recall that, by Lemma~\ref{lem:AC}, a.a.s.\ there exists a point $x\in [Z_{t_A},1-Z_{t_A}]^2$ for which $\cE_{x,t}\cap \cF_{x,\nu}\cap \cG_{x,t}$ holds.
Fixing such a point $x$, let $H\subseteq G$ be obtained from $G$ by removing all edges with one endpoint in $B(x,A\zeta)$ and the other outside $B(x,(1+t)A\zeta)$.
Clearly, $H$ is disconnected, and it follows from \eqref{eq:claim} that a.a.s.\ $H$ is a $(1/2+\eps)$-subgraph of $G$, which finishes the proof.
\end{proof}

%%%%%%%%%%%%%%%%%%%%%%%%%%%%%%%%%%%%%%
%%%%%%%%%%%%%%%%%%%%%%%%%%%%%%%%%%%%%%

\section{Long cycles}\label{sec:cycles}

This section is dedicated to the proof of the following technical result, which readily implies \cref{thm:long cycles intro}. 

\begin{theorem}\label{thm:long cycles}
For every integer $d\ge 1$, every $\eta\in [1/2, 1)$ and every $\eps\in (0, 1-\eta]$, there exist positive constants $c_3 = c_3(d, \eps)$ and $C_3 = C_3(d, \eps)$ such that, if $r\in [C_3 (\log n/n)^{1/d}, c_3]$, then a.a.s.\ $G_d(n,r)$ satisfies the following property:
for every integer $L\in [c_3^{-1}, 2\eta n/(1+\eta)]$ and every vertex $v\in V$, every $(\eta+\eps)$-subgraph of $G_d(n,r)$ contains a cycle of length $L$ containing $v$.
\end{theorem}

Observe that, compared with \cref{thm:long cycles intro}, \cref{thm:long cycles} can be used to obtain longer cycles provided that the subgraphs of $G_d(n,r)$ have higher degrees, and that these cycles become almost spanning as $\eta$ approaches $1$.
In fact, the same proof can be used to obtain cycles of lengths up to $(2\eta/(1+\eta)+\eps')n$ for a sufficiently small $\eps'$ depending on $\eps$.
As this small improvement does not yield a more interesting result, we omit the details for the clarity of our exposition.

\begin{remark}\label{rem:long cycles}
Under the conditions of \cref{thm:long cycles}, for every vertex $v$, we can also prove the existence of short even cycles going through $v$.
In fact, in the proof of Theorem~\ref{thm:long cycles}, we also construct a cycle of length $L$ through each vertex for all even $L\in [4, 2\eta n/(1+\eta)]$.
However, short odd cycles not only seem more complicated to construct in general, but in some cases the adversary has the power to destroy all such cycles containing a given vertex.
As an example, the adversary may delete all triangles containing a given vertex if $\eta=1/2$ and $\eps$ is sufficiently small.

Indeed, fix a vertex $v$ (at distance at least $2r$ from the boundary of $[0,1]^d$) and define the region $R \coloneqq B(v,r)\setminus B(v,t)$, where $t<r$ is chosen so that $|R| = (1/2+2\eps)|B(v,r)|$.\COMMENT{So we need to take $t=(1/2-2\eps)^{1/d}r$ (notice that this tends to $r$ as $d$ grows).}
Then, by choosing~$\eps$ sufficiently small, one can show that, for every point $x\in R$, $|B(x,r)\cap R|\leq(1/2-2\eps)|B(x,r)|$.\COMMENT{Intuitively, for fixed $\eta$ and $\eps$, the worst case must be $d=1$ (as $d$ increases, $B(x,r)\cap R$ seems to get smaller and smaller, since $t$ tends to $r$; in particular, for all $d\geq2$, if $\eps$ is sufficiently small, then $t\geq r/2$). For the case $d=1$, $R$ is the union of two disjoint segments, $R_1$ and $R_2$. Say $x\in R_1$. Observe that $x$ is at distance at least $(1/2-2\eps)r$ from $v$, and thus $|B(x,r)\cap R_2|\leq4\eps r$. Moreover, $R_1\subseteq B(x,r)$. It follows that $|B(x,r)\cap R|\leq(1/2+6\eps)r$. If $\eps\leq1/20$, it follows that $|B(x,r)\cap R|\leq(1/2-2\eps)|B(x,r)|$, as desired.}
Thus, it is not hard to check using Chernoff's bound that, when $r\ge C (\log n/n)^{1/d}$ for a suitably large constant $C>0$, a.a.s.\ deleting all edges between $v$ and $V\cap B(v,t)$ as well as all edges whose endpoints lie in $V\cap R$ leaves us with a $(1/2+\eps)$-subgraph of $G_d(n,r)$ with no triangle containing $v$.
\end{remark}

Let us begin with a sketch of the proof of \cref{thm:long cycles}.
For this sketch, we consider only Theorem~\ref{thm:long cycles intro} (that is, the case $\eta=1/2$), since no new ideas are required to prove the more general version.
The main part of the proof is to show that we can construct cycles of even length $L\in [4, 2n/3]$.
To achieve this, we first tessellate $[0,1]^d$ with cells of diameter $\delta r$, for some small constant $\delta > 0$, and consider a $(1/2+\eps)$-subgraph $H$ of $G_d(n,r)$ and a fixed vertex $v$. 
Then, we consider a colouring of~$V(H)$ in three colours so that $v$ is blue, roughly $2/3$ of the vertices are red, roughly $1/3$ of them are blue, and a very small proportion are green.
The idea is to build a cycle of length $L$ in $H$ by starting from $v$ in such a way that
\begin{enumerate}[label=(\roman*)]
    \item every second vertex is blue;
    \item many blue vertices contained in a single cell appear in a short segment of the cycle;
    \item red vertices are inserted between blue vertices which lie in the same cell, and
    \item green vertices ensure the transitions between blue vertices in different cells.  
\end{enumerate}
To ensure better control on the places where red and green vertices are inserted, we associate each red (resp.\ green) vertex to a unique cell (resp.\ pair of cells) within distance $r-\delta r$ from it, and only use it to connect blue vertices in this cell (resp.\ pair of cells).
We prove the existence of a colouring and an assignment of cells to vertices satisfying the properties that we need to construct our cycles by considering random ones and showing that our desired properties hold with positive probability.
The above associations to cells (for the red vertices) or pairs of cells (for the green vertices) allow us to avoid repetitions of red and green vertices along the cycles we aim to construct.
For each cell~$q$, we ensure the existence of long red-blue paths containing as many blue vertices in $q$ as we need by using a theorem of \citet{Jac81} which provides long paths in bipartite graphs (see \cref{thm:Jackson}).
Gluing the ends of such a path (if $L$ is sufficiently small) or gluing several paths (if $L$ is large) proves that we can construct the desired cycle of even length~$L$.
Extending the result to sufficiently long odd cycles is done by replacing a path of length two in an already constructed cycle of even length by a suitably long path of odd length.

Before we proceed with the formal proof of Theorem~\ref{thm:long cycles}, we state the theorem of \citet{Jac81} which will be important for our proof.

\begin{lemma}[{\cite[Theorem~1]{Jac81}}]\label{thm:Jackson}
Let\/ $k\in\mathbb{N}$.
If\/ $G=(A,B,E)$ is a bipartite graph where\/ $|A|\in[2,k]$,\/ $|B|\in [k,2k-2]$, and every vertex in\/ $A$ has degree at least\/ $k$, then $G$ contains a cycle of length\/ $2|A|$.
\end{lemma}

We can use Chernoff's inequality to obtain several properties about the distribution of a random set of points over $[0,1]^d$.
These properties will be useful for the proof of \cref{thm:long cycles}.

\begin{lemma}\label{lem:basic properties}
    Fix $0<1/C\lll\delta\lll\eps,1/d\leq1$ and let $r\in [C(\log n/n)^{1/d},\sqrt{d}]$.
    Let $V$ be a set of $n$ points sampled independently and uniformly from $[0,1]^d$.
    Then, a.a.s.\ the following properties hold:
    \begin{enumerate}[label=$(\mathrm{\alph*})$]
        \item\label{lem:basic propertiesitem1} If we tessellate $[0,1]^d$ with cells of side length $s\geq\delta r/2d$, then each cell contains $(1\pm\eps/100)s^dn$ points of\/ $V$, and there are no points on the boundary of any cell.
        \item\label{lem:basic propertiesitem2} For every point $v\in V$, we have that
        \begin{enumerate}[label=$(\mathrm{b.\arabic*})$]
            \item\label{lem:basic propertiesitem2.1} $B(v,r)$ contains at most $(1+\eps/50)|B(v,r)\cap[0,1]^d|n$ points of\/ $V$;
            \item\label{lem:basic propertiesitem2.2} $B(v,(1-\delta)r)$ contains at least $(1-\eps/20)|B(v,r)\cap[0,1]^d|n$ points of\/ $V$ other than $v$, and
            \item\label{lem:basic propertiesitem2.3} the annulus $B(v,r)\setminus B(v,(1-\delta)r)$ contains at most $\eps\theta_dr^dn/20$ points of\/ $V$.
        \end{enumerate}
        \item\label{lem:basic propertiesitem3} For every pair of points $u,v\in V$ such that $\lVert u-v\rVert\leq 2\delta r$, $B(u,(1-2\delta)r)\cap B(v,(1-2\delta)r)\cap[0,1]^d$ contains at least $(1-\eps/20)|B(v,r)\cap[0,1]^d|n$ points of\/ $V$ other than $u$ and $v$.
    \end{enumerate}
\end{lemma}

\begin{proof}
    Property~\ref{lem:basic propertiesitem1} follows from a direct application of Chernoff's bound (\cref{lem:chernoff}) together with a union bound over all cells\COMMENT{To have enough concentration, we require that $\delta r/2d$ is sufficiently large, which holds by having $1/C\lll\delta,1/d$.} and from the fact that the boundaries of all cells form a set of measure~$0$.
    So let us focus on~\ref{lem:basic propertiesitem2} and~\ref{lem:basic propertiesitem3}.
    
    For any two points $u,v\in[0,1]^d$, define the regions
    \begin{align*}
    S_1 = S_1(v) &\coloneqq [0,1]^d\cap B(v,(1-\delta)r,(1+\delta)r),\\
    S_2 = S_2(v) &\coloneqq [0,1]^d\cap B(v,r),\\
    S_3 = S_3(u,v) &\coloneqq [0,1]^d\cap B(u,(1-2\delta)r)\cap B(v,(1-2\delta)r).
    \end{align*}
Note that
\begin{equation}\label{equa:basic34}
    2^{-d}\delta d\theta_dr^d\le 2^{-d} ((1+\delta)^d - (1-\delta)^d) |B(v,r)|\le |S_1|\leq ((1+\delta)^d - (1-\delta)^d)|B(v,r)|\leq3\delta d \theta_d r^d.
\end{equation}
Now, suppose that $\lVert u-v\rVert\leq2\delta r$.
Then, similarly as above,
\begin{equation}\label{equa:basic99}
    |S_2\setminus S_3|\leq |B(v,(1-4\delta)r,r)|\le (1 - (1-4\delta)^d)|B(v,r)|\leq 5\delta d \theta_d r^d.
\end{equation}
Moreover, since $S_3\subseteq S_2$, \COMMENT{This uses also the fact that $\delta$ is small enough.}
\begin{equation}\label{equa:basic11}
4^{-d} \theta_d r^d\le |[0,1]^d\cap B(u, r/2)|\le |S_3|\le |S_2|\le \theta_d r^d.
\end{equation}
Combining \eqref{equa:basic34}, \eqref{equa:basic99} and \eqref{equa:basic11} with the choice of $\delta$, we conclude that\COMMENT{Combine the upper bound in \eqref{equa:basic34} with the lower bound on \eqref{equa:basic11} to obtain that $|S_1|\leq3\delta d4^d|S_3|$. The claimed bound follows by choosing $\delta$ sufficiently small. The next bound is now trivial since $S_3\subseteq S_2$.}
\begin{equation}\label{equa:basic33}
|S_1|\leq\eps|S_3|/50\leq\eps|S_2|/50
\end{equation}
and\COMMENT{Simply use that $|S_2|=|S_2\setminus S_3|+|S_3|$, the upper bound in \eqref{equa:basic99} and the lower bound in \eqref{equa:basic11} to obtain that $|S_2|\leq(1+5\delta d4^d)|S_3|$. Together (again) with a sufficiently small choice of $\delta$, we conclude what we want.}
\begin{equation}\label{equa:basic33bis}
    |S_2|\leq (1+\eps/30)|S_3|.
\end{equation}
    
    For~\ref{lem:basic propertiesitem2}, consider a fixed vertex $v\in V$ and reveal its position.
    By the lower bounds in \eqref{equa:basic34} and \eqref{equa:basic11} and the choice of $C$,\COMMENT{Simply choose $C$ large enough to guarantee that both $S_1$ and $S_2$ are large enough to have concentration (so we need the lower bounds on $|S_1|,|S_2|$).} direct applications of Chernoff's bound show that, with probability $1-o(1/n)$, we have that
    \begin{equation}\label{equa:basic1}
        \lvert|V\cap S_2|-|S_2|n\rvert\leq\eps|S_2|n/50
    \end{equation}
    and
    \begin{equation}\label{equa:basic2}
        |V\cap S_1|\leq3|S_1|n/2.
    \end{equation}
    We claim that, conditionally on these two events holding, it follows that \ref{lem:basic propertiesitem2} holds for $v$.
    Indeed, \ref{lem:basic propertiesitem2.1} follows immediately from \eqref{equa:basic1} and, since $B(v,r)\setminus B(v,(1-\delta)r)\subseteq S_1$, \ref{lem:basic propertiesitem2.3} follows immediately from \eqref{equa:basic2}, the upper bound in \eqref{equa:basic34} and the choice of $\delta$.\COMMENT{The number of vertices contained in $B(v,r)\setminus B(v,(1-\delta)r)$ is at most $3|S_1|n/2\leq9\delta d\theta_dr^dn/2\leq\eps\theta_dr^dn/20$, where the last inequality uses the choice of $\delta$.}
    Lastly, \ref{lem:basic propertiesitem2.2} can be derived from \eqref{equa:basic1} and \eqref{equa:basic2} in conjunction with \eqref{equa:basic33}.\COMMENT{By \eqref{equa:basic1}, $S_2$ contains at least $(1-\eps/50)|S_2|n$ vertices. 
    By \eqref{equa:basic2}, $S_1$ contains at most $3|S_1|n/2$ vertices.
    It follows that $B(v,(1-\delta)r)$ contains at least $(1-\eps/50)|S_2|n-3|S_1|n/2\geq(1-\eps/50)|S_2|n-3\eps|S_2|n/100=(1-\eps/20)|S_2|n$, where the inequality holds by \eqref{equa:basic33}.
    The extra vertex that we need to take care of can be hidden by making the constants smaller.}
    Therefore, \ref{lem:basic propertiesitem2} follows by a union bound over all vertices of the complementary events.

    Now we focus on \ref{lem:basic propertiesitem3}.
    Fix two arbitrary vertices $u,v\in V$ and reveal their positions.
    If $\lVert u-v\rVert> 2\delta r$, there is nothing to prove, so we may assume that $\lVert u-v\rVert\leq 2\delta r$.
    Consider the random variable $X\coloneqq|S_3\cap(V(G)\setminus\{u,v\})|$ with distribution $\mathrm{Bin}(n-2, |S_3|)$.
    By \eqref{equa:basic11}\COMMENT{This is here to give a lower bound on $|S_3|$.} and the choice of $C$, an application of \cref{lem:chernoff} implies that, with probability $1-o(n^{-2})$, we have 
    \[X\geq(1-\eps/60)|S_3|n\geq\frac{1-\eps/60}{1+\eps/30}|S_2|n\geq(1-\eps/20)|S_2|n,\]
    where the second inequality follows from \eqref{equa:basic33bis}.
    Therefore, a union bound over the complementary events for all $O(n^2)$ pairs of vertices leads to the desired conclusion.
\end{proof}

The proof of \cref{thm:long cycles} relies on several technical lemmas.
We first describe the precise framework in which these lemmas are phrased.
We will assume that $d\in\mathbb{N}$, $\eta\in[1/2,1)$ and $\eps\in(0,1-\eta]$ are defined globally.
Let $0<1/C_3\lll\delta,c_3\lll\eps,1/d\leq1$ (we will reiterate the necessary relations in each lemma).
Consider a tessellation of the hypercube $[0,1]^d$ with cells of side length $s\coloneqq\lceil d/\delta r\rceil^{-1}$ (so, in particular, each cell has diameter at most $\delta r$).\COMMENT{Each cell has diameter at most $\delta r/\sqrt{d}$.}
We denote the set of cells by~$\cQ$ and recall the auxiliary graph~$\Gamma$ with vertex set~$\cQ$ where two cells are joined by an edge whenever they share a common $(d-1)$-dimensional face. 
Next, we consider a set of~$n$ vertices and assign to each of them a uniformly random position on $[0,1]^d$, independently of each other.
This set will play the role of the vertex set of $G_d(n,r)$.
Note that \cref{lem:basic properties} applies with our choice of parameters. 
Then, we randomly assign to each vertex one of three colours: red, blue or green.
To be precise, we colour each vertex in blue with probability $\eta/(1+\eta)+\eps/40$, in red with probability $1/(1+\eta)-\eps/20$, and in green with probability $\eps/40$, independently from all other vertices (and independently of the position of the vertex in $[0,1]^d$).
Moreover, to each red vertex $u$, we associate a cell $q(u)\in\cQ$ by choosing it uniformly at random among all cells that are entirely contained in the ball $B(u,r)$.
Similarly, to each green vertex $z$, 
we associate a pair of cells $q_1, q_2\in\cQ$ with $q_1q_2\in E(\Gamma)$ by choosing uniformly at random one among all such pairs with $q_1\cup q_2\subseteq B(z,r)$.

Some parts of our construction become more technical if we consider vertices which are close to the boundary of $[0,1]^d$.
For this reason, in the coming lemmas, we will restrict most of our attention to vertices which are far from the boundary; we will deal with vertices close to the boundary later on.
Let $\cQ_{2r}\subseteq\cQ$ denote the subset of cells in~$\cQ$ which are at distance at least $2r$ from the boundary of~$[0,1]^d$, and let $Q_{2r}\coloneqq\bigcup_{q\in\cQ_{2r}}q$.
Our choice of~$\cQ_{2r}$ ensures that all neighbours of vertices in a cell of~$\cQ_{2r}$ are also at distance at least $r$ from the boundary; we will use this fact throughout.

First, we need some simple properties about how blue vertices are distributed over $[0,1]^d$ and over the neighbours of any vertex.
Both results follow from direct applications of Chernoff's bound (\cref{lem:chernoff}), and so we omit their proofs.

\begin{lemma}\label{lem:colours}
    Let $0<1/C\lll\delta\lll\eps,1/d\leq1$ and suppose that $C(\log n/n)^{1/d}\leq r\leq\sqrt{d}$. Let $V$ be a set of $n$ vertices on $[0,1]^d$ satisfying the property described in \cref{lem:basic properties}~\ref{lem:basic propertiesitem1}.
    Then, a.a.s.\ every cell $q\in\cQ$ contains at least $(\eta/(1+\eta)+\eps/70)s^dn$ and at most $(\eta/(1+\eta)+\eps/3)s^dn$ blue vertices.\COMMENT{Fix a cell $q\in\cQ$.
Consider first the upper bound.
By \cref{lem:basic properties}~\ref{lem:basic propertiesitem1}, $q$ contains at most $(1+\eps/100)s^dn$ vertices.
Now colour them randomly, and let $X$ denote the number of blue vertices.
It follows that $\mathbb{E}[X]\leq(\eta/(1+\eta)+\eps/40)(1+\eps/100)s^dn\leq(\eta/(1+\eta)+\eps/30)s^dn$ (where the second inequality follows by using e.g. the fact that $\eta/(1+\eta)\leq1/2$ and $\eps\leq1$).
Then, making $C$ sufficiently large, \cref{lem:chernoff} gives that what we want holds with probability $1-o(1/n)$, and so the conclusion follows by a union bound.\\
Now consider the lower bound, which is similar.
By \cref{lem:basic properties}~\ref{lem:basic propertiesitem1}, $q$ contains at least $(1-\eps/100)s^dn$ vertices, so $\mathbb{E}[X]\geq(\eta/(1+\eta)+\eps/40)(1-\eps/100)s^dn\geq(\eta/(1+\eta)+\eps/60)s^dn$ (where the second inequality follows again by using that $\eta/(1+\eta)\leq1/2$ and $\eps\leq1$).
Then, making $C$ sufficiently large, \cref{lem:chernoff} gives that what we want holds with probability $1-o(1/n)$, and so the conclusion follows by a union bound.}  
\end{lemma}

\begin{lemma}\label{lem:blue neighbours}
    Let $G$ be an $n$-vertex graph.
    Colour $V(G)$ in such a way that each vertex receives colour blue with probability at least $1/3$, independently of every other vertex.
    Then, a.a.s.\ every vertex $v\in V(G)$ with $d_G(v)\geq 40\log n$ has at least $d_G(v)/6\ge 1$ blue neighbours.\COMMENT{Fix any vertex with the right degree. The expected number of blue neighbours is at least $6\log n$, and thus, by \cref{lem:chernoff}, the probability that it has fewer than $\log n$ blue neighbours is bounded from above by
    \[2\nume^{-25\log n/18}=o(1/n).\]
    The conclusion follows by a union bound over the at most $n$ vertices which satisfy the degree condition.}
\end{lemma}

Our next lemma shows that, after colouring the vertices, we can find large sets of blue vertices contained in a small number of cells forming a path in $\Gamma$.
Moreover, we can find such sets even if a few of the vertices are ``forbidden''.
We will later show that we can construct cycles containing precisely the sets of blue vertices given by this result.

\begin{lemma}\label{lem:cycle_cubes}
Let $0<1/C,c\lll\eps,1/d\leq1$.
Suppose $r\in [C(\log n/n)^{1/d},c]$, and let $V$ be a set of~$n$ vertices in $[0,1]^d$ satisfying the property described in \cref{lem:basic properties}~\ref{lem:basic propertiesitem1}.
Then, a.a.s.\ the following event holds:
for every blue vertex $v\in V\cap Q_{2r}$, every set $A\subseteq (V\cap Q_{2r})\setminus\{v\}$ such that $|A\cap q|\leq2$ for all $q\in\cQ_{2r}$, and every integer $\ell\in[3, \eta n/(1+\eta)-|A|]$,  there exist
\begin{itemize}
    \item a positive integer $m\leq (\ell-1)/2$,
    \item distinct cells $q_1, \ldots, q_m\in \cQ_{2r}$, and
    \item for each $j\in[m]$, a set of blue vertices $S_j\subseteq (V\cap q_j)\setminus A$
\end{itemize}
such that $v\in S\coloneqq S_1\cupdot\ldots\cupdot S_m$ and $|S|=\ell$, and
\begin{enumerate}[label=$(\mathrm{\alph*})$]
    \item\label{lem:cycle_cubesitem1} for each $j\in [m]$, $|S_j\setminus\{v\}|\geq2$ and $S_j\subseteq q_j$, and
    \item\label{lem:cycle_cubesitem2} for each $j\in [m-1]$, $q_jq_{j+1}\in E(\Gamma)$.
\end{enumerate}
\end{lemma}

\begin{proof}
Let $0<1/C\lll\delta\lll\eps,1/d\leq1$ and $0<c\lll\eps,1/d\leq1$ so that $|[0,1]^d\setminus Q_{2r}|\leq\eps/100$.
Thus, by \cref{lem:colours}, a.a.s.\ $Q_{2r}$ contains at least $\eta n/(1+\eta)$ blue vertices\COMMENT{The number of cells which lie in $Q_{2r}$ is at least $(1-\eps/100)s^{-d}$, and each cell contains at least $(\eta/(1+\eta)+\eps/70)s^dn$ blue vertices by \cref{lem:colours}, so the total number of blue vertices is at least $(1-\eps/100)(\eta/(1+\eta)+\eps/70)n\geq\eta n/(1+\eta)$.} and each cell in $\cQ_{2r}$ contains at least six blue vertices.
We condition on these events.

Now fix a blue vertex $v$ in some cell $q\in\cQ_{2r}$.
As the auxiliary graph $\Gamma_{2r}\coloneqq\Gamma[\cQ_{2r}]$ contains a Hamilton path,\COMMENT{This graph is a $d$-dimensional grid, and the claim has a very simple proof by induction on the dimension.} for each $m\in [|\cQ_{2r}|]$ it must contain a path of length $m-1$ which itself contains~$q$.
Now, let $m$ be the smallest integer such that $\Gamma_{2r}$ contains a path $P$ of length $m-1$ with $q\in V(P)$ and such that the cells in $V(P)$ contain a total of at least $\ell$ blue vertices not contained in~$A$ (note that, as $Q_{2r}$ contains at least $\eta n/(1+\eta)$ blue vertices, $m$ is well defined).
In particular, since all cells in~$\cQ_{2r}$ contain at least four blue vertices outside $A$, $m\leq\lceil\ell/4\rceil\leq(\ell-1)/2$.
Let $V(P)=\{q_1,\ldots,q_m\}$, where the labelling is given by the order in which the cells appear on the path so that \ref{lem:cycle_cubesitem2} holds.
Now, let $S\subseteq (V\setminus A)\cap\bigcup_{j=1}^mq_j$ be such that $|S|=\ell$, $v\in S$ and $|(S\cap q_j)\setminus\{v\}|\geq2$ for all $j\in[m]$ (the existence of such a set is guaranteed by the minimality of $m\leq(\ell-1)/2$ in the definition of $P$ and the fact that each cell contains at least four vertices outside $A$).\COMMENT{If we did not take $m$ to be minimal, we might have taken e.g. a path which is so long that we can only take one vertex from each cell.}
In order to complete the proof, for each $j\in[m]$, let $S_j\coloneqq S\cap q_j$ so that the partition $S=\bigcupdot_{j=1}^m S_j$ satisfies~\ref{lem:cycle_cubesitem1}.
\end{proof}

Our next lemma gives us further information on the distribution of red and green vertices and guarantees that they can be used to string together long paths.

\begin{lemma}\label{lem:bipartite}
Let $0<1/C\lll \delta\lll\eps,1/d\leq1$.
Suppose $r\in [C(\log n/n)^{1/d}, \sqrt{d}]$,
and let $V$ be a set of vertices on $[0,1]^d$ satisfying the properties described in \cref{lem:basic properties}.
Consider the geometric graph $G=G(V,r)$ and let $H\subseteq G$ be an $(\eta+\eps)$-subgraph.
Then, a.a.s.\ the following properties hold:
\begin{enumerate}[label=$(\mathrm{\alph*})$]
    \item\label{lem:bipartiteitem1} For every cell $q\in \cQ_{2r}$ and every vertex $v\in V\cap q$, there are at least $(\eta/(\eta+1)+\eps/3)s^d n$ red $H$-neighbours $u$ of $v$ satisfying $q(u) = q$.
    \item\label{lem:bipartiteitem2} For every cell $q\in \cQ_{2r}$, there are at most $(1/(\eta+1) + \eps/20) s^d n$ red vertices $u$ such that $q(u) = q$.
    \item\label{lem:bipartiteitem3} For every pair of cells $q_1,q_2\in\cQ$ with $q_1q_2\in E(\Gamma)$, each pair of vertices $v_1,v_2\in V\cap(q_1\cup q_2)$ have at least five common green $H$-neighbours associated to $q_1, q_2$.
\end{enumerate}
\end{lemma}

\begin{proof}
For every point $x\in[r,1-r]^d$, the number of cells $q\in\cQ$ with $q\subseteq B(x,r)$ lies in the interval $[(1-\eps/20)\theta_dr^d/s^d, \theta_dr^d/s^d]$.\COMMENT{The upper bound is trivial, as we are simply using the fact that each cell has volume $s^d$.
For the lower bound, simply note that, by making $\delta$ arbitrarily small and the balls arbitrarily large with respect to $\delta$, we can make it so that the cells that intersect the boundary of the ball have an arbitrarily small proportion of the volume of the ball.}
Hence, for any vertex $u\in V\cap[r, 1-r]^d$ and any cell $q\in\cQ$ with $q\subseteq B(u,r)$, we have that
\begin{equation}\label{eq:binomial reds2}
    \mathbb{P}[u\text{ is red and }q(u) = q]\in[p_{\min},p_{\max}],
\end{equation}
where
\begin{equation}\label{eq:binomial reds}
p_{\min} \coloneqq \left(\frac{1}{1+\eta} - \frac{\eps}{20}\right)\frac{s^d}{\theta_dr^d}\qquad\text{ and }\qquad p_{\max} \coloneqq \left(\frac{1}{1+\eta} - \frac{\eps}{20}\right)\frac{s^d}{(1-\eps/20)\theta_dr^d}.
\end{equation}

We begin by proving \ref{lem:bipartiteitem1}.
Fix a vertex $v\in V\cap Q_{2r}$ and let $q\in\cQ_{2r}$ be the cell containing~$v$.
By \cref{lem:basic properties}~\ref{lem:basic propertiesitem2.2} and~\ref{lem:basic propertiesitem2.3},
\begin{align}
|N_H(v)\cap B(v,(1-\delta)r)|
&\geq (\eta+\eps)|N_G(v)\cap B(v, r)|-|N_G(v)\cap (B(v, r)\setminus B(v,(1-\delta)r))|\nonumber\\
&\geq (\eta+\eps) \left(1 - \frac{\eps}{20}\right)\theta_dr^dn - \frac{\eps}{20}\theta_dr^dn \geq (\eta + 0.9\eps)\theta_dr^dn.\label{eq:binomial reds3}
\end{align}
Now, let $X$ denote the number of red $H$-neighbours~$u$ of~$v$ with $q(u) = q$.
Since the diameter of $q$ is at most $\delta r$, if $u\in B(v,(1-\delta)r)$, then $q\subseteq B(u,r)$.
Thus, by \eqref{eq:binomial reds2} and \eqref{eq:binomial reds3}, $X$ dominates a binomial random variable $\mathrm{Bin}((\eta+0.9\eps)\theta_dr^dn,p_{\min})$.\COMMENT{This claim also uses the fact that the choices are made independently.}
Consequently, Chernoff's bound (\cref{lem:chernoff}) and \eqref{eq:binomial reds} imply that, with probability $1-o(n^{-1})$,\COMMENT{The inequality uses the fact that
\[\left(1-\frac{\eps}{20}\right)\cdot (\eta+0.9\eps)\cdot \left(\frac{1}{1+\eta} - \frac{\eps}{20}\right)\ge (\eta+0.9\eps)\cdot \left(\frac{1}{1+\eta} - \frac{\eps}{10}\right)\ge \frac{\eta}{1+\eta} + \frac{0.9\eps}{2} - \frac{\eps}{10} \ge \frac{\eta}{1+\eta} + \frac{\eps}{3}\]
(where the second inequality uses that $\eta+0.9\eps\leq1$).}
\[X\geq\left(1-\frac{\eps}{20}\right)\cdot (\eta+0.9\eps)\theta_dr^dn\cdot p_{\min}\ge \left(\frac{\eta}{1+\eta}+\frac{\eps}{3}\right) s^dn.\]
A union bound implies that a.a.s.\ the above holds for every vertex $v\in V\cap Q_{2r}$, thus proving~\ref{lem:bipartiteitem1}.

We now turn to~\ref{lem:bipartiteitem2}.
Fix $q\in\cQ_{2r}$ and a point $x$ in $q$.
Observe that, in order for a red vertex $u$ to satisfy $q(u)=q$, the distance between~$u$ and~$x$ must be at most~$r$.
By \eqref{eq:binomial reds2} and \cref{lem:basic properties}~\ref{lem:basic propertiesitem2.1}, the number of red vertices~$u$ at distance at most~$r$ from~$x$ which also satisfy $q(u) = q$ is dominated by a binomial random variable $\mathrm{Bin}((1+\eps/50)\theta_dr^dn,p_{\max})$.
Thus, by combining Chernoff's bound and~\eqref{eq:binomial reds} we get that, with probability $1-o(n^{-1})$, the number of red vertices~$u$ with $q(u) = q$ is at most\COMMENT{For the second inequality we used the fact that $t\in [1/2, \infty)\mapsto (t+\eps/20)/(t-\eps/20)$ is a decreasing function and $1/2\le 1/(1+\eta)\le 1$.}
\[\left(1+\frac{\eps}{50}\right)\cdot\left(1+\frac{\eps}{50}\right)\theta_dr^dn\cdot p_{\max}\le \frac{1+\eps/20}{1-\eps/20}\left(\frac{1}{1+\eta} - \frac{\eps}{20}\right) s^d n\le \left(\frac{1}{1+\eta}+\frac{\eps}{20}\right) s^d n.\]
A union bound over all $O(n)$ cells in $\cQ_{2r}$ finishes the proof of~\ref{lem:bipartiteitem2}.

Lastly, we concentrate on~\ref{lem:bipartiteitem3}. 
For each point $x\in[0,1]^d$, the number of edges $q_1q_2\in E(\Gamma)$ with $q_1\cup q_2\subseteq B(x,r)$ is at most $2d\theta_dr^d/s^d$.\COMMENT{We in fact have that this number lies in $[(1-\eps/20)\theta_dr^d/2^{d+1}s^d, 2d\theta_dr^d/s^d]$.
The upper bound is trivial, as we are simply using the fact that each cell has volume $s^d$ and that each cell shares a $(d-1)$-dimensional face with at most $2d$ other cells.
For the lower bound, simply note that, by making $\delta$ arbitrarily small and the balls arbitrarily large with respect to $\delta$, we can make it so that the cells that intersect the boundary of the ball have an arbitrarily small proportion of the volume of the ball. Moreover, every cell has at least one other cell to attach to inside the ball, the ball intersected with $[0,1]^d$ retains at least a $2^{-d}$ proportion of its volume, and each pair of cells is counted at most twice.}
Therefore, for any vertex $z\in V\cap[0,1]^d$ and any edge $q_1q_2\in E(\Gamma)$ with $q_1\cup q_2\subseteq B(z,r)$, we have that\COMMENT{For the last inequality we use the fact that $s\geq\delta r/2d$.}
\begin{equation}\label{eq:binomial green}
    \mathbb{P}[z\text{ is green and associated to }\{q_1,q_2\}]\geq\frac{\eps}{40}\frac{s^d}{2d\theta_dr^d}\geq\frac{\eps\delta^d}{5\cdot2^{d+4}d^{d+1}\theta_d}\eqqcolon p_*.
\end{equation}

Now, for any pair of distinct vertices $v,v'\in V$ and any graph $F\subseteq G$, let $q,q'\in\cQ$ be such that $v\in q$ and $v'\in q'$, and let\COMMENT{That is, $N^{\cap}_F(v,v')$ denotes the set of common $F$-neighbours $z$ of $v$ and $v'$ such that $q\cup q'$ is contained in the ball of radius $r$ centred at $z$. I think I prefer to have only the formula.}
\[N^{\cap}_F(v,v')\coloneqq\{z\in V\cap N_F(v)\cap N_F(v'):q\cup q'\subseteq B(z,r)\}.\]
For every edge $q_1q_2\in E(\Gamma)$, each pair of vertices $v_1,v_2\in V\cap(q_1\cup q_2)$ are at distance at most $2\delta r$ from each other.
Thus, by \cref{lem:basic properties}~\ref{lem:basic propertiesitem3}, $|N^{\cap}_G(v_1,v_2)|\geq(1-\eps/20)|B(v,r)\cap[0,1]^d|n$.
Moreover, by \cref{lem:basic properties}~\ref{lem:basic propertiesitem2.1}, every vertex $v\in V$ has degree at most $(1+\eps/50)|B(v,r)\cap[0,1]^d|n$.
Since $H\subseteq G$ is an $(\eta+\eps)$-subgraph of $G$ and $\eta\geq1/2$, it follows that\COMMENT{Note that $|N^{\cap}_{G}(v_1,v_2)|\geq(1-\eps/20)|B(v_1,r)\cap[0,1]^d|n$ and $|N^{\cap}_{G}(v_1,v_2)|\geq(1-\eps/20)|B(v_2,r)\cap[0,1]^d|n$ (by the symmetric application of \cref{lem:basic properties}~\ref{lem:basic propertiesitem3}), so $|N^{\cap}_{G}(v_1,v_2)|\geq(1-\eps/20)(|B(v_1,r)\cap[0,1]^d|+|B(v_2,r)\cap[0,1]^d|)n/2$.
Thus, we have
\begin{align*}
    |N^{\cap}_{H}(v_1,v_2)|&\geq |N^{\cap}_{G}(v_1,v_2)|-(1-\eta-\eps)(d_G(v_1)+d_G(v_2))\\
    &\geq(1-\eps/20)(|B(v_1,r)\cap[0,1]^d|+|B(v_2,r)\cap[0,1]^d|)n/2 \\
    &\quad- (1/2-\eps)(1+\eps/50)(|B(v_1,r)\cap[0,1]^d|+|B(v_2,r)\cap[0,1]^d|)n\\
    &\geq\eps(|B(v_1,r)\cap[0,1]^d|+|B(v_2,r)\cap[0,1]^d|)n/2\\
    &\geq\eps2^{-d}\theta_dr^dn,
\end{align*}
where in the third inequality we use that $(1-\eps/20)-2(1/2-\eps)(1+\eps/50)=\eps(2-1/20-1/25+\eps/25)\geq\eps$.}
\begin{equation}\label{eq:binomial green2}
    |N^{\cap}_{H}(v_1,v_2)|\geq |N^{\cap}_{G}(v_1,v_2)|-(1-\eta-\eps)(d_G(v_1)+d_G(v_2))\geq\eps2^{-d}\theta_dr^dn.
\end{equation}

Now, fix an edge $q_1q_2\in E(\Gamma)$ and a pair of vertices $v_1,v_2\in V\cap(q_1\cup q_2)$, and let $Z$ denote the number of green common $H$-neighbours~$z$ of~$v_1,v_2$ associated to $q_1,q_2$.
By \eqref{eq:binomial green} and \eqref{eq:binomial green2}, $Z$ dominates a binomial random variable $\mathrm{Bin}(\eps2^{-d}\theta_dr^dn,p_*)$.\COMMENT{This claim also uses the fact that the choices are made independently.}
Thus, Chernoff's bound implies that, with probability $1-o(n^{-2})$, $v_1$ and $v_2$ have at least five green common $H$-neighbours associated to~$q_1,q_2$.
A union bound over the $O(n^2)$ such pairs of vertices completes the proof.
\end{proof}

For each cell $q\in \cQ$, let $B_q$ denote the (random) set of blue vertices in $q$ and $R_q$ be the (random) set of red vertices $u$ such that $q(u) = q$.

\begin{corollary}\label{cor:bipartite}
Let $0<1/C\lll \delta\lll\eps,1/d\leq1$.
Suppose $r\in [C(\log n/n)^{1/d}, \sqrt{d}]$, and let $V$ be a set of vertices in $[0,1]^d$ satisfying all properties described in \cref{lem:basic properties}.
Consider the geometric graph $G=G(V,r)$ and let $H\subseteq G$ be an $(\eta+\eps)$-subgraph.
Then, a.a.s.\ the following holds: for every cell $q\in \cQ_{2r}$ and every set $B_q'\subseteq B_q$ with $|B_q'|\geq 2$, the bipartite graph $H[B_q', R_q]$ contains a cycle with $2|B_q'|$ vertices. 
\end{corollary}

\begin{proof}
Reveal the colours of each vertex in $V$ and, for each red vertex, reveal the cell it is associated to.
Condition on the events from \cref{lem:colours} and from \cref{lem:bipartite}~\ref{lem:bipartiteitem1} and~\ref{lem:bipartiteitem2}, which hold a.a.s.
We are going to verify that, conditionally on these events, for every cell $q\in \cQ_{2r}$ and every set $B_q'\subseteq B_q$ with $|B_q'|\geq2$, the graph $H[B_q',R_q]$ satisfies the assumptions of \cref{thm:Jackson} for $k\coloneqq(\eta/(\eta+1)+\eps/3)s^dn$.
Indeed, fix $q\in \cQ_{2r}$ and assume $B_q'$ plays the role of $A$ and $R_q$ plays the role of $B$ from \cref{thm:Jackson}.
On the one hand, \cref{lem:colours} readily ensures that $|B_q|\leq k$.
On the other hand, the minimum degree assumption on the vertices of $B_q'$ follows from \cref{lem:bipartite}~\ref{lem:bipartiteitem1}, which immediately implies that $|R_q|\geq k$. 
Lastly, by \cref{lem:bipartite}~\ref{lem:bipartiteitem2} and since $\eta\geq1/2$ we have that
\[|R_q|\leq\left(\frac{1}{\eta+1} + \frac{\eps}{20}\right) s^d n\leq\left(\frac{2\eta}{\eta+1} + \frac{\eps}{20}\right) s^d n\leq2k-2.\]
Thus, by \cref{thm:Jackson}, $H[B_q', R_q]$ contains a cycle of length $2|B_q'|$.
\end{proof}

We are now ready to prove Theorem~\ref{thm:long cycles}.

\begin{proof}[Proof of Theorem~\ref{thm:long cycles}]
Let $0<1/C_3\lll c_3\lll\delta\lll\eps,1/d\leq1$ be such that \cref{lem:basic properties,lem:colours,lem:cycle_cubes,lem:bipartite} (and thus also \cref{cor:bipartite}) hold with $C_3$ and $c_3$ playing the roles of $C$ and $c$, respectively.
Consider $G\sim G_d(n,r)$ and condition on the event that $V\coloneqq V(G)$ satisfies the properties described in \cref{lem:basic properties} (which hold a.a.s.).
Fix any vertex $v\in V$, and let $H$ be an $(\eta+\eps)$-subgraph of~$G$.
We are going to show that $H$ contains a cycle of each of the desired lengths which itself contains~$v$.

Consider the random colouring of the vertices and the random assignment of cells and pairs of cells to red and green vertices.
Since the conclusions of \cref{lem:colours}, \cref{lem:blue neighbours} (for $H$ with minimum degree at least $40\log n$ and probability $\eta/(1+\eta)\ge 1/3$ that a vertex is blue), \cref{lem:cycle_cubes,lem:bipartite,cor:bipartite} hold a.a.s., they also hold a.a.s.\ conditionally on $v$ being blue.
Hence, we can (and do) fix a colouring of $V$ where the vertex $v$ is blue and an assignment of cells and pairs of cells to red and green vertices satisfying the conclusions of \cref{lem:colours,lem:blue neighbours,lem:cycle_cubes,lem:bipartite,cor:bipartite}.

The following claim allows us to construct cycles of any even length $2\ell\le 2\eta n/(1+\eta)$ whose vertices are not too close to the boundary.
Moreover, it ensures that the cycles satisfy certain additional properties with respect to the colouring, which will be useful later to construct the cycles we want.

\begin{claim}\label{claim:paths}
    Let\/ $v'\in Q_{2r}$ be a blue vertex, let\/ $A\subseteq V\setminus\{v'\}$ be a set of blue vertices such that\/ $|A\cap q|\leq2$ for all\/ $q\in\cQ$, and let\/ $q'\in\cQ_{2r}$ be the cell containing\/ $v'$.
    Then, for every integer\/ $\ell\in[2,\eta n/(1+\eta)-|A|]$, there exists a cycle\/ $\mathfrak{O}'\subseteq H-A$ of length\/ $2\ell$ with\/ $v'\in V(\mathfrak{O}')$ such that
    \begin{enumerate}[label=$(\mathrm{P}\arabic*)$]
        \item\label{claim:pathsitem2} every blue vertex in\/ $\mathfrak{O}'$ is contained in\/ $Q_{2r}$;
        \item\label{claim:pathsitem4} every green vertex in\/ $\mathfrak{O}'$ joins two blue vertices which lie in distinct cells sharing a\/ $(d-1)$-dimensional face and is associated to this pair of cells;
        \item\label{claim:pathsitem5} for every pair of cells in\/ $\cQ_{2r}$ which share a $(d-1)$-dimensional face, the number of green vertices of\/ $\mathfrak{O}'$ associated to this pair of cells is at most\/ $2$, and
        \item\label{claim:pathsitem6} $\mathfrak{O}'$ contains a subpath of length\/ $2$ whose endpoints are $v'$ and another blue vertex in\/ $q'$.
    \end{enumerate}
\end{claim}

\begin{claimproof}
    By \cref{lem:colours}, $q'$ contains at least four blue vertices.
    Let~$v''\in V\setminus A$ be a blue vertex in~$q'$ other than~$v'$.
    By \cref{cor:bipartite}, $v'$ and $v''$ are contained in a cycle of length $4$ in $H$ which satisfies \ref{claim:pathsitem2}--\ref{claim:pathsitem6}. 
    Thus, we may assume that $\ell\in[3,\eta n/(1+\eta)-|A|]$.
    
    Fix $\ell\in[3,\eta n/(1+\eta)-|A|]$.
    Consider the integer $m$, the cells $q_1,\ldots,q_m\in \cQ_{2r}$ and the set of blue vertices $S=S_1\cupdot\ldots\cupdot S_m\subseteq V\setminus A$ with $S_j\subseteq q_j$, $|S|=\ell$ and $v'\in S$ given by \cref{lem:cycle_cubes} with $v'$ playing the role of $v$.
    For each \mbox{$j\in [m]$}, by \cref{cor:bipartite} and since $|S_j|\geq2$ by \cref{lem:cycle_cubes}~\ref{lem:cycle_cubesitem1}, the set $S_j$ is contained in a cycle $\mathfrak{O}_j\subseteq H-A$ of length $2|S_j|$ where blue vertices and red vertices associated to $q_j$ alternate.\COMMENT{This last property is here so that the two paths which are chosen below are edge-disjoint.}
    Since each red vertex is assigned to a single cell, the cycles $\mathfrak{O}_1,\ldots,\mathfrak{O}_m$ are vertex-disjoint.
    Now, for every $j\in [m]$, fix an orientation of $\mathfrak{O}_j$ and choose two distinct paths $P_j,P_j'\subseteq\mathfrak{O}_j$ of length~$2$ whose endpoints are blue and such that, for the value of $j$ satisfying $q_j = q'$, at least one of them does not contain $v'$ (which can be done by \cref{lem:cycle_cubes}~\ref{lem:cycle_cubesitem1} applied with $v'$ playing the role of $v$ and since $\ell\ge 3$).
    For every $j\in [m]$, let $w_{1,j}$ be the starting point and $w_{2,j}$ be the endpoint of $P_j$, let $w_{1,j}'$ be the starting point and $w_{2,j}'$ be the endpoint of $P_j'$, and let $T_j,T_j'\subseteq\mathfrak{O}_j$ be the $(w_{2,j},w_{1,j}')$-subpath and $(w_{2,j}',w_{1,j})$-subpath of $\mathfrak{O}_j$, respectively.\COMMENT{Note that the paths of length $2$ may overlap. Indeed, it could be that $w_{1,j}'=w_{2,j}$ or $w_{1,j}=w_{2,j}'$.
    So $T_j$ or $T_j'$ could consist of a single vertex.}
    Then, by \cref{lem:cycle_cubes}~\ref{lem:cycle_cubesitem2} and \cref{lem:bipartite}~\ref{lem:bipartiteitem3}, for every $j\in[m-1]$, there exist two distinct green vertices $z_{j,j+1}$ and $z_{j,j+1}'$ associated to $q_{j}$ and $q_{j+1}$ such that $w_{1,j}'z_{j,j+1},z_{j,j+1}w_{2,j+1},w_{1,j+1}z_{j,j+1}',z_{j,j+1}'w_{2,j}'\in E(H)$.\COMMENT{This only necessitates two distinct green vertices for each pair of blue vertices.}
    Moreover, since each green vertex is associated to a single pair of cells, the collection $\{z_{j,j+1},z_{j,j+1}':j\in[m-1]\}$ has size $2m-2$.
    Thus, $H-A$ contains the cycle
    \[\mathfrak{O}'\coloneqq w_{1,1}P_1\bigg(\!\bigtimes_{j=1}^{m-1} w_{2,j}T_jw_{1,j}'z_{j,j+1}\bigg)w_{2,m}T_mw_{1,m}'P_m'w_{2,m}'T_m'w_{1,m}\bigg(\!\bigtimes_{j=1}^{m-1} z_{m-j,m-j+1}'w_{2,m-j}'T_{m-j}'w_{1,m-j}\bigg),\]
    which contains $S$ (and no other blue vertices) and alternates between blue vertices and vertices of other colours, and thus has length $2\ell$.
    Now, \ref{claim:pathsitem2}, \ref{claim:pathsitem4} and \ref{claim:pathsitem5} hold by construction.
    Finally, in order to verify \ref{claim:pathsitem6}, assume that $q'=q_j$ and note that, by \cref{lem:cycle_cubes}~\ref{lem:cycle_cubesitem1} and the fact that at least one of $P_j,P_j'$ does not contain $v'$, the path among $T_j,T_j'$ containing $v'$ has length at least~2.
    \COMMENT{$\mathfrak{O}'$ contains at least three blue vertices, which appear non-consecutively. After removing two paths of length $2$ with blue endpoints, at least one such path remains.}
\end{claimproof}

Now, fix an integer $\ell\in [2, \eta n/(1+\eta)]$.
First, we are going to show that $H$ contains a cycle of length $2\ell$ which contains $v$.
If $v\in Q_{2r}$, then a cycle of length $2\ell$ containing $v$ is given immediately by \cref{claim:paths} (with $v$ playing the role of $v'$ and $A=\varnothing$).
Thus, suppose that $v\in[0,1]^d\setminus Q_{2r}$ and let $q\in \cQ\setminus \cQ_{2r}$ be the cell containing $v$.
Let $t$ be the length of a shortest $(q,\cQ_{2r})$-path in $\Gamma$, and let the cells of one such path be $q=q'_0, q'_1, \ldots, q'_t$, where the labelling follows the order in which they appear in the path.
We now consider two cases depending on the values of $\ell$ and $t$.

\pagebreak

\noindent
\emph{Case 1: $\ell\in [2t]$.}
In this case, fix a sequence $v = v_0, v_1, \ldots, v_{\ell-1}$ of distinct blue vertices such that $v_0\in q_0'$, $v_j,v_{\ell-j}\in q_j'$ for all $j\in [\lfloor \ell/2\rfloor - 1]$, and $v_{\lfloor \ell/2\rfloor}$ and $v_{\lceil \ell/2\rceil}$ (which coincide when $\ell$ is even) lie in $q_{\lfloor \ell/2\rfloor}'$.
(The existence of such a sequence is guaranteed by \cref{lem:colours}.)
For the sake of notation, let $v_\ell\coloneqq v$.
Then, by Lemma~\ref{lem:bipartite}~\ref{lem:bipartiteitem3}, for each $i\in[\ell]$, the blue vertices $v_{i-1},v_i$ lying in a pair of consecutive cells can be connected by a path of length two via distinct green vertices associated to these cells.
Moreover, when $\ell = 2k+1$ is odd, the vertices $v_k,v_{k+1}\in q_k'$ have an unused common green neighbour associated to the cells $q_k'$ and $q_{k-1}'$ (which is ensured by Lemma~\ref{lem:bipartite}~\ref{lem:bipartiteitem3}).
\vspace{0.5em}

\noindent
\emph{Case 2: $\ell > 2t$.} In this case, by \cref{claim:paths} with $A=\varnothing$, we may find a cycle of length $2\ell-4t+2$\COMMENT{Note that $2\ell-4t+2=2(\ell-2t+1)\geq2\cdot 2$, so we satisfy the required assumptions.} satisfying \ref{claim:pathsitem2}--\ref{claim:pathsitem6} such that the endpoints of the path $P$ of length $2$ guaranteed by \ref{claim:pathsitem6} are contained in~$q_t'$.
Let us label the endpoints of $P$ as $v_t$ and $v_t'$, and let us denote the path resulting from the cycle after deleting the middle vertex of $P$ by $P_1$.
Now, by \cref{lem:colours}, for every $j\in[t-1]$, we can choose two distinct blue vertices $v_j$ and $v_j'$ in $q_j'$.
Also, set $v_0=v_0'\coloneqq v$.
Then, by \cref{lem:bipartite}~\ref{lem:bipartiteitem3}, for every $j\in[t]$, we can find distinct green vertices that join $v_{j-1}$ to $v_j$ and $v_{j-1}'$ to $v_j'$ and are associated to $q_{j-1}'$ and~$q_j'$.
This results in a $(v_t,v_t')$-path $P_2$ of length $4t$ which, by \ref{claim:pathsitem2} and \ref{claim:pathsitem4} and since green vertices are assigned to a unique pair of cells, satisfies that $V(P_1)\cap V(P_2)=\{v_t,v_t'\}$. Hence, $P_1\cup P_2$ is a cycle containing $v$ of length $L=2\ell$, as desired.

\vspace{0.5em}

Finally, we have to consider cycles of odd length containing a vertex $v$ in a cell $q$.
Fix any odd $L\in[c_3^{-1},2\eta n/(1+\eta)]$.
By Lemmas~\ref{lem:colours} and~\ref{lem:blue neighbours} and the fact that $d_H(v) > 6(\eta/(1+\eta)+\eps/3)s^dn$,\COMMENT{By \cref{lem:colours}, each cell contains at most $(\eta/(1+\eta)+\eps/3)s^dn$ blue vertices, so if $v$ has more than that many blue neighbours, it must have a blue neighbour in another cell. By \cref{lem:blue neighbours}, the number of blue neighbours is at least $1/6$-th of the total degree, so it suffices to verify the condition above. And this holds easily by the choice of parameters.} there is an edge incident to $v$ whose other endpoint is a blue vertex outside $q$.
Let this other endpoint be denoted $v'$, and let $q'\in\cQ$ be the cell containing $v'$.

We begin by choosing a set of cells where we will pick blue vertices for our construction.
Consider a shortest $(q,q')$-path $\mathfrak{P}\subseteq\Gamma$, let $t$ be its length, and label its cells as $q=q_0,q_1,\ldots,q_t=q'$ following the order in which they appear in $\mathfrak{P}$.
Note that, by the choice of the side length $s$, we must have $t\leq 2d^2/\delta$.\COMMENT{Since $v'$ is a neighbour of $v$, it is at distance at most $r$ from $v$. As $s\geq\delta r/2d$, it follows that we must traverse at most $2d/\delta$ cells in each dimension. So just add up over all dimensions. This is a crude upper bound.}
Moreover, consider a shortest $(V(\mathfrak{P}),\cQ_{2r})$-path $\mathfrak{P}'\subseteq\Gamma$, and let $t'$ be its length.
Let $i^*\in[0,t]$ denote the index such that $q_{i^*}\in V(\mathfrak{P}')$ and label the cells in $V(\mathfrak{P}')$ as $q_{i^*}=q_0',q_1',\ldots,q_{t'}'$, following the order in which they appear in $\mathfrak{P}'$.
(We remark here that our choice of $c_3$ ensures that the paths we look for are long enough that we will always reach $Q_{2r}$, thus not requiring the same case distinction as the cycles of even length.)
Note that, similarly as above, we must have that $t'\leq4d^2/\delta$.
See \Cref{fig:shortodd} for a representation.

Next, we are going to choose vertices in the cells picked above and use them to construct a short path of odd length containing $v$ and whose endpoints are in $Q_{2r}$.
Recall that, by \cref{lem:colours}, every cell contains $\Theta(\log n)$ blue vertices.
Define $v_0\coloneqq v$, $v_t\coloneqq v'$ and, for every $i\in[t-1]$, choose an arbitrary blue vertex $v_i\in V\cap q_i$.
Let $w_0\coloneqq v_{i^*}$.
For every $i\in[t']$, choose a blue vertex $w_i\in V\cap q_i'$.
Finally, for every $i\in[0,t'-1]$, choose a blue vertex $w_i'\in (V\cap q_i')\setminus\{w_i\}$.
All of these vertices will be used to construct a path later.
Let $B\coloneqq\{v_i:i\in[0,t]\setminus\{i^*\}\}\cup\{w_i,w_i':i\in[0,t'-1]\}$ denote the set of all vertices defined above except $w_{t'}$.
Observe that, by the minimality of $\mathfrak{P}$ and $\mathfrak{P}'$ and our choices above, $B$ contains at most two vertices in each cell.
Define $\ell_P\coloneqq 2t+4t'+1=2|B|+1$ so that $\ell_P\leq 20d^2/\delta+1\leq c_{3}^{-1}-2\leq L-2$.

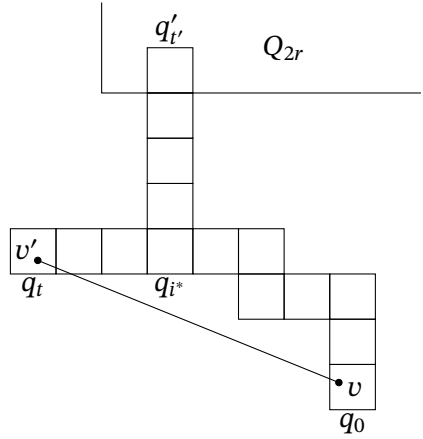
\begin{figure}[ht]
    \centering
    \begin{tikzpicture}[scale=0.6]
    \foreach \i/\j in
    {1/2,2/2,3/2,4/2,5/2,6/2, 6/1,7/1,8/1,8/0,8/-1,4/3,4/4,4/5,4/6}
    \draw (\i,\j) -- (\i+1,\j) -- (\i+1,\j+1) -- (\i,\j+1) -- cycle;
    \draw[color=black] (1.35,2.55) node {$v'$};
    \draw [fill=black] (1.6,2.3) circle (2pt);
    \draw[color=black] (8.55,-0.55) node {$v$};
    \draw [fill=black] (8.2,-0.4) circle (2pt);
    \draw (1.6,2.3) -- (8.2,-0.4) ;
    \draw (3,8) -- (3,6) -- (10,6);
    \draw[color=black] (7,7) node {$Q_{2r}$};
    \draw[color=black] (1.5,1.7) node {$q_t$};
    \draw[color=black] (4.5,7.4) node {$q'_{t'}$};
    \draw[color=black] (8.5,-1.3) node {$q_0$};
    \draw[color=black] (4.5,1.7) node {$q_{i^*}$};
    \end{tikzpicture}
    \caption{Sketch for the construction of a short odd cycle containing $v$, in the case $d=2$, when both $v$ and its blue neighbour $v'$ lie outside $Q_{2r}$.
    The highlighted cells are used to construct paths from $v$ and $v'$ to $q'_{t'}$, where an application of \cref{claim:paths} allows to close them into a cycle.}
    \label{fig:shortodd}
\end{figure}

By applying Claim~\ref{claim:paths} (with $w_{t'}$, $q_{t'}'$ and $B$ playing the roles of $v'$, $q'$ and $A$), one may find a cycle~$\mathfrak{O}'$ of length $L-\ell_P+2\geq4$\COMMENT{Note that, by the preceding upper bound on $\ell_P$, we have $(L-\ell_P+2)/2\geq2$, so we have the lower bound needed for \cref{claim:paths}. For the upper bound, we must check that $(L-\ell_P+2)/2\leq\eta n/(1+\eta)-|B|$.
Using the definition of $\ell_P$ and the assumed upper bound on $L$, it suffices to check that $|B|\geq\ell_p/2-1$, which holds.} satisfying \ref{claim:pathsitem2}--\ref{claim:pathsitem6}. 
Let $w_{t'}'$ be the other (blue) endpoint of the subpath of length $2$ given by \ref{claim:pathsitem6}, let $P'$ be the other subpath of $\mathfrak{O}'$ between $w_{t'}$ and $w_{t'}'$ (which has length $L-\ell_P$), and let $\cG$ denote the set of all green vertices of~$\mathfrak{O}'$.

By \cref{lem:bipartite}~\ref{lem:bipartiteitem3} in conjunction with \ref{claim:pathsitem5}, for each $i\in [t]\setminus \{i^*+1\}$, we may find a green vertex $z_i\in V\setminus\cG$ which is joined by an edge to both $v_{i-1}$ and $v_i$ and is associated to $q_{i-1}$ and $q_i$; additionally, for $i=i^*+1$, we may find a green vertex $z_{i^*+1}\in V\setminus\cG$ associated to $q_{i^*}$ and $q_{i^*+1}$ joined to both~$v_{i^*}'$ and~$v_{i^*+1}$.
Similarly, by \cref{lem:bipartite}~\ref{lem:bipartiteitem3} and \ref{claim:pathsitem5}, for each $i\in[t']$ we may find two green vertices~$y_i$ and~$y_i'$, both associated to $q_{i-1}'$ and $q_i'$, such that $y_i$ is joined by an edge to both $w_{i-1}$ and $w_i$, $y_i'$ is joined to both $w_{i-1}'$ and $w_i'$, and all these vertices are pairwise distinct and distinct from the $z_j$'s and the vertices in $\cG$. This results in the path
\[P\coloneqq \bigtimes_{j=0}^{t'-1}(w_{t'-j}'y_{t'-j}')v_{i^*}'\bigtimes_{j=i^*+1}^{t}(z_{j}v_j)v_0\bigtimes_{j=1}^{i^*}(z_{j}v_j)\bigtimes_{j=1}^{t'}(y_{j}w_{j}).\]
Observe that, by construction, $P$ is a path of length $\ell_P$.
Moreover, by the construction of $\mathfrak{O}'$ and $B$,  and since $P$ has no red vertices, it is internally disjoint from $\mathfrak{O}'$.
In conclusion, $P\cup P'$ is a cycle of length $L$ containing $v$, as desired.
\end{proof}

%%%%%%%%%%%%%%%%%%%%%%%%%%%%%%%%%%%%%
%%%%%%%%%%%%%%%%%%%%%%%%%%%%%%%%%%%%%

\section{Bounds on the local resilience of \texorpdfstring{$G_d(n,r)$}{RGGs} with respect to Hamiltonicity}\label{section:bounds}

Our ideas for proving lower bounds on the local resilience of random geometric graphs with respect to Hamiltonicity build on the classical approach to construct Hamilton cycles in $G_d(n,r)$.
This approach relies on the fact that vertices which are geometrically close in $[0,1]^d$ form large cliques.
More precisely, the general idea is to tessellate the hypercube with cells in such a way that, for every pair of cells sharing a $(d-1)$-dimensional face,
the vertices in them form a clique in $G_d(n,r)$.
If $r\geq C(\log n/n)^{1/d}$ for a sufficiently large $C$, then all cells will contain many vertices.
At this point, one may consider an auxiliary path which visits every cell once and construct a Hamilton cycle by traversing the cells twice following the path (forwards and then backwards), incorporating the vertices into the cycle while doing so. 
Technical variations of this idea have been used to obtain many results about Hamilton cycles in random geometric graphs~\cite{BBP-GP17,Es23,EH23,FP-G20,Petit01,DMP07,BBKMW11,Man23, MPW11,FMMS21}.

Here, we use a simple new variant where, instead of having cliques inside each region, we note that it suffices that the vertices in each cell can be covered by two paths (in a sufficiently flexible way), and that there are some edges joining vertices in neighbouring cells so that the paths can be ``glued'' together.
We show that both properties hold for $\alpha$-subgraphs of $G_d(n,r)$ for sufficiently large~$r$ and~$\alpha < 1$.

To be more precise, in order to prove \cref{thm:lowerboundgeneral}, we rely on a simple structural lemma (\cref{lemma:2paths}) providing a sufficient minimum-degree condition to have two vertex-disjoint paths with specified endpoints which contain all the vertices of a graph.
The minimum degree required is comparable with that of the original theorem of \citet{Dirac52}.
Our simple proof of this lemma uses the following classical result of \citet{BC76}, which also appears somewhat implicitly in an earlier work of \citet{Ore60}.

\begin{theorem}[\cite{BC76}]\label{thm:BC76}
    Let\/ $G$ be a graph on\/ $n\geq 3$ vertices.
    Assume\/ $u,v\in V(G)$ with\/ $uv\notin E(G)$ and\/ $d_G(u)+d_G(v)\geq n$.
    Then,\/ $G$ is Hamiltonian if and only if\/ $G\cup\{uv\}$ is Hamiltonian.
\end{theorem}

\begin{lemma}\label{lemma:2paths}
    Let\/ $G$ be a graph on\/ $n\geq4$ vertices.
    If\/ $\delta(G)\geq n/2+1$, then, for any two pairs of distinct vertices\/ $A=\{u,v\}$ and\/ $B=\{w,x\}$ such that\/ $A\cap B=\varnothing$, the graph\/ $G$ contains two vertex-disjoint\/ $(A,B)$-paths\/ $P$ and\/ $Q$ such that $V(P)\cup V(Q)=V(G)$.
\end{lemma}

\begin{proof}
    Suppose $\delta(G)\geq n/2+1$.
    Observe that the conclusion holds if and only if for any two pairs of distinct vertices $\{u,v\}$ and $\{w,x\}$ such that $\{u,v\}\cap\{w,x\}=\varnothing$, the graph 
    \[G'\coloneqq(V(G)\cup\{y,z\},E(G)\cup\{uy,yv,wz,zx\}),\]
    where $y,z$ are two new vertices, is Hamiltonian. 
    Note that $G'$ has $n'\coloneqq n+2$ vertices, and so all vertices of $V(G)$ have degree at least $n'/2$ in $G'$.
    Thus, by \cref{thm:BC76}, $G'$ is Hamiltonian if and only if the complete graph on $V(G)$ together with the paths $uyv$ and $wzx$ is Hamiltonian, which it clearly is since $\{u,v\}\cap\{w,x\}=\varnothing$.
\end{proof}

With this, we can already prove \cref{thm:lowerboundgeneral}.

\begin{proof}[Proof of \cref{thm:lowerboundgeneral}]
    Let $0<1/C\lll\delta\lll1/d,\eps\leq1$ and let $C(\log n/n)^{1/d}\leq r=o(1)$.
    For the convenience of the reader, we recall that we wish to show that a.a.s.\ every $(1-1/(2d^{d/2}\theta_d)+\eps)$-subgraph of $G_d(n,r)$ is Hamiltonian, where $\theta_d$ is the volume of a $d$-dimensional ball of radius~$1$.
    
    To achieve this, we consider two tessellations of $[0,1]^d$ with cells.
    First, we consider a tessellation with cells of side length $s\coloneqq \lceil\sqrt{d}/r\rceil^{-1}$;\COMMENT{Essentially, $s=r/\sqrt{d}$, only slightly smaller so that it really gives a partition into cells. Note that the bound $r=o(1)$ is in place so that really $s=(1-o(1))r/\sqrt{d}$.}
    we denote the set of all the resulting cells by $\cQ$.
    Then, we let~$k$ be the smallest multiple of $1/s$ with $k\geq 1/\delta s$ and consider a second tessellation $\cQ^*$ of $[0,1]^d$ with cells of side length $1/k$.
    Observe that $\cQ^*$ is a refinement of $\cQ$, that is, $\cQ^*$ tessellates each cell of~$\cQ$ with smaller cells.
    Also, recall the auxiliary graph $\Gamma$ with vertex set $\cQ$ where two cells are joined by an edge if and only if they share a $(d-1)$-dimensional face.

    Let $G\sim G_d(n,r)$.
    By the choice of $s$, for any vertex $v\in V(G)$, the cell containing $v$ is itself contained in $B(v,r)$, and so all vertices which lie in the same cell of $\cQ$ induce a clique in $G$.
    By Chernoff's bound (\cref{lem:chernoff}), a.a.s.\ the following properties hold:
    \begin{enumerate}[label=$(\mathrm{P}\arabic*)$]
        \item\label{item:lowbound2} every cell of $\cQ^*$ contains $(1\pm\eps/18)n/k^d$ vertices of $V(G)$, and
        \item\label{item:lowbound3} every vertex $v\in V(G)$ satisfies $d_G(v)=(1\pm\eps/8)|B(v,r)\cap[0,1]^d|n$.
    \end{enumerate}
    Condition on the event that $G$ satisfies these two properties.

    Let $H\subseteq G$ be any $(1-1/(2d^{d/2}\theta_d)+\eps)$-subgraph of $G$.
    For each $q\in\cQ$, let $G[q]\coloneqq G[V(G)\cap q]$ and $H[q]\coloneqq H[V(G)\cap q]$.
    Now, fix an arbitrary cell $q\in\cQ$.
    We claim that 
    \begin{equation}\label{equa:lowbound1}
        \delta(H[q])\geq(1+\eps)|V(H[q])|/2.
    \end{equation}
    Indeed, fix any vertex $v$ in $H[q]$.
    By \ref{item:lowbound2} and since the vertices in $q$ induce a clique in $G$, we know that $d_{G[q]}(v)\geq (1-\eps/16)s^dn$\COMMENT{Just apply \ref{item:lowbound2} for each of the subcells. Since here we are looking at degree, we have to subtract $1$ from what is given by \ref{item:lowbound2}. So we can simply write a worse constant in the error.}, and by \ref{item:lowbound3}, we know that $d_G(v)\leq(1+\eps/8)|B(v,r)\cap[0,1]^d|n$.
    Moreover, note that $|B(v,r)\cap[0,1]^d|\leq\theta_dr^d\leq(1+\eps/8)\theta_dd^{d/2}s^d$ by the definition of $s$.
    Now, \eqref{equa:lowbound1} follows by combining these three estimates.
    Indeed, since $H$ is a $(1-1/(2d^{d/2}\theta_d)+\eps)$-subgraph of~$G$, it follows that the number of edges of $G$ incident to $v$ which have been deleted when choosing $H$ satisfies that
    \begin{align*}
        d_G(v) - d_{H}(v)&\leq(1/(2d^{d/2}\theta_d)-\eps)d_G(v)\\
        &\leq(1/(2d^{d/2}\theta_d)-\eps)(1+\eps/8)|B(v,r)\cap[0,1]^d|n\\
        &\leq(1/(2d^{d/2}\theta_d)-\eps)(1+\eps/8)(1+\eps/8)\theta_dd^{d/2}s^dn\\
        &\leq(1/2-\eps)(1+\eps/8)^2s^dn\\
        &\leq(1/2-\eps)\frac{(1+\eps/8)^2}{1-\eps/16} d_{G[q]}(v)\\
        &\leq(1/2-\eps)(1+\eps/8)^3 d_{G[q]}(v)\\
        &\leq(1/2-3\eps/4) d_{G[q]}(v),
    \end{align*}
    where the fourth inequality uses that $\theta_d d^{d/2}\geq1$ by~\eqref{equa:ballbound}, the sixth one uses that $1/(1-x)\leq1+2x$ for $x\in[0,1/2]$, and for the last one we have that $(1+\eps/8)^3=1+3\eps/8+3\eps^2/64+(\eps/8)^3\leq1+\eps/2$ (which holds for $\eps\leq1$) and $(1/2-\eps)(1+\eps/2)\leq1/2-3\eps/4$.
    By considering the complement, it follows that $d_{H[q]}(v) \geq (1+\eps)d_{G[q]}(v)/2$, as desired.\COMMENT{The reason why we do not write $(1-\eps)/2$ in the last inequality is that that we can ignore the $\pm1$ which appears because of the degree of a vertex in a clique.}

    Moreover, we claim that, for any pair of cells $q,q'\in\cQ$ such that $qq'\in E(\Gamma)$,
    \begin{equation}\label{equa:lowbound2}
        \text{there are vertices $v,w\in V(G)\cap q$ and $v',w'\in V(G)\cap q'$ such that $vv',ww'\in E(H)$.}
    \end{equation}
    To prove this, consider an arbitrary pair of cells $q,q'\in\cQ$ with $qq'\in E(\Gamma)$ and let $q^*\in\cQ^*$ be a cell in $q$ which contains the centre of the $(d-1)$-dimensional face common for $q$ and $q'$.
    Note that $q^*$ contains at least $2$ vertices of $G$ by \ref{item:lowbound2}; let $v$ and $w$ be any such vertices and observe that, by the choice of $\delta$, the region $q'\cap B(v,r)\cap B(w,r)$ contains at least $(sk)^d-1$ cells of $\cQ^*$ (in fact, $q'\subseteq B(v,r)\cap B(w,r)$ for all $d\geq2$).\COMMENT{The bound on the number of cells is only in place for the case $d=1$.}
    Now, by the same argument using \ref{item:lowbound2} and \ref{item:lowbound3} as above, and using the choices of $\delta$ and $k$, we conclude that both $v$ and $w$ are $H$-neighbours to at least half of the $\omega(1)$ vertices in $q'$,\COMMENT{Let us discuss the argument for $v$, as it is analogous for $w$. The only difference that one needs to consider is the fact that we are not guaranteed that $q'\subseteq B(v,r)$ (when $d=1$), so we do not immediately have access to the lower bound on the $G$-degree of $v$ into $q'$. However, since $q'\cap B(v,r)$ contains at least $(sk)^d-1$ cells of $\cQ^*$, by applying \ref{item:lowbound2} we know that $v$ has at least \[\left(1-\frac{\eps}{18}\right)\frac{n}{k^d}(s^dk^d-1)=\left(1-\frac{\eps}{18}\right)s^dn-\left(1-\frac{\eps}{18}\right)\frac{n}{k^d}\geq\left(1-\frac{\eps}{18}\right)s^dn-\left(1-\frac{\eps}{18}\right)\delta^ds^dn\geq\left(1-\frac{\eps}{16}\right)s^dn\] $G$-neighbours (where the first inequality holds by the choice of $k$ and the last holds by the choice of $\delta$). From this point on, the argument is exactly the same as for \eqref{equa:lowbound1}.}
    and thus, we can find distinct $H$-neighbours $v',w'\in V(G)\cap q'$ of $v$ and $w$, respectively.

    We are now ready to construct a Hamilton cycle.
    Let $P\subseteq\Gamma$ be an arbitrary spanning path of~$\Gamma$.\COMMENT{Note that $\Gamma$ is simply a grid, so it does contain such a path.}
    Label the cells in $\cQ$ as $q_1,\ldots,q_{1/s^d}$ following the order given by $P$.
    For each $i\in[1/s^d-1]$, by applying \eqref{equa:lowbound2}, let $v_i,w_i\in V(G)\cap q_i$ and $v_{i+1}',w_{i+1}'\in V(G)\cap q_{i+1}$ be such that $v_iv_{i+1}',w_iw_{i+1}'\in E(H)$.
    Moreover, choose an arbitrary edge in $q_1$ with endpoints $v_1', w_1'$ different from $v_1, w_1$; similarly, choose an arbitrary edge in $q_{1/s^d}$ with endpoints $v_{1/s^d}, w_{1/s^d}$ different from $v_{1/s^d}', w_{1/s^d}'$.\COMMENT{Recall such edges must exist since $H[q]$ has large minimum degree, by combining \eqref{equa:lowbound1} and \ref{item:lowbound2}.}
    Then, for every $q\in\cQ$,~\eqref{equa:lowbound1} and \ref{item:lowbound2} show that $H[q]$ satisfies the conditions of \cref{lemma:2paths}.
    As a result, for every $i\in[1/s^d]$, we can find two vertex-disjoint $(\{v_i',w_i'\},\{v_i,w_i\})$-paths which together contain all the vertices of $H[q]$.
    We label these paths inductively as follows.
    
    Let $x_1'\coloneqq v_1'$.
    For each $i\in[1/s^d-1]$, starting at $i=1$ and increasing $i$ by one at each step, we let $P_i\subseteq H[q_i]$ denote the path given by \cref{lemma:2paths} having $x_i'$ as an endpoint, and let $P_i'$ denote the other path.
    Moreover, we let $x_i$ denote the other endpoint of $P_i$ and set 
    \[x_{i+1}'\coloneqq\begin{cases}
        v_{i+1}'&\text{ if }x_i=v_i,\\
        w_{i+1}'&\text{ if } x_i = w_i.
    \end{cases}\]
    Finally, when $i=1/s^d$, we similarly let $P_i\subseteq H[q_i]$ denote the path given by \cref{lemma:2paths} having $x_i'$ as an endpoint, let $P_i'$ denote the other path, and let $x_i$ denote the other endpoint of $P_i$.
    Now, for each $i\in[1/s^d]$, let $y_i\in\{v_i,w_i\}$ and $y_i'\in\{v_i',w_i'\}$ be such that $\{x_i,y_i\}=\{v_i,w_i\}$ and $\{x_i',y_i'\}=\{v_i',w_i'\}$.
    
    With the notation that we have set up, it follows that 
    \[y_1'\bigg(\bigtimes_{i=1}^{1/s^d} x_i'P_ix_i\bigg)\bigg(\bigtimes_{i=0}^{1/s^d-1} y_{1/s^d-i}P_{1/s^d-i}'y_{1/s^d-i}'\bigg)\] 
    is a Hamilton cycle.
\end{proof}

In order to improve \cref{thm:lowerboundgeneral} using similar ideas, one can simply observe that the minimum-degree condition in \cref{lemma:2paths} can be sharpened by considering degree sequences.
Given a graph~$G$ on $n$ vertices, consider an ordered list of the degrees of the vertices in $G$, which we label as \mbox{$d_1\leq\ldots\leq d_n$}.
We refer to this list as the \emph{degree sequence} of $G$.
The following lemma is a strengthening of \cref{lemma:2paths} and can be seen as a variant of a classical result of \citet{Posa62} which yields a sufficient condition for Hamiltonicity.

\begin{lemma}\label{lemma:2pathsv2}
    Let\/ $G$ be a graph on\/ $n\geq5$ vertices\COMMENT{The reason why we take $5$ rather than $4$ is that, the way the condition is written, it would give no restrictions at all for $n=4$. The only $4$-vertex graph for which the statement holds is $K_4$.} whose degree sequence $d_1\leq\ldots\leq d_n$ satisfies that\/ $d_i\geq i+3$ for all\/ $i<n/2-1$.
    Then, for any two pairs of distinct vertices\/ $A=\{u,v\}$ and\/ $B=\{w,x\}$ of\/ $G$ such that\/ $A\cap B=\varnothing$, the graph\/ $G$ contains two vertex-disjoint\/ $(A,B)$-paths\/ $P$ and\/ $Q$ such that $V(P)\cup V(Q)=V(G)$.
\end{lemma}

\begin{proof}
    Let $G$ be a graph satisfying the conditions from the statement.
    Observe that the conclusion holds for $G$ if and only if, for any two pairs of distinct vertices $\{u,v\}$ and $\{w,x\}$ of $G$ such that $\{u,v\}\cap\{w,x\}=\varnothing$, the graph 
    \[G'\coloneqq(V(G)\cup\{y,z\},E(G)\cup\{uy,yv,wz,zx\}),\] 
    where $y,z$ are two new vertices, is Hamiltonian.
    We are going to apply \cref{thm:BC76} iteratively to show that this is the case.

    Let $\ell\coloneqq \lceil n/2\rceil$.
    Note that $G'$ has $n'\coloneqq n+2$ vertices, and that its degree sequence $d'_1\leq\ldots\leq d'_{n'}$ satisfies that $d'_1=d'_2=2$ and $d'_i\geq i+1$ for all $i\in\{3,\ldots,\ell\}$ (where $\ell\coloneqq\lceil n'/2\rceil-1$).
    Let us label the vertices of $G'$ as $v_1,\ldots,v_{n'}$ in such a way that $d_{G'}(v_i)=d'_i$ for each $i\in[n']$.
    The degree sequence of~$G'$ guarantees that $d'_\ell\geq\lceil n'/2\rceil$, so at least $n'-\ell+1=\lfloor n'/2\rfloor+2$ vertices of $G'$ have degree at least~$n'/2$.
    Thus, by \cref{thm:BC76}, $G'$ is Hamiltonian if and only if the graph
    \[G_0\coloneqq(V(G'),E(G')\cup\{v_iv_j:\ell\leq i<j\leq n'\})\]
    is Hamiltonian.
    Note that $d_{G_0}(v_\ell)\geq\lfloor n'/2\rfloor+1$.

    We now define a sequence of nested graphs $G_0\subseteq G_1\subseteq\ldots\subseteq G_{\ell-3}$ where, for each $i\in[\ell-3]$, we set
    \[G_i\coloneqq(V(G'),E(G_{i-1})\cup\{v_{\ell-i}v_j:\ell-i+1\leq j\leq n'\}).\]
    For each $i\in[\ell-3]$, we have inductively that $d_{G_{i-1}}(v_j)\geq\lfloor n'/2\rfloor+i$ for all $j\in \{\ell-i+1,\ldots,n'\}$,  
    and since $d'_{\ell-i}\geq \ell-i+1$, it follows from \cref{thm:BC76} that $G_{i-1}$ is Hamiltonian if and only if $G_i$ is Hamiltonian.
    The conclusion then is that $G'$ is Hamiltonian if and only if $G_{\ell-3}$ is Hamiltonian.
    But~$G_{\ell-3}$ consists of a clique of size $n$ together with the edges $uy,yv,wz,zx$, which is clearly Hamiltonian since $\{u,v\}\cap\{w,x\}=\varnothing$.
\end{proof}

We can now use this lemma to prove a strengthening of \cref{thm:lowerboundgeneral} by ensuring that the graphs inside each cell satisfy the necessary degree-sequence condition instead of the minimum-degree condition of \cref{lemma:2paths}.
Checking this condition carefully becomes more complicated for higher dimensions, so here we consider only the case $d=1$.

\begin{proof}[Proof of \cref{thm:lowerboundd=1}]
    Let $0<1/C\lll\delta\lll\eps\leq1/3$ and let $C\log n/n\leq r=o(1)$.
    Consider two tessellations of $[0,1]$ with intervals.
    First, we consider a tessellation $\mathcal{Q}$ with intervals of length $s\coloneqq \lceil 3/4r\rceil^{-1}$.
    \COMMENT{Essentially, $s=4r/3$, only slightly smaller so that it really gives a tessellation with intervals of the same length. Note that the bound $r=o(1)$ is in place so that really $s=(1-o(1))4r/3$.}
    For each interval $I\in\mathcal{Q}$, let $I_c$ denote the interval of length $s/2$ whose midpoint coincides with the one of $I$. 
    Also, let $k$ be the smallest multiple of $4/s$ with $k\geq 1/\delta s$ and consider a second tessellation $\mathcal{Q}^*$ of $[0,1]$ with intervals of length $1/k$.
    Note that $\mathcal{Q}^*$ is a refinement of $\mathcal{Q}$, that is, $\mathcal{Q}^*$ tessellates each interval of $\mathcal{Q}$ with intervals of $\mathcal{Q}^*$, and also that each interval $I_c$ for $I\in\cQ$ is tessellated with intervals of $\cQ^*$.\COMMENT{This is by the choice that $k$ is $4$ times a multiple of $1/s$.}
    Recall the auxiliary graph $\Gamma$ with vertex set $\mathcal{Q}$ where two intervals are adjacent if and only if they share an endpoint (so $\Gamma$ is a path on $1/s$ vertices).

    Let $G\sim G_1(n,r)$.
    By Chernoff's bound (\cref{lem:chernoff}), a.a.s.\ the following properties hold:
    \begin{enumerate}[label=$(\mathrm{P}\arabic*)$]
        \item\label{item:lowbound2v2} every interval of $\mathcal{Q}^*$ contains $(1\pm\delta)n/k$ vertices of $V(G)$, and
        \item\label{item:lowbound3v2} every vertex $v\in V(G)$ satisfies $d_G(v)=(1\pm\delta)|B(v,r)\cap[0,1]|n$.
    \end{enumerate}
    Condition on the event that $G$ satisfies these two properties.
    Let $H\subseteq G$ be any $(2/3+\eps)$-subgraph of~$G$.
    For each $I\in\mathcal{Q}$, let $G[I]\coloneqq G[V(G)\cap I]$ and $H[I]\coloneqq H[V(G)\cap I]$.

    \begin{claim}\label{claim:lowbound1}
        Fix an arbitrary interval\/ $I\in\mathcal{Q}$ and let\/ $x\coloneqq|V(G)\cap I|$.
        Let\/ $d_1\leq\ldots\leq d_x$ be the degree sequence of\/ $H[I]$.
        Then, for all\/ $i<x/2-1$, we have that\/ $d_i\geq i+3$.
    \end{claim}

    \begin{claimproof}
        Let $A \coloneqq V(G)\cap I_c$.
        By the choice of $s$, for any point $p\in I_c$ we have that $I\subseteq B(p,r)$,\COMMENT{The maximal possible distance from a point in $I_c$ to a point in $I$ is at most $r$.} so every vertex in $A$ is adjacent (in $G$) to every other vertex inside $I$.
        Using \ref{item:lowbound2v2} shows that $x=(1\pm\delta)sn$ and $|A|=(1\pm\delta)sn/2$.\COMMENT{Just apply \ref{item:lowbound2v2} to each subinterval (and recall that both $I$ and $I_c$ are tessellated by $\cQ^*$).}
        This guarantees that $|A|\geq (1-\delta)x/(2+2\delta)\ge (1-2\delta)x/2$.
        Moreover, all vertices of~$A$ have degree~$x-1$ in $G[I]$.
        As $H$ is a $(2/3+\eps)$-subgraph of $G$, using \ref{item:lowbound3v2} we conclude that, for all $v\in A$,
        \COMMENT{To be more precise, note that we have
        \begin{align*}
            d_{H(I)}(v)&\geq x-1-(1/3-\eps)d_G(v)\\
            &=(1\pm\delta)sn-1-(1/3-\eps)(1\pm\delta)|B(v,r)\cap[0,1]^d|n\\
            &\geq(1-\delta)sn-(1+\delta)2rn/3+(1+\delta)\eps2rn-1\\
            &\geq(1-\delta)sn-(1+2\delta)sn/2+(1+\delta)\eps2rn-1\\
            &\geq (s/2+\eps2r)n\geq(1-\delta)(x/2+\eps2rn)\geq(1+5\delta)x/2,
        \end{align*}
        where in the fourth line we use that $r\leq(1+\delta/2)3s/4$ since $s=(1-o(1))4r/3$, and the last inequality simply uses the fact that $x=\Theta(rn)$ and that $\delta\ll\eps$.}
        \begin{equation}\label{equa:resbounddegseq1}
            d_{H[I]}(v)\geq x-1-(1/3-\eps)d_G(v)\geq(1+5\delta)x/2.
        \end{equation}

        Now, consider the set $\mathcal{Q}^*[I]\coloneqq\{I^*\in\mathcal{Q}^*:I^*\subseteq I\}$.
        We partition $\mathcal{Q}^*[I]$ into sets $\mathcal{Q}^*_1,\ldots,\mathcal{Q}^*_{ks/2}$ where each $\mathcal{Q}^*_i$ with $i\in[ks/2]$ contains the two intervals $I^*\in\mathcal{Q}^*[I]$ at distance $(i-1)/k$ from the endpoints of $I$.
        Our next goal is to show that the vertices contained in intervals $I^*\in\mathcal{Q}_i^*$ with $i\in[ks/4]$ have sufficiently large degrees in $H[I]$ to satisfy the desired degree sequence condition (larger values of $i$ satisfy the desired condition by \eqref{equa:resbounddegseq1}).
        By using the definitions of $s$ and $k$, it readily follows that, for each $i\in[ks/4]$ and each $I^*\in\mathcal{Q}_i$, every point $p\in I^*$ satisfies that 
        \[|\{I^{**}\in\mathcal{Q}^*[I]:I^{**}\subseteq B(p,r)\}|\geq \frac{3ks}{4}+i-1.\]
        Therefore, by \ref{item:lowbound2v2}, all vertices $v\in V(G)$ contained in an interval of $\mathcal{Q}^*_i$ with $i\in[ks/4]$ satisfy that\COMMENT{The $2\delta$ is there only so we can ignore the $-1$ missing from ignoring the vertex itself.}
        \[d_{G[I]}(v)\geq (1-\delta) \bigg(\frac{3ks}{4}+i-1\bigg) \frac{n}{k} - 1\geq (1-2\delta) \bigg(\frac{3sn}{4}+\frac{(i-1)n}{k}\bigg)\]
        and, as $H$ is a $(2/3+\eps)$-subgraph of $G$ and $d_G(v)\le (1+\delta)2rn$ by \ref{item:lowbound3v2}, we conclude that
        \COMMENT{\aedc{I believe this footnote is now outdated. I checked again the calculation and am happy, though the last inequality is nontrivial to me right now} We have
        \begin{align*}
            d_{H[I]}(v)&\geq(1-2\delta)\frac12\left(1+\frac{i-1}{\lceil ks/4\rceil}\right)sn-(1/3-\eps)d_{G}(v)\\
            &\geq(1-2\delta)\frac12\left(1+\frac{i-1}{\lceil ks/4\rceil}\right)sn-(1+\delta)(1/3-\eps)2rn\\
            &\geq(1-2\delta)\frac{i-1}{\lceil ks/4\rceil}\frac{sn}{2}-4\delta sn+\eps2rn\\
            &\geq\frac{i-1}{\lceil ks/4\rceil}\frac{x}{2}+10\delta x,
        \end{align*}
        where in the third line we use that $2r/3\leq(1+\delta/2)s/2$, so $(1+\delta)2rn/3\geq(1+2\delta)sn/2$, and the last inequality follows again since $x=\Theta(rn)$ and $\delta\ll\eps$.}
        \begin{align}
        d_{H[I]}(v)
        &\geq(1-2\delta)\bigg(\frac{3sn}{4}+\frac{(i-1)n}{k}\bigg)-\bigg(\frac{1}{3}-\eps\bigg)(1+\delta)2rn\nonumber\\
        &\geq \frac{1-2\delta}{1+\delta} \bigg(\frac{3}{4} + \frac{i-1}{ks}\bigg) x -\bigg(\frac{1}{3}-\eps\bigg)\frac{3(1+\delta)}{2} x\nonumber\\
        &\geq\left(\frac{2i}{ks}+5\delta\right)x\geq(1+\delta)\frac{2in}{k}+3.\label{eq:bounddegree1}
        \end{align}
        Thus, for each $i\in[ks/4]$, by \ref{item:lowbound2v2} and \eqref{eq:bounddegree1}, at most $2(i-1)\cdot (1+\delta)n/k$ vertices have degree at most $(1+\delta)2in/k+3$ in $H$.
        This readily implies that the desired condition holds.\COMMENT{\aedc{Same as the previous, this footnote is now outdated.}It suffices to verify that
        \[(i/\lceil ks/4\rceil+5\delta)x/2>(1+\delta)2in/k+10.\]
        For this, using the lower bound on $x$, it suffices to have that
        \[(i/\lceil ks/4\rceil+5\delta)(1-\delta)s>(1+\delta+o(1))4i/k.\]
        So it suffices to verify that
        \[(1-\delta)\left(\frac{i}{\lceil ks/4\rceil}+5\delta\right)s>(1+\delta+o(1))\frac{i}{ks/4}s,\]
        which clearly holds (since $i/\lceil ks/4\rceil\in[0,1]$).}
    \end{claimproof}

    Arguing analogously to the proof of \eqref{equa:lowbound2}, we also have the following claim.

    \begin{claim}\label{claim:lowbound2}
        For any pair of intervals $I,I'\in\mathcal{Q}$ such that $II'\in E(\Gamma)$, there are vertices $v,w\in V(G)\cap I$ and $v',w'\in V(G)\cap I'$ such that $vv',ww'\in E(H)$.
    \end{claim}

    We are now ready to construct a Hamilton cycle along the lines of the proof of \cref{thm:lowerboundgeneral}.
    Label the intervals in $\mathcal{Q}$ as $I_1,\ldots,I_{1/s}$ following the order given by $\Gamma$.
    For each $i\in[1/s-1]$, by applying \cref{claim:lowbound2}, let $v_i,w_i\in V(G)\cap I_i$ and $v_{i+1}',w_{i+1}'\in V(G)\cap I_{i+1}$ be such that $v_iv_{i+1}',w_iw_{i+1}'\in E(H)$.
    Choose an arbitrary edge in $H[I_1]-\{v_1,w_1\}$ and let $v_1'$ and $w_1'$ be its endpoints; similarly, choose an arbitrary edge in $G[I_{1/s}]-\{v'_{1/s},w'_{1/s}\}$ and let $v_{1/s}$ and $w_{1/s}$ be its endpoints.
    Now, for each $i\in[1/s]$, by \cref{lemma:2pathsv2} (which we can apply by \cref{claim:lowbound1}), we can find two vertex disjoint $(\{v_i',w_i'\},\{v_i,w_i\})$-paths which together cover all vertices of $V(G)\cap I_i$.
    We label these paths inductively as follows.
    
    Let $x_1'\coloneqq v_1'$.
    For each $i\in[1/s-1]$, sequentially, we let $P_i\subseteq H(I_i)$ denote the $(\{v_i',w_i'\},\{v_i,w_i\})$-path which has $x_i'$ as an endpoint, let $P_i'$ denote the other path and $x_i$ denote the other endpoint of $P_i$, and set 
    \[x_{i+1}'\coloneqq\begin{cases}
        v_{i+1}'&\text{ if }x_i=v_i,\\
        w_{i+1}'&\text{ if }x_i=w_i.
    \end{cases}\]
    Lastly, for $i=1/s$, we let $P_i\subseteq H(I_i)$ denote the path given by \cref{lemma:2paths} having $x_i'$ as an endpoint, let $P_i'$ denote the other path, and let $x_i$ denote the other endpoint of $P_i$.
    Now, for each $i\in[1/s]$, let $y_i\in\{v_i,w_i\}$ and $y_i'\in\{v_i',w_i'\}$ be such that $\{x_i,y_i\}=\{v_i,w_i\}$ and $\{x_i',y_i'\}=\{v_i',w_i'\}$.
    
    With this notation, it follows that 
    \[y_1'\bigg(\bigtimes_{i=1}^{1/s} x_i'P_ix_i\bigg)\bigg(\bigtimes_{i=0}^{1/s-1} y_{1/s-i}P_{1/s-i}'y_{1/s-i}'\bigg)\] 
    is a Hamilton cycle.
\end{proof}

\begin{remark}\label{remark:lowerboundPosa}
    The same proof ideas can be adapted for higher dimensions. To do this, one must choose the region of each cell where we want to impose that all vertices have degree at least half of the number of vertices in the cell, and verify that the remaining vertices have degrees which satisfy the condition of \cref{lemma:2pathsv2}.
    If we impose that, for every cell $q$, all the vertices contained in a subhypercube centred at the centre of $q$ of volume half the volume of $q$ have degree at least half of the number of vertices in $q$, it is possible to find the optimal length of the side of a cell and prove that the vertices in $q$ outside this subhypercube satisfy the right degree condition.
    This leads to a proof that a.a.s.\ every $(1-\alpha_d+\eps)$-subgraph of $G_d(n,r)$ is Hamiltonian, where 
    \[\alpha_d\coloneqq\frac{2^d}{(2^{1/d}+1)^dd^{d/2}\theta_d}.\]
    This results in an improvement over \cref{thm:lowerboundgeneral}, as the proportion of edges incident to each vertex which can be removed increases by a factor of $2^{d+1}/(2^{1/d}+1)^d$, which is an increasing function of $d$ tending to $\sqrt{2}$.

    We believe that this approach cannot be sharpened any further for $G_1(n,r)$.
    For $d\geq2$, it is possible to improve upon it by optimising on the ``shapes'' that we use.\COMMENT{Indeed, the locus of all points contained in a cell which are at distance at most $r$ from all other points in the cell is not a hypercube, which is what we have used above, when $d\geq2$.
    Using this locus and optimising over the side-length of the cells could yield further substantial improvements on the lower bound for the local resilience.}
    However, the bounds which can be obtained in this way would still tend to $0$ at least exponentially fast as $d$ grows.\COMMENT{
    Indeed, consider an arbitrary cell~$c$ (which is far from the boundary of $[0,1]^d$) and let~$v$ be a vertex which lies very close to a corner of~$c$.
    Here, the axis-parallel hyperplanes which define this corner of~$c$ split the neighbours of~$v$ into $2^d$ regions, each of them containing roughly a $2^{-d}$ proportion of the neighbours of~$v$.
    In order to apply our proof method, $v$ would need to retain at least some neighbours in $c$, which is contained in one of the orthants.
    But this means the adversary cannot be allowed to delete more than a $2^{-d}$ proportion of the edges incident to $v$.
    As this is very far from \cref{mainconjecture} and the analysis becomes more complicated, we have chosen to only present the idea that we use in its simplest form, which is also the case when we come closest to \cref{mainconjecture}.
    In order to obtain further improvements on this, we need techniques which allow us to deal with the case when the graph induced on the cells is not Hamiltonian.}
\end{remark}

\begin{remark}
The proofs of \cref{thm:lowerboundgeneral,thm:lowerboundd=1} can be adapted to each different norm, resulting in different values for $\alpha_d$.
\end{remark}

%%%%%%%%%%%%%%%%%%%%%%%%%%%%%%%%%%%%%%
%%%%%%%%%%%%%%%%%%%%%%%%%%%%%%%%%%%%%%

\section{The one-dimensional torus and powers of cycles}\label{sec:torus}

This last section is dedicated to the proofs of \cref{lemma:sandwich} and \cref{prop:square}.

\begin{proof}[Proof of \cref{lemma:sandwich}]
    When considering $T_1(n,r)$, the embedding of the vertices into the unit circle naturally defines a cyclic order on them (recall that a.s.\ no two vertices are assigned to the same position), with each vertex being joined by edges to a number of vertices on each side.
    Thus, it suffices to show that the number of neighbours on each side is sufficiently concentrated.
    This follows from standard concentration inequalities.

    More precisely, let $0<1/C\lll\eps<1$.
    Fix a vertex $i\in[n]$ and reveal its position $x_i\in \mathbb{R}/\mathbb{Z}$.
    Conditionally on $x_i$, let $X_L$ and $X_R$ denote the number of vertices in $[n]\setminus\{i\}$ falling into $[x_i-r, x_i)+\mathbb Z$ and $(x_i, x_i+r]+\mathbb Z$, respectively.
    We have that $X_L,X_R\sim\mathrm{Bin}(n-1,r)$, and so $\mathbb{E}[X_L]=\mathbb{E}[X_R]\sim rn$.
    A standard application of Chernoff's bound (\cref{lem:chernoff}) shows that 
    \[\mathbb{P}[|X_L-\mathbb{E}[X_L]|\geq\eps rn]=\mathbb{P}[|X_R-\mathbb{E}[X_R]|\geq\eps rn]=o(1/n),\]
    and a union bound over all vertices of $T_1(n,r)$ completes the proof.
\end{proof}

\begin{proof}[Proof of \cref{prop:square}]
    The only $4$-vertex graph $H$ with $\delta(H)\geq3$ is $K_4$, which is Hamiltonian, so we may assume that $n\geq5$.
    We also assume that $V(C_n)=\mathbb Z/n\mathbb{Z}$ with $ij\in E(C_n)$ if and only if $i-j\equiv\pm1 \pmod{n}$ (so $ij\in E(C_n^2)$ if and only if~$i-j\equiv\pm1 \pmod{n}$ or $i-j\equiv\pm2 \pmod{n}$).
    If~$E(C_n)\subseteq E(H)$, then $H$ is Hamiltonian, so we may assume this is not the case.
    
    Next, we claim that, if $E(C_n^2)\setminus E(H)\subseteq E(C_n)$, then $H$ is Hamiltonian.
    Indeed, if $n$ is odd, then the edges $\{ij:i-j\equiv \pm 2\pmod{n}\}\subseteq E(H)$ form a Hamilton cycle, so we may assume $n$ is even.
    In this case, the edges $\{ij:i-j\equiv\pm2\pmod{n}\}\subseteq E(H)$ form two cycles of length $n/2$ each.
    By relabelling the vertices if necessary, we may assume that $\{0,1\}\notin E(H)$, which implies that $\{n-1,0\},\{1,2\}\in E(H)$.
    Then, one obtains a Hamilton cycle by starting at $1$, following the long cycle containing the odd integers up to $n-1$, then using the edge $\{n-1,0\}$, following the cycle containing the even integers backwards until $2$, and finally using the edge $\{1,2\}$.
    Therefore, we may assume that $E(C_n^2)\setminus E(H)\nsubseteq E(C_n)$.
    
    Now, by relabelling the vertices if necessary, we may assume that $\{0,1\}\notin E(H)$ and $\{2,3\}\in E(H)$.
    Note that the first condition implies that $\{n-2,0\},\{n-1,0\},\{0,2\},\{n-1,1\},\{1,2\},\{1,3\}\in E(H)$.
    We are first going to construct a path containing most of the vertices of $H$, and then show that it can be closed into a Hamilton cycle.

    Define $P_3\coloneqq(1,2,3)$ and let $i$ be a counter initiated as $i=3$.
    Now, while $i<n-2$, suppose that we have defined a $(1,i)$-path $P_i$ such that $V(P_i)=[i]$.
    If $\{i,i+1\}\in E(H)$, then we define $P_{i+1}\coloneqq (P_i,i+1)$ and update the counter by adding $1$ to it.
    Otherwise, as $\delta(H)\geq3$, it must be the case that $\{i,i+2\},\{i+1,i+2\},\{i+1,i+3\}\in E(H)$.
    In this case, we define $P_{i+3}\coloneqq(P_i,i+2,i+1,i+3)$ and update the counter by adding $3$ to it.
    When this process ends, $i$ can take one of three values.
    If $i=n-2$, then $(P_{n-2},0,n-1,1)$ forms a Hamilton cycle.
    If $i=n-1$, let $P'\coloneqq P_i-\{1,2\}$; then, $(P',0,2,1,3)$ forms a Hamilton cycle.
    Lastly, if $i=n$, then we again let $P'\coloneqq P_i-\{1,2\}$ and observe that $(P',2,1,3)$ is a Hamilton cycle.
\end{proof}

\section*{Acknowledgements}

Parts of the research leading to these results were conducted during visits of (subsets of) the authors to Universit\'e Lyon 1, Technische Universit\"at Ilmenau and IST Austria.
We are grateful to these institutions for their hospitality.
L.~Lichev would like to thank Matthew Kwan for turning our attention to the work of~\citet{Sti96}.

% Use with natbib, not biblatex:
\bibliographystyle{mystyle} 
\bibliography{HamResRGG}

% Use with biblatex, not natbib:
% \printbibliography

\appendix

\section{A simple proof for the local resilience with respect to connectivity}\label{appen:connectivity}

In this section, we provide a simpler proof of \cref{thm:connectivityintro} compared to the one offered in \cref{sect:conn} (in exchange for worse control on the growth of the constants with respect to the dimension).
In fact, we prove the following stronger result, which readily implies \cref{thm:connectivityintro}.

\begin{theorem}\label{thm:connectivity_appen}
For every constant\/ $\eps\in (0,1/2]$ and integer\/ $d\geq1$, there exist constants\/ $c_1 = c_1(d, \eps) > 0$ and\/ $C_1=C_1(d, \eps)>0$ such that, for every\/ $r\in[C_1(\log n/n)^{1/d},\sqrt{d}]$, a.a.s.\ every\/ $(1/2+\eps)$-subgraph of\/ $G_d(n,r)$ is\/ $(c_1r^dn)$-connected.
\end{theorem}

Note that the only difference between \cref{thm:connectivity_appen} and \cref{thm:connectivity} is the growth of the constant $C_1$ as a function of $d$.
The proof of \cref{thm:connectivity_appen} relies mostly on the fact that, for a suitable choice of the parameters and a sufficiently small $\delta = \delta(d, \eps) > 0$, all pairs of vertices at Euclidean distance at most $\delta r$ from each other have many common neighbours in any $(1/2+\eps)$-subgraph~$H$ of~$G_d(n,r)$. 
Thus, the square of $H$ contains $G_d(n,\delta r)$, which is a.a.s.\ connected for all sufficiently large~$r$.
In fact, this same approach yields the higher connectivity.

To be more precise, in order to prove \cref{thm:connectivity_appen}, we first prove the following auxiliary lemma.

\begin{lemma}\label{lem:close points_appen}
For every fixed integer\/ $d\ge 1$ and\/ $\eps\in (0, 1/2]$, there exist positive constants\/ $\delta = \delta(d, \eps)$,\/ $C_1' = C_1'(d, \eps)$ and\/ $c_1' = c_1'(d, \eps)$ such that, for any\/ $r\in [C_1' (\log n/n)^{1/d},\sqrt{d}]$, a.a.s.\ the following holds:
every pair of vertices of\/ $G_d(n,r)$ at Euclidean distance at most\/ $\delta r$ have at least\/ $c_1'r^dn$ common neighbours in each\/ $(1/2+\eps)$-subgraph of\/ $G_d(n,r)$.
\end{lemma}

\begin{proof}
    Let $G\sim G_d(n,r)$ and condition on the event that $V(G)$ satisfies the properties described in \cref{lem:basic properties}~\ref{lem:basic propertiesitem2.1} and~\ref{lem:basic propertiesitem3}, which holds a.a.s.
    By \cref{lem:basic properties}~\ref{lem:basic propertiesitem2.1}, every vertex $v\in V(G)$ has degree at most $(1+\eps/50)|B(v,r)\cap[0,1]^d|n$.
    By \cref{lem:basic properties}~\ref{lem:basic propertiesitem3}, every pair of vertices $u,v\in V(G)$ such that $\lVert u-v\rVert\leq\delta r$ have at least $(1-\eps/20)|B(v,r)\cap[0,1]^d|n$ common neighbours.
    Therefore, if $H\subseteq G$ is a $(1/2+\eps)$-subgraph of $G$, for any pair of vertices $u,v\in V(G)$ such that $\lVert u-v\rVert\leq\delta r$, we have that\COMMENT{Note that $|N_{G}(u)\cap N_{G}(v)|\geq(1-\eps/20)|B(v,r)\cap[0,1]^d|n$ and $|N_{G}(u)\cap N_{G}(v)|\geq(1-\eps/20)|B(u,r)\cap[0,1]^d|n$, so $|N_{G}(u)\cap N_{G}(v)|\geq(1-\eps/20)(|B(u,r)\cap[0,1]^d|+|B(v,r)\cap[0,1]^d|)n/2$.
    Thus, we have
    \begin{align*}
        |N_{H}(u)\cap N_{H}(v)|&\geq |N_{G}(u)\cap N_{G}(v)|-(1/2-\eps)(d_G(u)+d_G(v))\\
        &\geq(1-\eps/20)(|B(u,r)\cap[0,1]^d|+|B(v,r)\cap[0,1]^d|)n/2 - (1/2-\eps)(1+\eps/50)(|B(u,r)\cap[0,1]^d|+|B(v,r)\cap[0,1]^d|)n\\
        &\geq\eps(|B(u,r)\cap[0,1]^d|+|B(v,r)\cap[0,1]^d|)n/2\\
        &\geq\eps2^{-d}\theta_dr^dn,
    \end{align*}
    where in the third line we use that $(1-\eps/20)-2(1/2-\eps)(1+\eps/50)=\eps(2-1/20-1/25+\eps/25)\geq\eps$.}
    \[|N_{H}(u)\cap N_{H}(v)|\geq |N_{G}(u)\cap N_{G}(v)|-(1/2-\eps)(d_G(u)+d_G(v))\geq\eps2^{-d}\theta_dr^dn.\qedhere\]
\end{proof}

We will also use the fact that sufficiently dense random geometric graphs have high connectivity, which is a particular case of a result of \citet[Theorem~13.2]{Pen03}.
Here, we present a simple proof for the sake of completeness.

\begin{lemma}\label{lem:highconnectivity_appen}
For each $d\geq1$, there exist constants $C_1''=C_1''(d)>0$ and $c_1''=c_1''(d)>0$ such that, if $r\in [C_1''(\log n/n)^{1/d},\sqrt{d}]$, then a.a.s.\ $G_d(n,r)$ is $c_1''r^dn$-connected.
\end{lemma}

\begin{proof}
Consider a tessellation $\cQ$ of $[0,1]^d$ with cells of side length $s\coloneqq\lceil\sqrt{d+3}/r\rceil^{-1}$, and recall the auxiliary graph $\Gamma$ where two cells are joined by an edge whenever they share a $(d-1)$-dimensional face.
Observe that, for any cells $q$ and $q'$ with $qq'\in E(\Gamma)$ and any pair of points $p,p'\in q\cup q'$, the choice of $s$ guarantees that $\lVert p-p'\rVert\leq r$.
   In particular, all vertices of the graph $G\sim G_d(n,r)$ contained in two such cells form a complete graph.
   Moreover, note that $r/2\sqrt{d+3}\leq s\leq r/\sqrt{d+3}$\COMMENT{More precisely, using the fact that $\sqrt{d+3}/r\leq\lceil\sqrt{d+3}/r\rceil\leq\sqrt{d+3}/r+1$, the fact that $s\leq r/\sqrt{d+3}$ follows directly.
   For the lower bound, we get that
   \[s\geq\frac{r}{\sqrt{d+3}+r},\]
   so $s/r$ is a decreasing function of $r$. 
   Therefore, it is minimised when $r=\sqrt{d}$, and we have that $s\geq r/(\sqrt{d+3}+\sqrt{d})$.
   Since we do not need the exact constant, we simply use the weaker bound.} and so, in particular, the volume of each cell is $s^d\geq2^{-d}(d+3)^{-d/2}r^d$.
   By letting, for example, $c_1''=2^{-d-1}(d+3)^{-d/2}$ and $C_1''=28\sqrt{d+3}$, it follows from Chernoff's bound (\cref{lem:chernoff}) and a union bound that a.a.s.\ every $q\in\cQ$ contains at least $c_1''r^dn$ vertices of~$G$.\COMMENT{Given the lower bound on $r$, the number of cells is at most linear, so it suffices to obtain the conclusion for each cell with probability $1-o(1/n)$.
   And the probability of failure for a fixed cell $q$, by \cref{lem:chernoff}, is
   \[\mathbb{P}[|V(G)\cap q|<c_1''r^dn]\leq\mathbb{P}[|V(G)\cap q|\leq s^dn/2]\leq2\nume^{-s^dn/12}\leq2\nume^{c_1''r^dn/6}\leq2\nume^{c_1''(C_1'')^d\log n/6},\]
   So it suffices to have that $c_1''(C_1'')^d/6>1$, for which it suffices to have $C_1''\geq(7/c_1'')^{1/d}=2\cdot14^{1/d}\sqrt{d+3}$.}

   Now, conditionally on this event, consider any two vertices $u,v\in V(G)$ and any vertex set $U\subseteq V(G)\setminus\{u,v\}$ of size $|U|\leq c_1''r^dn-1$, and let $G'\coloneqq G-U$.
   If there exist $q,q'\in\cQ$ such that $qq'\in E(\Gamma)$ and $u,v\in q\cup q'$, then $uv\in E(G')$, so they are joined by a path.
   Thus, we may assume that there are cells $q,q'\in\cQ$ with $u\in q$, $v\in q'$ and $qq'\notin E(\Gamma)$.
   Since $\Gamma$ is connected, there is a $(q,q')$-path $P\subseteq\Gamma$.
   Let us label the cells of $P$ as $q=q_0,q_1,\ldots,q_\ell=q'$ following the order in which they appear in $P$, and note that every cell contains at least one vertex of $G'$.
   Let $u_0\coloneqq u$.
   Now we can greedily construct a $(u,v)$-path in $G'$ by arbitrarily choosing a vertex $u_i\in V(G)\cap q_i$ for each $i\in[\ell-1]$, and noting that $u_{i-1}u_i\in E(G')$ and $u_{\ell-1}v\in E(G')$.
\end{proof}

\begin{proof}[Proof of Theorem~\ref{thm:connectivity_appen}]
Let $\delta$, $C_1'$ and $c_1'$ be the constants given by \cref{lem:close points_appen} applied with $d$ and $\eps$, and let $C_1''$ and $c_1''$ be the constants given by \cref{lem:highconnectivity_appen} with $d$ as input.
Fix $C_1\coloneqq\max\{C_1',C_1''/\delta\}$ and $c_1\coloneqq\min\{c_1',c_1''\}$.
Couple $G_d(n,\delta r)$ and $G_d(n,r)$ in a way that their vertices have the same positions in $[0,1]^d$ and fix a $(1/2+\eps)$-subgraph $H$ of $G_d(n,r)$.
Moreover, condition on the events that $G_d(n,\delta r)$ is $c_1r^dn$-connected and that, for every edge $uv\in E(G_d(n,\delta r))$, the vertices $u$ and $v$ have at least $c_1r^dn$ common neighbours in $H$ (note that both events hold a.a.s.\ by \cref{lem:close points_appen,lem:highconnectivity_appen} with our choice of $C_1$ and $c_1$).

Fix two arbitrary vertices $u,v\in V(G)$ and let $U\subseteq V(G)\setminus\{u,v\}$ be an arbitrary set of size $|U|\leq c_1r^dn-1$.
By the conditioning, $G_d(n,\delta r)-U$ is connected, so there exists a $(u,v)$-path $P$ in $G_d(n,\delta r)-U$, and $G_d(n,\delta r)\subseteq H$, so $H-U$ also contains a $(u,v)$-path.
This implies that $H$ is $c_1r^dn$-connected, as desired.
\end{proof}

\end{document}